\documentclass[a4paper]{amsart}

\usepackage[a4paper,width={16cm},left=2.5cm,bottom=3cm, top=3cm]{geometry}
\usepackage{enumerate}

\usepackage[english]{babel}

\usepackage{a4wide}
\usepackage[dvipsnames]{xcolor}
\usepackage{amsmath}
\usepackage{amssymb}
\usepackage{amsfonts}
\usepackage{amsthm,graphicx}
\usepackage{comment}
\usepackage{csquotes}

\usepackage{hyperref}
\hypersetup{
	linktocpage=true,
	colorlinks=true,
	linkcolor=blue!70!black,
	citecolor=orange!40!red,
	urlcolor=magenta!80!black,
}

\usepackage{mathtools}
\mathtoolsset{showonlyrefs}

\numberwithin{equation}{section}
\newtheorem{theorem}{Theorem}[section]
\newtheorem{lemma}[theorem]{Lemma}

\newtheorem{proposition}[theorem]{Proposition}
\theoremstyle{definition}

\newtheorem{remark}[theorem]{Remark}

\DeclareRobustCommand{\rchi}{{\mathpalette\irchi\relax}}
\renewcommand{\Bbb}{\mathbb}
\newcommand{\irchi}[2]{\raisebox{\depth}{$#1\chi$}} 
\newcommand{\mc}{\mathcal}
\newcommand{\mb}{\mathbb}

\newcommand{\la}{\lambda}
\newcommand{\norm}[1]{\left\lVert#1\right\rVert}

\newcommand{\R}{\mb{R}}
\newcommand{\C}{\mb{C}}
\newcommand{\N}{\mb{N}}
\newcommand{\Z}{\mb{Z}}

\newcommand{\e}{\varepsilon}

\newcommand{\dive}{\mathop{\rm{div}}}

\usepackage{subfig}
\usepackage{tikz}
\usepackage{xcolor}
\usepackage{mathrsfs}
\usetikzlibrary{arrows}
\usetikzlibrary{shadings}
\makeatletter

\begin{document}

\title[On Aharonov-Bohm operators
with multiple colliding poles of any circulation]
{On Aharonov-Bohm operators
with multiple colliding poles of any circulation}

\author[V. Felli, B. Noris,  and G. Siclari]{Veronica Felli, Benedetta Noris, and Giovanni Siclari}

\address{Veronica Felli
  \newline \indent Dipartimento di Matematica e Applicazioni
  \newline \indent
Universit\`a degli Studi di Milano–Bicocca
\newline\indent Via Cozzi 55, 20125 Milano, Italy}
\email{veronica.felli@unimib.it}

\address{Benedetta Noris  and Giovanni Siclari 
  \newline \indent Dipartimento di Matematica
  \newline \indent Politecnico di Milano
  \newline\indent Piazza Leonardo da Vinci 32, 20133 Milano, Italy}
\email{benedetta.noris@polimi.it,  giovanni.siclari@polimi.it}

\date{June 14, 2024.}

\begin{abstract}
  This paper deals with quantitative spectral stability for
  Aharonov-Bohm operators with many colliding poles of whichever
  circulation.  An equivalent formulation of the eigenvalue problem is
  derived as a system of two equations with real coefficients, coupled
  through prescribed jumps of the unknowns and their normal
  derivatives across the segments joining the poles with the collision
  point.  Under the assumption that the sum of all circulations is not
  integer, the dominant term in the asymptotic expansion for
  eigenvalues is characterized in terms of the minimum of an energy
  functional associated with the configuration of poles. Estimates of
  the order of vanishing of the eigenvalue variation are then deduced
  from a blow-up analysis, yielding sharp asymptotics in some
  particular examples.
\end{abstract}

\maketitle

\hskip10pt{\footnotesize {\bf Keywords.} Aharonov–Bohm operators, spectral theory, asymptotics of eigenvalues, blow-up analysis.

\medskip 

}

\section{Introduction}
Continuing the study initiated in \cite{FNOS_multipole}, we deal with
the problem of  spectral stability for Aharonov-Bohm operators with many coalescing poles.
 The main novelty of the present paper lies in the generality of the assumptions imposed on the circulations of poles.
Indeed, we do not restrict our attention to the case of half-integer circulations, as done in \cite{FNOS_multipole}.

Specifically, we study how the eigenvalues of Aharonov-Bohm operators respond to variations in the position of the poles. 
For every $b=(b_1,b_2) \in \R^2$, the Aharonov-Bohm vector potential with pole $b$ and circulation $\rho\in\R$ is defined as 
\begin{equation}\label{def_AB_vector}
A_b^\rho(x_1,x_2):=\rho\left(\frac{-(x_2-b_2)}{(x_1-b_1)^2+(x_2-b_2)^2},\frac{x_1-b_1}{(x_1-b_1)^2+(x_2-b_2)^2}\right), \quad (x_1,x_2)\in\R^2\setminus	\{b\}.
\end{equation}
The associated  Aharonov-Bohm magnetic field  arises when  an
infinitely long thin solenoid intersects perpendicularly the plane
$(x_1,x_2)$ at the point $b$, with  the radius of the solenoid going to
zero and the magnetic flux remaining constantly equal to $\rho$, see \cite{AT1998,AB1959}.
The most  studied case in the  literature concernes  half-integer circulations  $\rho\in \frac12+\Z$, whose  mathematical relevance is related to  applications  to the problem of  spectral minimal partitions,  as highlighted in \cite{BNHHO2009,HHT2009,NT2010}. 

For Schr\"odinger operators of the form $(i\nabla+A_{b}^\rho)^2$, with Aharonov-Bohm magnetic potentials as in \eqref{def_AB_vector}, the continuity of eigenvalues with respect to the  pole's location  is established in \cite{BNNNT} for a single pole, and in \cite{Le2015} for multiple (possibly colliding) poles. Starting from this, several papers have delved into determining the precise asymptotic behavior of the eigenvalue variation, when the configuration of the poles undergoes small perturbations.
In the case of a single moving pole with half-integer circulation, \cite{BNNNT}  identifies a relation between the convergence rate of eigenvalues and the number of nodal lines of the corresponding eigenfunction. More refined asymptotic expansions for simple eigenvalues are presented in \cite{AF2015}, where a pole moving along the tangent to a nodal line of the limit eigenfunction is considered, and in \cite{AF2016}, for a pole moving in any direction. The scenario of a pole approaching the boundary is addressed in \cite{AFNN2017} and \cite{NNT}, while some genericity issues are discussed in \cite{A2019} and \cite{AN}.

Dealing with multiple colliding poles poses additional significant difficulties. References  \cite{AFHL} and \cite{AFL2020} consider two coalescing poles (both with half-integer circulation), moving along an axis of symmetry of the domain, in the cases where this axis is, or respectively is not, tangent to a nodal line of the limit eigenfunction. The symmetry assumption  is dropped in \cite{AFL2017}, in the case of   two poles 
colliding at an interior point outside the nodal set  of the limit eigenfunction. The recent paper  \cite{FNOS_multipole} provides an asymptotic expansion of the eigenvalue variation for Aharonov-Bohm operators with many coalescing poles; the leading term in this expansion is related to the minimum of an energy functional, associated to the pole configuration and  defined on a space of functions jumping along cracks, aligned with the moving directions of the poles.

In the aforementioned papers, asymptotic expansions for the  eigenvalue variation are discussed only in the case of  half-integer circulations. So far, the case of  non-half-integer circulation appears to have been exclusively addressed in  \cite{AFNN2018}, where estimates (rather than the exact asymptotic behavior) are obtained for a single moving pole.

We consider a bounded  connected domain $\Omega \subset \R^2$ and $k$ poles  moving in $\Omega$  towards  a fixed  point $P \in \Omega$ along straight lines. It is not restrictive to fix $P=0$, so 
that   the moving poles can be rewritten as multiples of $k$ fixed points  $\{a^j\}_{j=1,\dots,k}$ with the same infinitesimal parameter $\e$.
Furthermore, since we are interested in asymptotic expansions of eigenvalues as $\e \to 0^+$, we may assume that, for some $R \in (0,1)$,
\begin{equation*}
	\{a^j\}_{j=1,\dots,k} \subset D_R\subset \Omega,
\end{equation*}
where, for every $r>0$, $D_r:=\{y \in\R^2: |y|<r\}$.
We consider  configurations in which each pole  is the only one on the straight line connecting it to the collision point, i.e. the origin, so that,
  for every $j=1,\dots, k$, there exist $r_j>0$ and $\alpha^j\in
(-\pi,\pi]$ such that $\alpha^{j}\neq\alpha^\ell$, $\alpha^{j}\neq\alpha^\ell\pm \pi $   if $j\neq\ell$ and  
\begin{equation}\label{def_aj}
a^j= r_j(\cos(\alpha^j),\sin(\alpha^j)).
\end{equation}
For every $j=1,\dots,k$ and $\e \in (0,1]$, let
\begin{equation*}
a_\e^j:=\e a^j.
\end{equation*}
For every $(\rho^1,\dots,\rho^k)\in\R^k$ and $\e\in(0,1]$, we are interested in the multi-singular vector  potential
\begin{equation}\label{def_AB_vector_multi}
\quad \mathcal	A_\e^{(\rho^1,\dots,\rho^k)}:=\sum_{j=1}^{k}A_{a^j_\e}^{\rho^j}
\end{equation}
and the corresponding eigenvalue problem
with Dirichlet boundary conditions
\begin{equation}\label{prob_Aharonov-Bohm_multipole}
\begin{cases}
\left(i\nabla +\mathcal A_\e^{(\rho^1,\dots,\rho^k)}\right)^2 u= \la \,u, &\text{ in } \Omega,\\
u=0, &\text{ on } \partial \Omega,
\end{cases}
\end{equation}  
where the magnetic Schr\"odinger operator $(i\nabla +\mathcal A_\e^{(\rho^1,\dots,\rho^k)})^2$ 
acts as 
\begin{equation}\label{eq:sch-op}
\left(i\nabla +\mathcal A_\e^{(\rho^1,\dots,\rho^k)}\right)^2u:=-\Delta u +2 i \,\mathcal A_\e^{(\rho^1,\dots,\rho^k)}\cdot \nabla u+\big|\mathcal A_\e^{(\rho^1,\dots,\rho^k)}\big|^2 u.
\end{equation}
It is not restrictive to suppose that 
\begin{equation}\label{eq:all-non-int}
\rho^j \not \in \Z\quad\text{for all $ j =1,\dots,k$},
\end{equation}
since 
$\mathcal A_\e^{(\rho^1,\dots,\rho^k)}$ is gauge equivalent to the vector potential
$\mathcal A_\e^{(\rho^1+n_1,\dots,\rho^k+n_k)}$ for any $n_1\dots,n_k \in \Z$.
It follows that the corresponding Schr\"odinger operators
are unitarily equivalent (see \cite[Theorem 1.2]{L_gauge_invariance} and \cite[Proposition 2.2]{Le2015}), and therefore they  have the same spectrum.
In the following we assume that 
\begin{equation}\label{def_rho}
\rho:=\sum_{j=1}^k \rho^j \not \in \Z.
\end{equation} 
By the gauge equivalence mentioned above, it is not restrictive to suppose that 
\begin{equation}\label{hp_cirulation}
\rho \in  (0,1).
\end{equation}
By classical spectral theory, problem \eqref{prob_Aharonov-Bohm_multipole} admits a  diverging sequence of positive real  eigenvalues $\{\la_{\e,n}\}_{n\geq1}$ with finite multiplicity. In the sequence $\{\la_{\e,n}\}_{n \geq1}$ we repeat each eigenvalue according to its multiplicity.

As emerges from \cite[Theorem 1.2]{Le2015}, the following limit problem arises as  $\e\to0^+$:
\begin{equation}\label{prob_Aharonov-Bohm_0}
\begin{cases}
\left(i\nabla +A_0^{\rho} \right)^2u=\la u, &\text{ in } \Omega,\\
u=0, &\text{ on } \partial \Omega.
\end{cases}
\end{equation} 
where $A_0^{\rho}$ is as  in \eqref{def_AB_vector} with pole at $0$.
Also for \eqref{prob_Aharonov-Bohm_0},
the classical Spectral Theorem  provides a diverging sequence of real positive eigenvalues $\{\la_{0,n}\}_{n\in \N\setminus\{0\}}$ with finite multiplicity, which are 
repeated in the enumeration  according to their multiplicity.

Furthermore,   by \cite[Theorem 1.2]{Le2015},
\begin{equation*}
	\text{the function}\quad\e \mapsto \la_{\e,n} \quad \text{is continuous in } [0,1], 
\end{equation*}
so that, in particular, 
\begin{equation}\label{limit_lae_la0}
	\lim_{\e\to0^+}\la_{\e,n} = \la_{0,n}
\end{equation}
for every $n \in \mathbb{N}\setminus\{0\}$. 
The present paper aims at estimating the vanishing order of the variation $\la_{\e,n} -\la_{0,n}$ of simple eigenvalues  with respect to the moving configuration of poles.

The case $\rho_j\in\frac12+\Z$  discussed in \cite{FNOS_multipole} exhibits peculiarities that enable a perspective and an approach not entirely applicable in the non-half-integer case. Specifically, if $\rho_j\in\frac12+\Z$, a gauge transformation allows the problem to be reduced to an eigenvalue problem for the Laplacian under prescribed jumping conditions on cracks. In this situation, the eigenfunctions can be assumed to be real-valued and possess an odd number of nodal lines branching off from  each pole. However, if $\rho_j\in\R\setminus(\Z/2)$ for some $j$, the eigenfunctions remain complex even after the gauge transformation; moreover, the poles with non-half-integer circulation are isolated zeroes, see \cite[Section 7]{FFT}. 

In the present paper, the first step consists in a new equivalent formulation in terms of a real eigenvalue problem. Specifically, we observe that  problem \eqref{prob_Aharonov-Bohm_multipole} is equivalent to a system of two equations with real coefficients, where the  unknowns are the real and imaginary parts of the gauged eigenfunction, see \eqref{prob_eigenvalue_gauged}. Such equations are coupled through prescribed jumps of the unknowns and their normal derivatives across cracks, directed as   the  segments joining the poles with the collision point.  
Moreover,  each eigenvalue's multiplicity doubles when passing to the new formulation \eqref{prob_eigenvalue_gauged}.

 Given such an equivalent formulation, in Theorem \ref{theorem_asymptotic_not_precise} we derive an asymptotic
expansion of the variation of simple eigenvalues of the form
\begin{equation*}
\la_{\e,n}-\la_{0,n}=2(\mc{E}_\e+L_\e(v_0,w_0))+o(\|\nabla V_\e\|_{L^2}^2+\|\nabla W_\e\|_{L^2})\quad \text{as $\e \to 0^+$},
\end{equation*}
where $\mathcal{E}_\e$ is the minimum,  attained by the couple of functions $(V_\e,W_\e)$, of an energy functional associated
with the configuration of poles, see \eqref{def_Ee}, and $L_\e$ is a suitable linear
functional involving the limit eigenfunction $(v_0,w_0)$ of \eqref{prob_eigenvalue_gauged} on the cracks.

A blow-up analysis allows us to identify the asymptotic behavior of $\mc{E}_\e$ and $(V_\e,W_\e)$ as $\e\to0^+$, thus estimating the vanishing order of the eigenvalue variation in  Theorem \ref{theorem_asymptotic_precise} and obtaining sharp estimates on the behavior of the eigenfunctions in Theorem \ref{theorem_blow_up_eigenfunctions}.
We observe that detecting the 
exact vanishing order of $2(\mc{E}_\e+L_\e(v_0,w_0))$, and consequently of $\la_{\e,n}-\la_{0,n}$, is much more delicate in the non-half-integer case compared
to the half-integer one, as was already the case for one pole \cite{AFNN2018}. 
In specific scenarios, the blow-up result provided in Theorem \ref{theorem_asymptotic_precise} is sufficient to yield the exact asymptotics. For instance, identifying the precise convergence rate is feasible when the sum of (non-half-integer) circulations of all poles is half-integer, as observed in Proposition \ref{prop_positive_negative}.

The paper is organized as follows. In the next section, we state and discuss the main results. In Section \ref{sec_preliminaries}, we introduce the gauge transformation which allows us to obtain the  equivalent formulation \eqref{prob_eigenvalue_gauged} for the eigenvalue problem; furthermore, in Subsection \ref{subsec:asy-limit-eige} we describe the asymptotic behaviour of eigenfunctions of the limit problem, depending on whether the sum $\rho$ of the circulations of all colliding poles  is  half-integer or not. In Section \ref{sec_propoerties_Ee} we derive some  preliminary estimates on the quantity $\mathcal E_\e$. Section \ref{sec_asymptotic_eigenvalue_variation}
is devoted to the proof of Theorem \ref{theorem_asymptotic_not_precise}, while in Section \ref{sec:blowup} we perform a blow-up analysis, providing precise information about the asymptotic behaviour of $\mathcal E_\e$ and $(V_\e,W_\e)$ as $\e\to0^+$, as stated in Theorem \ref{theorem_asymptotic_precise}, see also Proposition \ref{prop_blow_up}; the characterization of the concrete functional space containing the blow-up limit profile is possible thanks to the Hardy type inequality obtained in Subsection \ref{subsection_hardy_ineq}.

\section{Statement of the main results}\label{sec_main_results}
For every $\e\in (0,1]$, a variational formulation of problem \eqref{prob_Aharonov-Bohm_multipole}  can be given in the functional space
\begin{equation*}
    H^{1,\e}(\Omega,\C)=\left\{\varphi\in
H^1(\Omega,\C):\frac{\varphi}{|\cdot-a^j_\e|}\in L^2(\Omega,\C)\text{ for
	all }j=1,\dots,k\right\},
\end{equation*}
which can be equivalently defined as the completion of
\begin{equation*}
\{\varphi \in H^1(\Omega,\C) \cap C^\infty(\Omega,\C): \varphi\equiv 0\text{ in  a neighbourhood of } a^j_\e \text{ for all } j=1,\dots,k\}
\end{equation*}
with respect to the norm 
\begin{equation}\label{def_norm_H1a}
\norm{w}_{H^{1,\e}(\Omega,\C)}:=\bigg(\norm{\varphi}_{L^2(\Omega,\C)}^2+\norm{\nabla \varphi}_{L^2(\Omega,\C^2)}^2+\sum_{j=1}^{k}\Big\|\tfrac{\varphi}{|\cdot-a^j_\e|}\Big\|_{L^2(\Omega,\C)}^2\bigg)^{\!\!1/2}.
\end{equation}
A deep connection between the space $H^{1,\e}(\Omega,\C)$ and the operator \eqref{eq:sch-op} is induced by the following magnetic Hardy-type inequality proved in \cite{LW}, see also \cite{AFJT} and \cite[Lemma 3.1,
Remark 3.2]{FFT}:
\begin{equation*}
\int_{D_r(b)}|i \nabla \varphi + A_b^\rho \varphi|^2 \, dx \ge \Big(\min_{j \in \mb{Z}}|j-\rho|\Big)^2\int_{D_r(b)}\frac{|\varphi(x)|^2}{|x-b|^2}\, dx
\end{equation*}
for every $ b \in \R^2$ and
$\varphi \in C^\infty_{\rm c}(\overline{D_r(b)}\setminus\{b\},\C)$,
where $D_r(b):=\{y \in\R^2: |y-b|<r\}$.
In particular, under assumption \eqref{eq:all-non-int}, the norm
\eqref{def_norm_H1a} is equivalent to the norm
\begin{equation*}
\left(\norm{\left(i\nabla +\mathcal	A_\e^{(\rho_1,\dots,\rho_k)}\right)\varphi}_{L^2(\Omega,\C^2)}^2+\norm{\varphi}_{L^2(\Omega,\C)}^2\right)^{\!1/2}.
\end{equation*}
We also consider the subspace
\begin{equation*}
H_0^{1,\e}(\Omega,\C)=\left\{\varphi \in  H_0^1(\Omega,\C):\tfrac{\varphi}{|\cdot-a^j_\e|}\in L^2(\Omega,\C) \text{ for all } j=1,\dots,k\right\}.
\end{equation*}
We say that $\lambda \in\R$ is an eigenvalue of \eqref{prob_Aharonov-Bohm_multipole} if there exists an eigenfunction $u \in H_0^{1,\e}(\Omega,\C)\setminus\{0\}$
such that
\begin{equation}\label{eq_eigenfunctions_Aharonov_Bohm_multipole}
\int_\Omega \left(i\nabla +\mathcal A_\e^{(\rho_1,\dots,\rho_k)}\right)u\cdot \overline{\left(i\nabla+\mathcal A_\e^{(\rho_1,\dots,\rho_k)}\right)\varphi} \, dx =\la \int_\Omega u \overline{\varphi} \, dx
\quad \text{for all } \varphi \in H_0^{1,\e}(\Omega,\C).
\end{equation}
The limit problem \eqref{prob_Aharonov-Bohm_0} is settled in the functional space
\begin{equation*}
H_0^{1,0}(\Omega,\C)=\left\{\varphi \in  H_0^1(\Omega,\C):\tfrac{\varphi}{|x|}\in L^2(\Omega,\C)\right\};
\end{equation*}
we say that $\lambda \in\R$ is an eigenvalue of \eqref{prob_Aharonov-Bohm_0} if there exists an eigenfunction $u \in H_0^{1,0}(\Omega,\C)\setminus\{0\}$ 
such that
\begin{equation*}
\int_\Omega \left(i\nabla +A_0^\rho\right)u\cdot \overline{\left(i\nabla+\mathcal A_0^\rho \right)\varphi} \, dx =\la \int_\Omega u \overline{\varphi} \, dx
\quad \text{for all } \varphi \in H_0^{1,0}(\Omega,\C).
\end{equation*}

\begin{remark}\label{rem:propertyP}
We observe that, if $\rho=\frac12$ and $\lambda$ is an eigenvalue of \eqref{prob_Aharonov-Bohm_0}, the associated eigenspace admits a basis consisting of $K$-real eigenfunctions, i.e. of eigenfunctions 
invariant under the action of the antilinear operator $Ku(r\cos t,r\sin t)=e^{i(t+2\Lambda)}\overline{u}(r\cos t,r\sin t)$, where 
\begin{equation*}
\Lambda=\frac\pi 2-\sum_{j=1}^k\rho^j\alpha^j,
\end{equation*}
see \cite[Lemma 3.3]{HHOO99},  \cite[Lemma 2.3]{BNNNT}, and  \cite[Remark 3.5]{FNOS_multipole}. We could use here any other real constant $\Lambda$; this specific choice is made just to simplify the writing of Proposition \ref{prop_vw_asympotic_1/2}. We note that $u$ is $K$-real if and only if it satisfies
the property
\begin{equation}\label{eq:propertyP}
 \text{$e^{-i(\frac t2+\Lambda})u(r\cos t,r\sin t)$ is a real-valued function.}
\end{equation}
\end{remark}
By a suitable gauge
transformation, 
\eqref{prob_Aharonov-Bohm_multipole} and \eqref{prob_Aharonov-Bohm_0} can be reformulated as 
eigenvalue problems for the Laplacian in domains with straight cracks.
For every $\e \in [0,1]$ and $j=1,\dots,k$, we define 
\begin{equation*}
\Sigma^j:=\{ta^j: t \in \R\}, \quad \Gamma^j_\e:=\{ta^j: t \in (-\infty,\e]\}, \quad S_\e^j:=\{ta^j: t \in [0,\e]\}.
\end{equation*}
Let 
\begin{equation*}
\Gamma_\e:=\bigcup_{j=1}^{k} \Gamma_\e^j
\end{equation*}
and $\mathcal H_\e$ be  the functional space defined as the closure of
\begin{equation*}
\left\{\varphi \in H^1(\Omega \setminus \Gamma_\e)=H^1(\Omega \setminus \Gamma_\e,\R): \, \varphi=0 \text{ on a		neighbourhood of } \partial \Omega\right\}
\end{equation*}
in $H^1(\Omega \setminus\Gamma_\e)$ with respect to the norm $\norm{\varphi}_{H^1(\Omega \setminus\Gamma_\e)}=\|\nabla \varphi\|_{L^2(\Omega\setminus \Gamma_\e)}+\|\varphi\|_{L^2(\Omega)}$.
We observe that $\mc{H}_{\e}$-functions satisfy the following Poincaré-type inequality: 
\begin{equation}\label{prop:appendix3_2}
  \int_{\Omega} \varphi^2 dx \le C_P \int_{\Omega \setminus \Gamma_\e} |\nabla \varphi|^2 \, dx,\quad\text{for every }\varphi \in \mc{H}_{\e},
\end{equation}
for some constant $C_P>0$ which is independent of $\e$, see \cite[Proposition 3.2]{FNOS_multipole}.

By  \eqref{prop:appendix3_2}, the norm
\begin{equation*}
\norm{\varphi}_{\mc{H}_\e}:= \left(\int_{\Omega \setminus \Gamma_\e} |\nabla \varphi|^2 \, dx\right)^{\!\!1/2}
\end{equation*}
on $\mc{H}_{\e}$ is equivalent to $\norm{\varphi}_{H^1(\Omega \setminus  \Gamma_\e)}$. We denote
the corresponding scalar product as $(\cdot,\cdot)_{\mathcal H_\e}$.

By classical trace  and embedding theorems for fractional Sobolev spaces in dimension $1$, for every $j=1,\dots,k$ and
$p \in [2,+\infty)$  there exist continuous trace operators
\begin{equation}\label{def_traces}
\gamma_+^j:H^1(\pi_+^j\setminus  \Gamma_1 ) \to L^p(\Sigma^j)  \quad
\text{and} \quad
\gamma_-^j:H^1(\pi_-^j\setminus  \Gamma_1 ) \to L^p(\Sigma^j),
\end{equation} 
where, letting  $\nu^j:=\big(-\sin (\alpha^j), \cos (\alpha^j)\big)$,
\begin{equation*}
\pi_+^j:=\{x \in \R^2: x\cdot \nu^j>0\}\quad\text{and}\quad
\pi_-^j:=\{x \in \R^2: x\cdot \nu^j<0\}.
\end{equation*}
 We observe that, for every $\e\in[0,1]$, the restrictions to $\mathcal H_\e$ of the operators
$\gamma_+^j$ and $\gamma_-^j$  are  continuous and
compact from $\mathcal H_\e$ into $L^p(\Sigma^j\cap\Omega)$ for all
$p\in[1,+\infty)$.

For every $j=1,\dots,k$, we define 
\begin{equation}\label{def_bj_dj}
b_j:=\cos(2 \pi \rho^j), \quad d_j:=\sin(2 \pi \rho^j),
\end{equation}
and the linear  operators 
\begin{align}
&\label{def_R}R^j:H^1(\R^2\setminus \Gamma_1)\times 
H^1(\R^2\setminus \Gamma_1)\to L^p(\Sigma^j), \\ 
&\notag R^j(\varphi,\psi):=\gamma^j_-(\varphi\vert_{\pi_-^j})-b_j \gamma^j_+(\varphi\vert_{\pi_+^j})+d_j \gamma^j_+(\psi\vert_{\pi_+^j}), \\
&\label{def_I} I^j:H^1(\R^2\setminus \Gamma_1)\times 
H^1(\R^2\setminus \Gamma_1) \to L^p(\Sigma^j),\\ &\notag I^j(\varphi,\psi):=\gamma^j_-(\psi\vert_{\pi_-^j})-d_j \gamma^j_+(\varphi\vert_{\pi_+^j})-b_j \gamma^j_+(\psi\vert_{\pi_+^j}),
\end{align}
which are continuous when
$H^1(\R^2\setminus \Gamma_1)\times 
H^1(\R^2\setminus \Gamma_1)$ is endowed with the 
norm 
\begin{equation*}
\norm{(\varphi,\psi)}_{H^1(\R^2\setminus \Gamma_1)\times 
H^1(\R^2\setminus \Gamma_1)}
:=\sqrt{\norm{\varphi}_{H^1(\R^2\setminus \Gamma_1)}^2+\norm{\psi}^2_{H^1(\R^2\setminus \Gamma_1)}}.
\end{equation*}
For every $\e \in [0,1]$, we consider the Hilbert space
$\mathcal{H}_\e\times \mathcal{H}_\e$, endowed with the norm
\begin{equation}\label{def_norm_HexHe}
\norm{(\varphi,\psi)}_{\mc{H}_\e\times \mc{H}_\e}:=\sqrt{\norm{\varphi}^2_{\mc{H}_\e}+\norm{\psi}^2_{\mc{H}_\e}},
\end{equation}
and its closed subspace 
\begin{equation}\label{def_tilde_H_e}
\widetilde{\mc{H}}_\e:= 
\left\{(\varphi,\psi)\in \mc{H}_\e\times \mc{H}_\e: R^j(\varphi,\psi)=I^j(\varphi,\psi)=0 \text{ on $\Gamma_\e^j$ for all } j=1,\dots,k\right\}.
\end{equation}
Finally, we endow the Hilbert space $L^2(\Omega) \times L^2(\Omega)$  with the norm 
\begin{equation*}
\norm{(\varphi,\psi)}_{L^2(\Omega)\times L^2(\Omega)}:=\sqrt{\norm{\varphi}^2_{L^2(\Omega)}+\norm{\psi}^2_{L^2(\Omega)}}
\end{equation*}
and the corresponding scalar product
\begin{equation*}
\big((\varphi_1,\psi_1),(\varphi_2,\psi_2)\big)_{L^2(\Omega)\times L^2(\Omega)}:=\int_\Omega (\varphi_1\varphi_2+\psi_1\psi_2)\,dx.
\end{equation*}

\begin{remark}\label{remark_psi_varphi}
It is worth noticing that, if $(\varphi,\psi) \in \widetilde{\mc{H}}_\e$, then also $(-\psi,\varphi)$ 
belongs to $ \widetilde{\mc{H}}_\e$ in view of \eqref{def_R}, \eqref{def_I} and \eqref{def_tilde_H_e}.
\end{remark}

 For every $\e\in (0,1]$ there exists  a function 
\begin{equation*}
	\Theta_\e:\R^2\setminus\{a^j_\e:j=1,\dots,k\}\to\R
\end{equation*}
such that 
\begin{equation}\label{eq:proprieta_Theta-eps}
\begin{cases}
\Theta_\e\in C^\infty(\R^2\setminus \Gamma_\e),\\
\text{$\nabla \Theta_\e$ can be extended to be in $C^\infty(\R^2\setminus\{a^j_\e:j=1,\dots,k\})$ with $\nabla\Theta_\e=\mathcal A_\e^{(\rho_1,\dots,\rho_k)}$},
\end{cases}        
\end{equation}
see Section \ref{subsec_equivalent_Aharonov_Bohm} for the construction of $\Theta_\e$. 
If $u\in H^{1,\e}(\Omega,\mb{C})$, then letting
\begin{equation}\label{def_RI_e}
v:=\mathop{\rm{Re}}(e^{-i\Theta_\e} u), \quad  w:= \mathop{\rm{Im}}(e^{-i\Theta_\e} u),
\end{equation}
we have that $(v,w)\in \widetilde{\mc{H}}_\e$ and moreover, by \eqref{eq:proprieta_Theta-eps}, 
\begin{equation}\label{eq:tras-grad}
    (i\nabla +\mc{A}_\e^{(\rho_1\dots,\rho_k)})u=ie^{i\Theta_\e}(\nabla v+i\nabla w)\quad \text{in }\Omega\setminus\Gamma_\e.
\end{equation}
It follows that,
if $\la$ is an eigenvalue of problem \eqref{prob_Aharonov-Bohm_multipole},   with $u \in H^{1,\e}_0(\Omega,\C)$ being a corresponding eigenfunction, then the pair $(v,w)\in \widetilde{\mc{H}}_\e$ defined in \eqref{def_RI_e}
solves the system
\begin{equation}\label{prob_eigenvalue_gauged}
\begin{cases}
-\Delta v= \la  v,  &\text{in } \Omega \setminus \Gamma_\e,\\
-\Delta w= \la  w,  &\text{in } \Omega \setminus \Gamma_\e,\\
v=w=0, &\text{on } \partial \Omega,\\
R^j(v,w)=I^j(v,w)=0,&\text{on }\Gamma^j_\e \text{ for all } j=1,\dots,k,\\
R^j(\nabla v\cdot \nu^j,\nabla w\cdot \nu^j)=I^j(\nabla v\cdot \nu^j, \nabla w\cdot \nu^j)=0, &\text{on }\Gamma^j_\e \text{ for all } j=1,\dots,k.\\
\end{cases}
\end{equation}
More precisely, problems \eqref{prob_Aharonov-Bohm_multipole} and  \eqref{prob_eigenvalue_gauged} share the same eigenvalues. Moreover,  through the transformation \eqref{def_RI_e}, every eigenfunction $u$ of problem \eqref{prob_Aharonov-Bohm_multipole} generates two linearly independent real eigenfunctions of \eqref{prob_eigenvalue_gauged}, given by the couples $(v,w)$ and $(-w,v)$, with  $v,w$ as in \eqref{def_RI_e}; hence the multiplicity of each eigenvalue of \eqref{prob_Aharonov-Bohm_multipole} doubles when considered as an eigenvalue of \eqref{prob_eigenvalue_gauged}.

An analogous transformation can be made for the limit problem, by means of a function 
\begin{equation*}
\Theta_0:\R^2\setminus\{0\}\to\R
\end{equation*}
introduced in Section \ref{subsec_equivalent_Aharonov_Bohm},
satisfying
\begin{equation}\label{eq:proprieta_Theta-0}
\begin{cases}
\Theta_0\in C^\infty(\R^2\setminus \Gamma_0),\\
\text{$\nabla \Theta_0$ can be extended to be in $C^\infty(\R^2\setminus\{0\})$ with $\nabla\Theta_0=A_0^{\sum_{j=1}^k \rho_j}=A_0^\rho$}.
\end{cases}        
\end{equation}
From \eqref{eq:proprieta_Theta-0} it follows that, 
if $u\in H^{1,0}(\Omega,\mb{C})$,
 $v:=\mathop{\rm{Re}}(e^{-i\Theta_0}u)$, and 
$w:=\mathop{\rm{Im}}(e^{-i\Theta_0}u)$, then 
 $(v,w)\in \widetilde{\mc{H}}_0$ and
\begin{equation}\label{eq:gradiente-gauge}
    (i\nabla +A^\rho_0)u=ie^{i\Theta_0}(\nabla v+i\nabla w)\quad \text{in }\Omega\setminus\Gamma_0.
\end{equation}
In the following we let
\begin{equation}\label{hp_la0_simple}
\la_{0,n_0} \text{ be a simple eigenvalue of problem \eqref{prob_Aharonov-Bohm_0}}   \end{equation}
and 
\begin{equation}\label{def_u0}
u_0\text{ be an eigenfunction of \eqref{prob_Aharonov-Bohm_0} associated to $\lambda_{0,n_0}$ such that	$\norm{u_0}_{L^2(\Omega,\C)}=1$.}
\end{equation}
If $\rho=\frac12$, in view of Remark \ref{rem:propertyP} it is not restrictive to  assume also 
 that $u_0$ satisfies \eqref{eq:propertyP}.
Let
\begin{equation}\label{def_v0_w0}
v_0:=\mathop{\rm{Re}}(e^{-i\Theta_0} u_0), 
\quad \text{ and } \quad 
w_0:= \mathop{\rm{Im}}(e^{-i\Theta_0} u_0).
\end{equation}
Then $(v_0,w_0)$ and $(-w_0,v_0)$ solve \eqref{prob_eigenvalue_gauged} with $\e=0$ and $\lambda=\la_{n,0}$. 
In particular, if $\la_{n,0}$ is considered as an eigenvalue of \eqref{prob_eigenvalue_gauged}, it has multiplicity $2$. In general,
the limit eigenvalue problem \eqref{prob_Aharonov-Bohm_0} and \eqref{prob_eigenvalue_gauged} with $\e=0$ share the same eigenvalues and the eigenspaces match each other through  the  multiplication by the  phase $e^{i\Theta_0}$ and the doubling of the eigenfunctions 
 $e^{i\Theta_0}(v+iw)$ into $(v,w)$ and $(-w,v)$.
See Section \ref{subsec_equivalent_Aharonov_Bohm} for  details. 

 Recalling the definition of $b_j$, $d_j$ in \eqref{def_bj_dj}, for every $\e\in(0,1]$ we define $L_\e: \mc{H}_1\times \mc{H}_1 \to\R$ as
\begin{align}\label{def_Le}
L_\e(\varphi,\psi):&= 
\sum_{j=1}^{k}(b_j-1)\int_{S^j_\e}[ \nabla v_0 \cdot \nu^j	\gamma_+^j (\varphi)+ \nabla w_0 \cdot \nu^j\gamma_+^j (\psi)] \, dS \\
\notag&\quad -\sum_{j=1}^{k}d_j\int_{S^j_\e}[ \nabla v_0 \cdot \nu^j\gamma_+^j (\psi)- \nabla w_0 \cdot \nu^j	\gamma_+^j (\varphi)]\, dS,
\end{align}
and $J_\e: \mc{H}_\e \times \mc{H}_\e\to\R$ as
\begin{equation}\label{def_Je}
J_\e(\varphi,\psi)
:=\frac{1}{2}\int_{\Omega \setminus \Gamma_\e} (|\nabla \varphi|^2+|\nabla \psi|^2) \, dx -L_\e(\varphi,\psi).
\end{equation}
By standard minimization arguments, for every $\e\in(0,1]$ there exists a unique $(V_\e,W_\e)\in\mathcal H_\e\times \mathcal H_\e$ such that 
\begin{equation}\label{def_VeWe}
\begin{cases}
(V_\e -v_0,W_\e-w_0)\in \widetilde{\mathcal H}_\e,\\
J_\e(V_\e,W_\e)=  \min\left\{J_\e(\varphi,\psi): (\varphi,\psi)\in \mathcal H_\e\times \mathcal H_\e \text{ and }(\varphi -v_0,\psi-w_0)  \in  \widetilde{\mc{H}}_\e \right\},
\end{cases}
\end{equation}
see  Proposition \ref{prop_potential}.

Our first main result states that the eigenvalue  variation $\la_{\e,n_0} - \la_{0, n_0}$ admits the following asymptotic expansion, as $\e\to0^+$, in terms of the quantities $L_\e(v_0,w_0)$ and  
\begin{equation}\label{def_Ee}
\mathcal E_\e:=J_\e(V_\e,W_\e).
\end{equation} 
\begin{theorem}\label{theorem_asymptotic_not_precise}
Suppose that \eqref{hp_cirulation} and \eqref{hp_la0_simple} hold and let $(v_0,w_0)$ be as in \eqref{def_v0_w0} with $u_0$ as in \eqref{def_u0} (and \eqref{eq:propertyP} if $\rho=\frac12$). Then 
\begin{equation}\label{eq_asymptotioc_eigenvlaues_not_precise}
\la_{\e,n_0}-\la_{0,n_0}=2(\mc{E}_\e+L_\e(v_0,w_0))+o(\|(V_\e,W_\e)\|^{2}_{\mc{H}_\e \times \mc{H}_\e})
\quad \text{as } \e \to 0^+,
\end{equation}
where $\mc{E}_\e$ and $(V_\e,W_\e)$ are as in \eqref{def_Ee} and \eqref{def_VeWe} respectively.
\end{theorem}

\begin{remark}
 If $u_0$ satisfies \eqref{def_u0}, then $e^{i\tau}u_0$ satisfies \eqref{def_u0} as well, for every $\tau\in\R$. Letting $v_0$ and $w_0$ be as in \eqref{def_v0_w0}, we have $e^{i\tau}u_0=e^{i\Theta_0} (v_{\tau}+i w_{\tau})$, with $v_\tau=v_0\cos (\tau)-w_0\sin(\tau)$ and $w_\tau=v_0\sin (\tau)+w_0\cos(\tau)$. We define  $L_{\e,\tau}$, $J_{\e,\tau}$, and $\mathcal E_{\e,\tau}$  as in \eqref{def_Le},  \eqref{def_Je}, and \eqref{def_Ee}, replacing $v_0$ and $w_0$  with $v_\tau$  and $w_\tau$, respectively. 
 We observe that 
 \begin{equation*}
     L_{\e,\tau}(\varphi,\psi)=L_\e(\varphi\cos (\tau)+\psi\sin(\tau),\psi\cos(\tau)-\varphi\sin(\tau))
 \end{equation*} 
 and 
 \begin{equation*}
     J_{\e,\tau}(\varphi,\psi)=J_\e(\varphi\cos (\tau)+\psi\sin(\tau),\psi\cos(\tau)-\varphi\sin(\tau)).
 \end{equation*}  
In particular, $L_{\e,\tau}(v_\tau,w_\tau)=L_\e(v_0,w_0)$ and $\mathcal E_{\e,\tau}=\mathcal E_\e$ for all $\tau\in\R$. 
 Hence, the coefficient $\mc{E}_\e+L_\e(v_0,w_0)$ appearing in the expansion \eqref{eq_asymptotioc_eigenvlaues_not_precise} does not depend on the choice of the eigenfunction $u_0$ in \eqref{def_u0}.
\end{remark}

Under assumption \eqref{hp_cirulation}, it is possible to describe the asymptotic behavior of $\mc{E}_\e$ as $\e \to 0^+$  in terms of the vanishing order of $v_0$ and $w_0$ at the collision point $0$. 
In order to do so we define,
letting $\{\alpha^j\}_{j=1,\dots,k}$ and $\{\rho^j\}_{j=1,\dots,k}$ be as in \eqref{def_aj} and \eqref{def_AB_vector_multi} respectively, 
\begin{equation} \label{def_f}
f(t): [0,2\pi) \to \R, \quad f(t):=\sum_{j=1}^k \rho^j \chi_{[\alpha^j+\pi,2 \pi)}(t),
\end{equation}
where 
\begin{equation}\label{eq:chi_step_func}
  \rchi_{[\alpha^j+\pi,2 \pi)}(t):=
\begin{cases}
  0, & \text {if } t \in [0,\alpha^j+\pi),\\
  1, & \text {if } t \in [\alpha^j+\pi,2 \pi).
\end{cases}
\end{equation}
As proved in  Proposition \ref{prop_vw_asympotic_not_1/2}, if $\rho \neq \frac{1}{2}$,
there exist $m\in\Z$, $\beta\in(0,+\infty)$, and $\gamma\in\big[0,\frac{2\pi}{|m+\rho|}\big)$
such that, as $\delta \to 0^+$,
\begin{equation}\label{limit_v_not_1/2}
\delta^{-|m+\rho|} v_0\big(\delta\cos t,\delta\sin t\big) \to 
\beta\cos(2\pi f(t)+(m+\rho)(\gamma-t))
\end{equation}
and 
\begin{equation}\label{limit_w_not_1/2}
\delta^{-|m+\rho|} w_0\big(\delta\cos t,\delta\sin t\big) \to 
\beta\sin(2\pi f(t)+(m+\rho)(\gamma-t))
\end{equation}
in $C^{1,\tau}\big([0,2\pi]\setminus\{\alpha^j+\pi\}_{j=1}^{k},\R\big)$ for all $\tau \in (0,1)$.
On the other hand, if $\rho =\frac{1}{2}$,   
we choose $u_0$ satisfying \eqref{eq:propertyP}, so that, in view of  Proposition \ref{prop_vw_asympotic_1/2}, there exist
 $m \in \mb{N}$, $\beta\in(0,+\infty)$, and $\gamma\in\big[0,\frac{4\pi}{2m+1}\big)$ 
 such that, as $\delta\to0^+$,
\begin{equation}\label{limit_v_1/2}
\delta^{-(m+\frac{1}{2})} v_0\big(\delta\cos t,\delta\sin t\big)\to 
\beta \cos(2\pi f(t))\cos\big((m+\tfrac 12)(\gamma+t)\big)
\end{equation}
and 
\begin{equation}\label{limit_w_1/2}
\delta^{-(m+\frac{1}{2})} w_0\big(\delta\cos t,\delta\sin t\big)\to 
\beta \sin(2\pi f(t))\cos\big((m+\tfrac 12)(\gamma+t)\big)
\end{equation}  
in $C^{1,\tau}\big([0,2\pi]\setminus\{\alpha^j+\pi\}_{j=1}^{k},\R\big)$ for all $\tau \in (0,1)$.
We introduce the corresponding homogeneous functions 
\begin{align}\label{def_Psi_rho}
&\Phi_0(x)=\Phi_0(r\cos t,r\sin t)=
\begin{cases}
    \beta \,r^{|m+\rho|}\cos(2\pi f(t)+(m+\rho)(\gamma-t)),&\text{if }\rho\neq\frac12,\\[3pt]
    \beta\, r^{m+\frac12}\cos(2\pi f(t))\cos\big((m+\tfrac 12)(\gamma+t)\big),&\text{if }\rho=\frac12,
\end{cases}\\
\label{def_Phi_rho}
&\Psi_0(x)=\Psi_0(r\cos t,r\sin t)=
\begin{cases}
    \beta \,r^{|m+\rho|}\sin(2\pi f(t)+(m+\rho)(\gamma-t)),&\text{if }\rho\neq\frac12,\\[3pt]
    \beta\, r^{m+\frac12}\sin(2\pi f(t))\cos\big((m+\tfrac 12)(\gamma+t)\big),&\text{if }\rho=\frac12,
\end{cases}
\end{align}
where $\beta,m,\gamma$ are as in \eqref{limit_v_not_1/2}-\eqref{limit_w_not_1/2} and \eqref{limit_v_1/2}-\eqref{limit_w_1/2}.

Let us define the functional space 
\begin{equation}\label{def_tilde_X}
\widetilde{\mc{X}} :=\left\{\!\!\!
\begin{array}{ll}
(\varphi,\psi) \in L^1_{\rm loc}(\R^2)\times L^1_{\rm loc}(\R^2): 
\!\!\!&\varphi,\psi \in H^1(D_r \setminus \Gamma_1) \text{ for all } r>0,\\[3pt]
&\nabla\varphi,\nabla \psi\in  L^2(\R^2 \setminus \Gamma_1,\R^2), \\[3pt]
&R^j(\varphi,\psi)=I^j(\varphi,\psi)=0 \text{ on } \Gamma^j_0 \text{ for all } 1\leq j\leq k
\end{array}
\right\}
\end{equation}
and its closed subspace 
\begin{equation}\label{def_tilde_H}
 \widetilde{\mc{H}} :=\{(\varphi,\psi) \in \widetilde{\mc{X}}: R^j(\varphi,\psi)=I^j(\varphi,\psi)=0 \text{ on } S^j_1 \text{ for all } 1\leq j\leq k\} . 
\end{equation}
We also consider the linear functional 
\begin{multline}\label{def_L}
L: \widetilde{\mc{X}} \to \R, \quad  L(\varphi,\psi)
= \sum_{j=1}^{k}(b_j-1)\int_{S^j_1}[ \nabla \Phi_0 \cdot \nu^j	\gamma_+^j (\varphi)+ \nabla \Psi_0 \cdot \nu^j\gamma_+^j (\psi)] \, dS \\
-\sum_{j=1}^{k}d_j\int_{S^j_1}[ \nabla \Phi_0 \cdot \nu^j\gamma_+^j (\psi)- \nabla \Psi_0 \cdot \nu^j	\gamma_+^j (\varphi)]\, dS
\end{multline}
and the quadratic one
\begin{equation}\label{def_J}
J: \widetilde{\mc{X}} \to \R, \quad  J(\varphi,\psi)= \frac{1}{2}\int_{\R^2\setminus \Gamma_1} (|\nabla \varphi|^2+|\nabla \psi|^2) \, dx -L(\varphi,\psi).
\end{equation}
It is worth noticing that  $L$ is actually well-defined in $H^1(D_r \setminus \Gamma_1)\times H^1(D_r \setminus \Gamma_1)$ for any $r>1$.

Let $\eta \in C^{\infty}_c(\R^2)$  be a radial cut-off function   such that 
\begin{equation}\label{def_eta}
\begin{cases}
0\le \eta(x) \le 1 \text{ for any } x \in \R^2,\\  
\eta(x)=1 \text{ if } x \in D_1, \quad  \eta(x)=0 \text{ if } x \in \R^2 \setminus D_2,\\
|\nabla \eta| \le 2  \text{ if } x \in D_2\setminus D_1.
\end{cases}
\end{equation}
By standard minimization methods,  there exists a unique  $(\widetilde{V},\widetilde{W})\in \widetilde{\mathcal X}$ such that 
\begin{equation}\label{min_prob_tilde}
\begin{cases}
 (\widetilde{V},\widetilde{W})-\eta(\Phi_0,\Psi_0) \in \widetilde{\mc{H}},\\
 J(\widetilde{V},\widetilde{W})=\min\Big\{J(\varphi,\psi): (\varphi,\psi) \in \widetilde{\mc{X}} \text{ and }
(\varphi,\psi)-\eta(\Phi_0,\Psi_0) \in \widetilde{\mc{H}}\Big\},
\end{cases}
\end{equation}
see Proposition \ref{prop_existence_tilde_VW}.

A blow-up analysis allows us to study the asymptotic behavior of the quantities $\mc{E}_\e$ and $L_\e(v_0,w_0)$ appearing in the expansion provided by Theorem \ref{theorem_asymptotic_not_precise}, consequently yielding the following more explicit result.
\begin{theorem}\label{theorem_asymptotic_precise}
Suppose that \eqref{hp_cirulation} and \eqref{hp_la0_simple} hold. Let $(v_0,w_0)$ be as in \eqref{def_v0_w0}, with $u_0$ as in \eqref{def_u0} (and $u_0$ satisfying the additional assumption \eqref{eq:propertyP} in the case $\rho=\frac12$). Let $m \in \mathbb{Z}$ be such that $|m+\rho|$ is the vanishing order of $v_0$ and $w_0$ at $0$ as in \eqref{limit_v_not_1/2}--\eqref{limit_w_not_1/2}   
or \eqref{limit_v_1/2}--\eqref{limit_w_1/2} (with $m\in\N$ in the case $\rho=\frac12$).  Then 
\begin{itemize}
\item [(i)]   $\lim_{\e \to 0^+}\e^{-2|m+\rho|} \mc{E}_\e=\mc{E}$, where 
\begin{equation}\label{def_E}
\mc{E}:=J(\widetilde{V},\widetilde{W})\text{ and $\widetilde{V},\widetilde{W}$ are as in \eqref{min_prob_tilde}.}
\end{equation}
\item [(ii)] $\la_{\e,n_0}-\la_{0,n_0}=2\e^{2|m+\rho|}\big(\mc{E}+L(\Phi_0,\Psi_0)\big)+o(\e^{2|m+\rho|})$ as  $\e \to 0^+$,
where  $\Phi_0$ and $\Psi_0$ are defined  in \eqref{def_Psi_rho} and  \eqref{def_Phi_rho}, respectively.
\end{itemize}
\end{theorem}

In the case $\rho=\frac{1}{2}$, it is possible to exhibit configurations of poles such that $\mc{E}+L(\Phi_0,\Psi_0)>0$, and other configurations for which $\mc{E}+L(\Phi_0,\Psi_0)<0$;
in these situations, Theorem \ref{theorem_asymptotic_precise}-(ii) identifies the sharp  vanishing order
of the eigenvalue variation $\la_{\e,n_0}-\la_{0,n_0}$ as $\e\to0^+$.

\begin{proposition}\label{prop_positive_negative}
Suppose that $\rho=\frac{1}{2}$. Let $v_0,w_0$ be as in \eqref{def_v0_w0}, 
with $u_0$ as in \eqref{def_u0} and \eqref{eq:propertyP},
and $m, \gamma$ be as in \eqref{limit_v_1/2}--\eqref{limit_w_1/2}. Assume that $k\le 2m+1$ and let $\alpha^j$ be as in \eqref{def_aj} for every $j=1, \dots,k$. 
\begin{itemize}
\item[(i)] 
If $\alpha^j \in  -\gamma+\frac{\pi}{2m+1}(1+2 \mathbb Z)$ for every $j=1,\dots,k$,
then 
\begin{equation}
\mc{E}<0 \quad \text{ and } \quad L(\Phi_0,\Psi_0)=0.
\end{equation}
In particular, $\la_{\e,n_0} < \la_{0,n_0}$ for sufficiently small $\e>0$. 
\item[(ii)] 
If $\alpha^j \in  -\gamma+\frac{2\pi}{2m+1}\mathbb Z$ for every $j=1,\dots,k$, 
then 
\begin{equation}
\mc{E}>0 \quad \text{ and } \quad L(\Phi_0,\Psi_0)=0.
\end{equation}
In particular, $\la_{\e,n_0} > \la_{0,n_0}$ for sufficiently small $\e>0$. 
\end{itemize}
\end{proposition}

The results of Subsection \ref{subsec_blow_up} also give us  the following insight concerning blow-up and convergence rate of eigenfunctions.
\begin{theorem}\label{theorem_blow_up_eigenfunctions}
Let $n_0 \in \mathbb{N}$ be such that \eqref{hp_la0_simple} holds. Let $u_0$ be an eigenfunction of \eqref{prob_Aharonov-Bohm_0} such that \eqref{def_u0} (together with \eqref{eq:propertyP} in the case $\rho=\frac12$)
is satisfied. For every $\e \in (0,1]$, let $u_\e$ be an eigenfunction of  \eqref{prob_Aharonov-Bohm_multipole} associated to the eigenvalue 
$\la_{\e,n_0}$ such that
\begin{equation}\label{hp_ue_remormalised}
\int_{\Omega}|u_\e|^2 \, dx=1 \quad \text{ and } \quad  \int_{\Omega} e^{-i(\Theta_\e-\Theta_0)} u_\e \overline{u_0}\, dx \text{ is a positive and real number}.
\end{equation}
Then 
\begin{equation}\label{limit_blow_up_eigenfunction_not_gauged}
\lim_{\e \to 0^+}\e^{-|m+\rho|}u_\e(\e \cdot) \to e^{i\Theta_1}(\Phi_0-\widetilde{V}+i(\Psi_0-\widetilde{W})) \quad \text{ as } \e \to 0^+,
\end{equation}
strongly in $H^{1,1}(D_r,\mb{C})$ for all $r>0$, where 
$\Phi_0$ and $\Psi_0$ are as in \eqref{def_Psi_rho} and  \eqref{def_Phi_rho}, respectively, 
and $\widetilde{V},\widetilde{W}$ are as in \eqref{min_prob_tilde}. Furthermore 
\begin{multline}\label{limit_convergence_rate_eigenfunction}
 \lim_{\e \to 0^+}\e^{-2|m+\rho|}   \int_{\Omega\setminus \Gamma_1}
|e^{-i\Theta_\e}(i\nabla +\mc{A}_\e^{(\rho_1,\dots,\rho_k)}) u_\e-e^{-i\Theta_0}(i\nabla +A^\rho_0)u_0|^2 \, dx \\=
\|\nabla \widetilde{V}\|_{L^2(\R^2\setminus \Gamma_1)}^2+
\|\nabla \widetilde{W}\|_{L^2(\R^2\setminus \Gamma_1)}^2.
\end{multline}
\end{theorem}
We observe that, by \eqref{hp_la0_simple} and \eqref{limit_lae_la0},  $\la_{\e,n_0}$ is simple as an eigenvalue of \eqref{prob_Aharonov-Bohm_multipole}, provided $\e$ is small enough. Hence, for any sufficiently small $\e$,  the eigenfunctions of \eqref{prob_Aharonov-Bohm_multipole} associated to $\la_{\e,n_0}$
 are multiples of a given one. Condition \eqref{hp_ue_remormalised} identifies, among all these, the one for which $e^{-i\Theta_\e} u_\e\to e^{-i\Theta_0} u_0$, see 
Remark~\ref{remark_spect_stability_eigenfunctions}.

\section{Preliminaries}\label{sec_preliminaries}
\subsection{Scalar potential functions}\label{subsec_scalar_potential}
For every $b=(b_1,b_2) \in \R^2$, let  $\theta_b: \R^2\setminus\{b\}
\to [0,2\pi)$ be defined as 
\begin{equation*}
\theta_b\big(b+r(\cos t,\sin t)\big)=t\quad\text{for all $t \in [0,2\pi)$ and $r>0$}.
\end{equation*}
We observe that  $\theta_b \in C^\infty(\R^2\setminus\{(x_1,b_2):x_1\geq b_1\})$ and that $\nabla \theta_b$ can be extended to be in $C^\infty(\R^2\setminus\{b\})$, 
with  $\nabla \big(\rho \theta_b\big)= A^\rho_b$ in $\R^2\setminus\{b\}$.

For every $b\in\R^2$, $\alpha \in \R$, and $x=(x_1,x_2)\in \R^2$, we consider the 
 rotation $R_{b,\alpha}$ about $b$ by an angle $\alpha$, i.e.
\begin{equation*}
R_{b,\alpha}(x):=
\begin{bmatrix}
b_1\\
b_2\\
\end{bmatrix}
+M_\alpha
\begin{bmatrix}
x_1-b_1\\
x_2-b_2\\
\end{bmatrix},
\end{equation*}
where
\begin{equation*}
M_\alpha:=\begin{bmatrix}
\cos\alpha&-\sin\alpha\\
\sin\alpha&\cos\alpha\\
\end{bmatrix}.
\end{equation*}
Letting $\theta_{b,\alpha}:= \theta_b \circ  R_{b,\alpha}$, we have
\begin{equation*}
\theta_{b,\alpha}(b+r(\cos t,\sin t))=\alpha+t\quad\text{for
every $r>0$ and $t\in[-\alpha,-\alpha+2\pi)$}.
\end{equation*}  We observe that
$\theta_{b,\alpha}$ is smooth in
$\R^2\setminus\{b+r(\cos\alpha,-\sin\alpha):r\geq0\}$ and 
$\nabla \theta_{b,\alpha}$ can be extended to be in
$C^\infty(\R^2\setminus\{b\})$, with
$\nabla \big(\rho\theta_{b,\alpha}\big)= A^\rho_b$ in $\R^2\setminus\{b\}$.

  \subsection{An equivalent eigenvalue problem by gauge transformation }\label{subsec_equivalent_Aharonov_Bohm}
For every $\e \in (0,1]$,  let
\begin{equation*}
\theta_\e^j:=	\theta_{a^j_\e,\pi-\alpha^j} \quad \text {for any  } j=1,\dots, k,
\end{equation*}
with  $\alpha^j$ as in \eqref{def_aj}. We observe that the function $\Theta_\e$ defined as
\begin{equation*}
\Theta_\e:\R^2\setminus\{a^j_\e:j=1,\dots,k\}\to\R, \quad \Theta_\e:=\sum_{j=1}^{k}\rho^j\theta_\e^j
\end{equation*}
 verifies \eqref{eq:proprieta_Theta-eps}.

For $\e\in (0,1]$, let $\la$ be an eigenvalue of problem \eqref{prob_Aharonov-Bohm_multipole}, $u \in H^{1,\e}_0(\Omega,\C)$ be a corresponding eigenfunction in the weak sense clarified in \eqref{eq_eigenfunctions_Aharonov_Bohm_multipole}, and $v$, $w$ be as in \eqref{def_RI_e}, that is
\begin{equation*}
    u=e^{i\Theta_\e} \left(v+iw \right).
\end{equation*}
It descends from the definition of $\Theta_\e$ that $(v,w) \in \widetilde{\mc{H}}_\e$, see \eqref{def_tilde_H_e},
and moreover solves the system
\begin{equation*}
\begin{cases}
\int_{\Omega \setminus \Gamma_\e} (\nabla v \cdot \nabla \varphi +\nabla w \cdot \nabla\psi)\, dx = \la \int_{\Omega}  (v \varphi+w \psi) \,  dx, \\
\int_{\Omega \setminus \Gamma_\e} (\nabla w \cdot \nabla \varphi-\nabla v \cdot \nabla \psi) \,  dx = \la \int_{\Omega}  (w \varphi - v \psi)\, dx,
\end{cases}
\end{equation*}
for any $(\varphi,\psi) \in \widetilde{\mc{H}}_\e$.
Furthermore, in view of Remark \ref{remark_psi_varphi}, the two equations of the system above are actually equivalent. We conclude that $(v,w)\in \widetilde{\mc{H}}_\e$ satisfies 
\begin{equation}\label{eq_multipole_gauged_e}
\int_{\Omega \setminus \Gamma_\e} (\nabla v \cdot \nabla \varphi +\nabla w \cdot \nabla \psi)\, dx = \la \int_{\Omega}  (v \varphi+w \psi) \,  dx \quad\text{for all } (\varphi,\psi) \in \widetilde{\mc{H}}_\e,
\end{equation}
which is the weak formulation of problem  \eqref{prob_eigenvalue_gauged}, as shown directly by integration by parts. 
On the other hand, if $(v,w) \in \widetilde{\mc{H}}_\e$ is a solution of \eqref{prob_eigenvalue_gauged} in the weak sense given by \eqref{eq_multipole_gauged_e},
then $e^{i\Theta_\e}(v+iw)$ belongs to $H^{1,\e}_0(\Omega,\C)$ and solves \eqref{eq_eigenfunctions_Aharonov_Bohm_multipole}. In conclusion, problems    \eqref{prob_eigenvalue_gauged} and \eqref{prob_Aharonov-Bohm_multipole} are equivalent, in the sense that they have the same spectrum and every eigenfunction $u=e^{i\Theta_\e} (v+iw)$ of \eqref{prob_Aharonov-Bohm_multipole} corresponds to the two linearly independent eigenfunctions $(v,w)$ and $(-w,v)$ of \eqref{prob_eigenvalue_gauged}.

In the limit case $\e=0$, we define
\begin{equation}\label{def_Theta0}
\Theta_0:\R^2\setminus\{0\}\to\R, \quad \Theta_0:=\sum_{j=1}^{k}\rho^j\theta_0^j,
\end{equation}
where $\theta^j_0=\theta_0\circ R_{0, \pi-\alpha^j}$, i.e.
\begin{equation}\label{eq:theta0j}
\theta_0^j(\cos t,\sin t)=-\alpha^j+t+\pi(1-2\rchi_{[\alpha^j+\pi,2\pi)}),\quad t\in[0,2\pi),
\end{equation}
with $\rchi$ as in \eqref{eq:chi_step_func}.
As $\Theta_0$ satisfies  \eqref{eq:proprieta_Theta-0}, arguing 
as above we have that  $u\in H_0^{1,0}(\Omega,\C)$ is an eigenfunction of \eqref{prob_Aharonov-Bohm_0} associated to the eigenvalue $\la$ if and only if, letting
\begin{equation*}
v:=\mathop{\rm{Re}}(e^{-i\Theta_0} u) \quad \text{ and } \quad w:= \mathop{\rm{Im}}(e^{-i\Theta_0} u),
\end{equation*}
$(v,w) \in \widetilde{\mc{H}}_0$  weakly solves  \eqref{prob_eigenvalue_gauged} with $\e=0$, that is
\begin{equation}\label{eq_multipole_gauged_0}
\int_{\Omega \setminus \Gamma_0} (\nabla v \cdot \nabla \varphi +\nabla w \cdot \nabla\psi)\, dx = \la \int_{\Omega}  (v \varphi+w \psi) \,  dx 
\quad \text{for all } (\varphi,\psi) \in \widetilde{\mc{H}}_0.
\end{equation}

\begin{remark}\label{rermark_u_vw}
Let  $\e \in (0,1]$, $u\in H^{1,\e}(\Omega,\mb{C})$,  
$v=\mathop{\rm{Re}}(e^{-i\Theta_\e}u)$, and 
$w=\mathop{\rm{Im}}(e^{-i\Theta_\e}u)$. Let also 
 $\xi\in H^{1,0}(\Omega,\mb{C})$,
 $g=\mathop{\rm{Re}}(e^{-i\Theta_0}\xi)$, and 
$h=\mathop{\rm{Im}}(e^{-i\Theta_0}\xi)$.
By direct computations we have
\begin{equation*}
  \int_{\Omega} e^{-i(\Theta_\e-\Theta_0)} u \bar{\xi}\, dx
  =\int_{\Omega} (vg+wh)\,dx+i\int_{\Omega} (wg-vh)\,dx,
\end{equation*}
hence
\begin{equation*}
\int_{\Omega} (wg-vh)\,dx =0 \quad \text{ and } \quad \int_{\Omega} (vg+wh)\,dx >0
\end{equation*}
if and only if 
\begin{equation*}
\int_{\Omega} e^{-i(\Theta_\e-\Theta_0)} u \bar{\xi}\, dx \text{ is a positive real number}.
\end{equation*} 
\end{remark}

\subsection{Asymptotics of  eigenfunctions of the limit problem}\label{subsec:asy-limit-eige}
The asymptotic behaviour of eigenfunctions of the limit problem \eqref{prob_Aharonov-Bohm_0} depends on weather the quantity $\rho$ defined in 
\eqref{def_rho} is half-integer or not.

\begin{proposition}\label{prop_vw_asympotic_not_1/2}
Let $\rho$ be as in \eqref{def_rho}-\eqref{hp_cirulation} and assume that $\rho \neq \frac12$. If $(v,w) \in \widetilde{\mc{H}}_0\setminus\{(0,0)\}$ satisfies \eqref{eq_multipole_gauged_0},
then there exist $m \in \mb{Z}$, $\beta\in(0,+\infty)$, and $\gamma\in\big[0,\frac{2\pi}{|m+\rho|}\big)$
such that, as $\delta \to 0^+$,
\begin{equation}\label{limit_vw_not_1/2_bis}
\delta^{-|m+\rho|} v\big(\delta\cos t,\delta\sin t\big) \to 
\beta\cos(2\pi f(t)+(m+\rho)(\gamma-t))
\end{equation}
and 
\begin{equation}\label{limit_vw_not_1/2_tris}
\delta^{-|m+\rho|} w\big(\delta\cos t,\delta\sin t\big) \to 
\beta\sin(2\pi f(t)+(m+\rho)(\gamma-t))
\end{equation}
in $C^{1,\tau}\big([0,2\pi]\setminus\{\alpha^j+\pi\}_{j=1}^{k},\R\big)$ for all $\tau \in (0,1)$, with $f$ as in \eqref{def_f}. 
Furthermore 
\begin{equation}\label{limit_vw_not_1/2_H1}
\delta^{-|m+\rho|} v(\delta \cdot) \to \Phi \quad \text{  and }  \quad \delta^{-|m+\rho|} w(\delta \cdot) \to \Psi\quad \text{as $\delta \to 0^+$},
\end{equation}
strongly in $H^1(D_r\setminus \Gamma_0)$ for all $r>0$, where, for all $\sigma\geq0$ and $t\in[0,2\pi)$, 
\begin{align*}
&\Phi(\sigma\cos t,\sigma\sin t)=\beta \sigma^{|m+\rho|} \cos(2\pi f(t)+(m+\rho)(\gamma-t)),\\
&\Psi(\sigma\cos t,\sigma\sin t)=\beta \sigma^{|m+\rho|}
\sin(2\pi f(t)+(m+\rho)(\gamma-t)).
\end{align*}
Finally, there exists a constant $C>0$ such that
\begin{equation}\label{ineq_v_pointiweise_not_1/2}
|v(x)| \le C|x|^{|m+\rho|} \quad \text{ and } \quad |\nabla v(x)| \le C|x|^{|m+\rho|-1} \quad\text{ for all } x \in \Omega\setminus\Gamma_0,
\end{equation}
and 
\begin{equation}\label{ineq_w_pointiweise_not_1/2}
|w(x)| \le C|x|^{|m+\rho|} \quad \text{ and } \quad |\nabla w(x)| \le C|x|^{|m+\rho|-1}\quad\text{ for all } x \in \Omega\setminus\Gamma_0.
\end{equation}
\end{proposition}
\begin{proof}
The function $u:=e^{i\Theta_0}(v+iw)$ is an eigenfunction of \eqref{prob_Aharonov-Bohm_0} as observed in Section \ref{subsec_equivalent_Aharonov_Bohm}.
Since $\rho \neq \frac12$, by \cite[Theorem 1.3, Section 7]{FFT} there exist $m \in \mathbb{Z}$ 
and a constant $c \in \mb{C}\setminus\{0\}$ such that, as $\delta \to 0^+$,
\begin{equation}\label{proof:prop_vw_asympotic_1}
\delta^{-|m+\rho|}u(\delta \cos t,\delta \sin t) \to c e^{-i mt} 
\end{equation}
in $C^{1,\tau}([0,2\pi])$
and 
\begin{equation}\label{proof:prop_vw_asympotic_2}
\delta^{1-|m+\rho|}\nabla u(\delta \cos t,\delta \sin t) \to 
ce^{-i mt}\Big(|m+\rho| (\cos t,\sin t)-im(-\sin t,\cos t)\Big)
\end{equation}
in $C^{0,\tau}([0,2\pi])$,
for every $\tau \in (0,1)$.
Furthermore, in view of \eqref{def_Theta0}, \eqref{eq:theta0j} and \eqref{def_f}
\begin{equation}\label{proof:prop_vw_asympotic_2.5}
\Theta_0(\delta\cos t,\delta \sin t)=\rho t  -2\pi f(t) +\pi \rho-\sum_{j=1}^k \rho^j \alpha^j.
\end{equation}
It follows that, as $\delta \to 0^+$,
\begin{equation}\label{proof:prop_vw_asympotic_3}
\delta^{-|m+\rho|}(v(\delta\cos t,\delta \sin t)+iw(\delta\cos t,\delta \sin t))
\to c e^{i\left(\sum_{j=1}^k \rho^j \alpha^j-\pi \rho\right)}e^{2\pi i f(t)}e^{-i(m+\rho)t}
\end{equation}
in $C^{1,\tau}([0,2\pi]\setminus \{\alpha^j+\pi\}_{j=1}^k)$ for any $\tau \in (0,1)$. We deduce  \eqref{limit_vw_not_1/2_bis} and \eqref{limit_vw_not_1/2_tris} from \eqref{proof:prop_vw_asympotic_3}  by taking the real and imaginary parts, respectively.
Letting
\begin{equation*}
\tilde u_\delta:=\delta^{-|m+\rho|}u(\delta \cdot) \quad \text{ and } \quad \Upsilon(x)=\Upsilon(\sigma \cos t,\sigma\sin t)= c \,\sigma^{|m+\rho|} \,e^{-i mt},
\end{equation*}
from  \eqref{proof:prop_vw_asympotic_1}, \eqref{proof:prop_vw_asympotic_2}, and  the Dominated Convergence Theorem it follows that 
\begin{equation}\label{proof:prop_vw_asympotic_4}
\nabla \tilde u_\delta\to  \nabla \Upsilon \quad   \text{ and }
\quad \frac{\tilde u_\delta}{|x|}\to  \frac{\Upsilon}{|x|} \quad \text{ strongly in }
L^2(D_r)
\end{equation}
as $\delta \to 0^+$, for any $r>0$. 
Since $u=e^{i\Theta_0}(v+iw)$, \eqref{limit_vw_not_1/2_H1} follows from \eqref{proof:prop_vw_asympotic_4}.
Finally we can deduce \eqref{ineq_v_pointiweise_not_1/2} and \eqref{ineq_w_pointiweise_not_1/2} 
from \eqref{proof:prop_vw_asympotic_1} and  \eqref{proof:prop_vw_asympotic_2}.
\end{proof}

\begin{proposition}\label{prop_vw_asympotic_1/2}
Let $\rho=\frac12$ and $u\in H^{1,0}_0(\Omega,\C)\setminus\{0\}$ be an eigenfunction of problem \eqref{prob_Aharonov-Bohm_0} satisfying \eqref{eq:propertyP}.
Let 
$v=\mathop{\rm{Re}}(e^{-i\Theta_0} u)$ and 
$w=\mathop{\rm{Im}}(e^{-i\Theta_0} u)$. Then there exist  $m \in \mb{N}$, $\beta\in(0,+\infty)$,  and 
$\gamma\in\big[0,\frac{4\pi}{2m+1}\big)$ such that 
\begin{equation}\label{limit_v_1/2_bis}
\delta^{-(m+\frac{1}{2})} v\big(\delta\cos t,\delta\sin t\big)\to 
\beta \cos(2\pi f(t)) \cos\big((m+\tfrac 12)(\gamma+t)\big)
\end{equation}
and 
\begin{equation}\label{limit_v_1/2_tris}
\delta^{-(m+\frac{1}{2})} w\big(\delta\cos t,\delta\sin t\big)\to 
\beta \sin(2\pi f(t)) \cos\big((m+\tfrac 12)(\gamma+t)\big)
\end{equation}
as $\delta \to 0^+$ in $C^{1,\tau}\big([0,2\pi]\setminus\{\alpha^j+\pi\}_{j=1}^{k},\R\big)$ for all $\tau \in (0,1)$. 
Furthermore 
\begin{equation}\label{limit_vw_1/2_H1}
\delta^{-(m+\frac 12)} v(\delta \cdot) \to \Phi \quad \text{  and }  \quad \delta^{-(m+\frac 12)} w(\delta \cdot) \to \Psi\quad \text{as $\delta \to 0^+$},
\end{equation}
strongly in $H^1(D_r\setminus \Gamma_0)$ for all $r>0$, where, for all $\sigma\geq0$ and $t\in[0,2\pi)$, 
\begin{align*}
    &\Phi(\sigma\cos t,\sigma\sin t)=\beta \sigma^{m+\frac 12}\cos(2\pi f(t))\cos\big((m+\tfrac 12)(\gamma+t)\big),\\
&    \Psi(\sigma\cos t,\sigma\sin t)=\beta \sigma^{m+\frac 12}\sin(2\pi f(t))\cos\big((m+\tfrac 12)(\gamma+t)\big).
\end{align*}
Finally, there exists a constant $C>0$ such that
\begin{equation}\label{ineq_v_pointiweise_1/2}
|v(x)| \le C|x|^{m+\frac 12} \quad \text{ and } \quad |\nabla v(x)| \le C|x|^{m-\frac 12} \quad\text{ for all } x \in \Omega\setminus\Gamma_0,
\end{equation}
and 
\begin{equation}\label{ineq_w_pointiweise_1/2}
|w(x)| \le C|x|^{m+\frac 12} \quad \text{ and } \quad |\nabla w(x)| \le C|x|^{m-\frac 12}\quad\text{ for all } x \in \Omega\setminus\Gamma_0.
\end{equation}
\end{proposition}
\begin{proof}
Since $u:=e^{i\Theta_0}(v+iw) $ is an eigenfunction of problem \eqref{prob_Aharonov-Bohm_0} with  $\rho = \frac12$, by \cite[Theorem 1.3, Section 7]{FFT} there exist  $m \in \mb{N}$  
and  $(c_1,c_2) \in \mb{C}\setminus\{(0,0)\}$ such that, as $\delta \to 0^+$,
\begin{equation}\label{eq:conve1}
\delta^{-(m+\frac{1}{2})}u(\delta \cos t,\delta \sin t) \to e^{\frac{i}{2}t}\left(c_1\cos\left((m+\tfrac{1}{2})t\right)+c_2\sin\left((m+\tfrac{1}{2})t\right)\right)
\text{ in } C^{1,\tau}([0,2\pi],\C).
\end{equation}
Furthermore, since $u$ satisfies \eqref{eq:propertyP}, we can rewrite the right hand side of 
\eqref{eq:conve1} as 
\begin{equation*}
    \beta e^{i\frac t2}e^{i\Lambda}\cos\big((m+\tfrac 12)(\gamma+t)\big)
\end{equation*}
for some $\beta\in (0,+\infty)$ and $\gamma\in\big[0,\frac{4\pi}{2m+1}\big)$.
 
By \eqref{proof:prop_vw_asympotic_2.5} it follows that, as $\delta \to 0^+$,
\begin{align*}
\frac{v(\delta\cos t,\delta \sin t)}{\delta^{m+\frac{1}{2}}}\to 
&\mathop{\rm Re}\left(\beta e^{i2\pi f(t)}\cos\big((m+\tfrac 12)(\gamma+t)\big)
\right)=\beta \cos\big(2\pi f(t))\cos\big((m+\tfrac 12)(\gamma+t)\big)\\
\frac{w(\delta\cos t,\delta \sin t)}{\delta^{m+\frac{1}{2}}}
\to &
\mathop{\rm Im}\left(\beta e^{i2\pi f(t)}\cos\big((m+\tfrac 12)(\gamma+t)\big)
\right)=\beta \sin\big(2\pi f(t))\cos\big((m+\tfrac 12)(\gamma+t)\big)
\end{align*}
in $C^{1,\tau}([0,2\pi]\setminus \{\alpha^j+\pi\}_{j=1}^k)$ for any $\tau \in (0,1)$, thus proving \eqref{limit_v_1/2_bis} and \eqref{limit_v_1/2_tris}. Finally, \eqref{limit_vw_1/2_H1}, \eqref{ineq_v_pointiweise_1/2}, and \eqref{ineq_w_pointiweise_1/2}
 can be proved arguing as in Proposition \ref{prop_vw_asympotic_not_1/2}.
\end{proof}

\section{Properties of $\mc{E}_\e$}\label{sec_propoerties_Ee}
Let $n_0 \in \mb{N} \setminus\{0\}$ and  $u_0$ be an eigenfunction of problem \eqref{prob_Aharonov-Bohm_0} associated to the eigenvalue $\la_{0,n_0}$, such that  \eqref{def_u0} (together with \eqref{eq:propertyP} if $\rho=\frac12$) is satisfied. 
Let  $(v_0,w_0)\in \widetilde{\mc{H}}_0$ be as \eqref{def_v0_w0}. 

The linear functional $L_\e$ in \eqref{def_Le} is well-defined by the H\"older inequality.
Indeed, by \eqref{ineq_v_pointiweise_not_1/2}--\eqref{ineq_w_pointiweise_not_1/2} and \eqref{ineq_v_pointiweise_1/2}--\eqref{ineq_w_pointiweise_1/2}, 
for all $j=1,\dots,k$, we have  
\begin{equation*}
|\nabla v_0|,|\nabla w_0| \in L^p(S_\e^j) \quad \text{ for every } p \in \left[1,
\min\bigg\{\frac1\rho,\frac1{1-\rho}\bigg\}\right);
\end{equation*}
on the other hand, if $\varphi \in \mc{H}_1$, then $\gamma^j_+(\varphi) \in L^q(S_\e^j)$ for all $q \in [1,+\infty)$ and 
 $j=1,\dots,k$, by \eqref{def_traces} and boundedness of $S_\e^j$.
We now prove that $L_\e\in (\mc{H}_1\times\mc{H}_1 )^*$, being $(\mc{H}_1\times\mc{H}_1 )^*$ the dual space of $\mc{H}_1\times\mc{H}_1$, providing an estimate of the dual norm.

\begin{proposition}\label{prop_Le_continuity}
Let $(v_0,w_0)$ be as above and let $m$ be as in Proposition \ref{prop_vw_asympotic_not_1/2},  if $\rho \neq \frac12$, or as in 
Proposition \ref{prop_vw_asympotic_1/2}, if $\rho=\frac12$, with $(v,w)=(v_0,w_0)$. Then, for every $\e \in (0,1]$, the functional $L_\e$ defined in \eqref{def_Le} belongs to $(\mc{H}_1\times \mc{H}_1)^*$ and 
\begin{equation*}
\norm{L_\e}_{ (\mc{H}_1\times\mc{H}_1 )^*}=O(\e^{|m+\rho|-1+\frac{1}{p}}) \quad \text{ as } \e \to 0^+,
\end{equation*}
for every $p \in \big(1,\min\big\{\frac1\rho,\frac1{1-\rho}\big\}\big)$. In particular, $\lim_{\e\to0^+}\norm{L_\e}_{ (\mc{H}_1\times\mc{H}_1 )^*}=0$.
\end{proposition}

\begin{proof}
For every $p \in \big(1,\min\big\{\frac1\rho,\frac1{1-\rho}\big\}\big)$,   $(\varphi,\psi) \in \mc{H}_1\times\mc{H}_1$, and $\e \in (0,1]$,  by the H\"older inequality, the continuity of the trace operators in \eqref{def_traces}, \eqref{ineq_v_pointiweise_not_1/2}--\eqref{ineq_w_pointiweise_not_1/2}, \eqref{ineq_v_pointiweise_1/2}--\eqref{ineq_w_pointiweise_1/2}, and \eqref{def_Le},  we have 
\begin{align*}
|L_\e(\varphi,\psi)|&\le  \sum_{j=1}^{k}|b_j-1| \Big(\norm{\nabla v_0}_{L^p(S_\e^j)}	\|\gamma_+^j (\varphi)\|_{L^{p'}(S_\e^j)}+
\norm{\nabla w_0}_{L^p(S_\e^j)}\|\gamma_+^j (\psi)\|_{L^{p'}(S_\e^j)}\Big) \\
&\qquad+\sum_{j=1}^{k}|d_j| \Big( \norm{\nabla v_0}_{L^p(S_\e^j)}\|\gamma_+^j (\psi)\|_{L^{p'}(S_\e^j)}
+ \norm{\nabla w_0}_{L^p(S_\e^j)}\|\gamma_+^j (\varphi)\|_{L^{p'}(S_\e^j)}\Big) \\
&\le C \e^{|m+\rho|-1+\frac{1}{p}}\norm{(\varphi,\psi)}_{\mc{H}_1\times \mc{H}_1},
    \end{align*}
where $p':=\frac{p}{p-1}$, for some constant $C>0$ independent of $\e$ and $(\varphi,\psi)$.
\end{proof}

For any $\e \in (0,1]$, we   now consider  the minimization problem 
\begin{equation}\label{min_prob}
 \inf\left\{J_\e(\varphi,\psi): (\varphi,\psi)\in \mathcal H_\e\times \mathcal H_\e 
 \text{ and }(\varphi -v_0,\psi-w_0)  \in  \widetilde{\mc{H}}_\e \right\},
\end{equation}
where $J_\e$ and $\widetilde{\mc{H}}_\e$ are defined  in \eqref{def_Je} and  \eqref{def_tilde_H_e}, respectively.

\begin{proposition}\label{prop_potential}
The infimum in \eqref{min_prob} is achieved by a unique couple $(V_\e,W_\e) \in \mc{H}_\e\times\mc{H}_\e$. Furthermore, $(V_\e,W_\e)$ is a weak solution of the problem 
\begin{equation*}
\begin{cases}
-\Delta V_\e=0,  &\text{in } \Omega \setminus \Gamma_\e,\\
-\Delta W_\e=0,  &\text{in } \Omega \setminus \Gamma_\e,\\
V_\e=W_\e=0, &\text{on } \partial \Omega,\\
R^j(V_\e-v_0,W_\e-w_0)=I^j(V_\e-v_0,W_\e-w_0)=0,&\text{on }\Gamma^j_\e \text{ for all } j=1,\dots,k,\\
R^j(\nabla (V_\e-v_0)\cdot \nu^j,\nabla (W_\e-w_0)\cdot \nu^j)=0, &\text{on }\Gamma^j_\e \text{ for all } j=1,\dots,k,\\
I^j(\nabla (V_\e-v_0)\cdot \nu^j, \nabla  (W_\e-w_0)\cdot \nu^j)=0, &\text{on }\Gamma^j_\e \text{ for all } j=1,\dots,k,
\end{cases}
\end{equation*}
i.e.  $(V_\e,W_\e) \in \mc{H}_\e\times \mc{H}_\e$,  $(V_\e-v_0,W_\e-w_0) \in \widetilde{\mc{H}}_\e$, and 
\begin{equation}\label{eq_potential}
\int_{\Omega \setminus \Gamma_\e}  (\nabla V_\e \cdot \nabla \varphi+\nabla W_\e \cdot \nabla \psi)\, dx =L_\e(\varphi,\psi) \quad 
\text{for all } (\varphi,\psi) \in \widetilde{\mc{H}}_\e.
\end{equation}
\end{proposition}
\begin{proof}
In view of Proposition \ref{prop_Le_continuity}, the functional $J_\e$  is convex, continuous, and coercive on  the closed convex set
$(v_0,w_0)+\widetilde{\mc{H}}_\e:=\{(\varphi,\psi) \in \mc{H}_\e \times \mc{H}_\e: (\varphi-v_0,\psi-w_0) \in \widetilde{\mc{H}}_\e \}$. Then, there exists $(V_\e,W_\e) \in 
(v_0,w_0)+\widetilde{\mc{H}}_\e$ attaining the infimum in
\eqref{min_prob}, and hence satisfying \eqref{eq_potential}. 

Let $(V_{\e,1},W_{\e,1})$ and $(V_{\e,2},W_{\e,2})$ both satisfy \eqref{eq_potential}, so that
\begin{equation}\label{proof:prop_potential_1}
\int_{\Omega \setminus \Gamma_\e}  \big(\nabla( V_{\e,1}-V_{\e,2} )\cdot \nabla \varphi+\nabla (W_{\e,1}-W_{\e,2}) \cdot \nabla \psi\big)\, dx=0\quad\text{for every $(\varphi,\psi) \in \widetilde{\mc{H}}_\e$}.
\end{equation}
Since  $(V_{\e,1}-V_{\e,2},W_{\e,1}-W_{\e,2}) \in \widetilde{\mc{H}}_\e$, we may test \eqref{proof:prop_potential_1} 
with $(V_{\e,1}-V_{\e,2},W_{\e,1}-W_{\e,2})$ and conclude that $\nabla (V_{\e,1}-V_{\e,2})=\nabla(W_{\e,1}-W_{\e,2})=0$ in $\Omega\setminus\Gamma_\e$.
From \eqref{prop:appendix3_2} we deduce that $V_{\e,1}=V_{\e,2}$ and $W_{\e,1}=W_{\e,2}$, thus proving the uniqueness of the minimizer.
\end{proof}

For every $r>0$, let
\begin{equation}\label{def_eta_r}
\eta_r(x):=\eta\left(\frac{x}{r}\right),
\end{equation}
with $\eta$ as in \eqref{def_eta}.

\begin{proposition}\label{prop_Ee_lower_upper_estimates}
Let  $m$ be as in Proposition \ref{prop_vw_asympotic_not_1/2} if $\rho \neq \frac12$ or as in 
Proposition \ref{prop_vw_asympotic_1/2} if $\rho=\frac12$, with $(v,w)=(v_0,w_0)$. 
Let $\mathcal E_\e$ be defined  in \eqref{def_Ee}.
Then there exist   $C_1>0$ and, for every $p \in \big(1,\min\big\{\frac1\rho,\frac1{1-\rho}\big\}\big)$,
  $C_2=C_2(p)>0$, such that
\begin{equation}\label{ineq_Ee_upper_lower_bounds}
\mc{E}_\e \le C_1 \e^{2|m+\rho|} \quad  \text{ and } \quad \mc{E}_\e \ge - C_2 \e^{2|m+\rho|-2+\frac{2}{p}} \quad \text{for all $\e \in (0,1]$}.
\end{equation}
 In  particular, $\lim_{\e\to0^+}\mc{E}_\e= 0$.
\end{proposition}

\begin{proof}
Let $\eta_\e$ be as in \eqref{def_eta_r} with $r=\e$. Then, since $(\eta_\e v_0, \eta_\e w_0)-(v_0,w_0) \in \widetilde{\mc{H}}_\e$, 
\begin{align}\label{proof:prop_Ee_lower_upper_estimates_1}
J_\e(V_\e,&W_\e) \le J_\e(\eta_\e v_0, \eta_\e w_0) \\
&\notag
\le \frac12 \int_{\Omega \setminus \Gamma\e}(|\nabla(\eta_\e v_0)|^2+|\nabla(\eta_\e w_0)|^2  )\, dx\\
&\notag \qquad+\sum_{j=1}^k \int_{S^j_\e}[ |b_j-1| (|\nabla v_0| |v_0| + |\nabla w_0| |w_0|)+|d_j|(|\nabla v_0| |w_0|+|\nabla w_0| |v_0|)] \, dS\\
&\notag\le  \int_{\Omega\cap D_{2\e}}|\nabla \eta_\e|^2(|v_0|^2+|w_0|^2 )\, dx
+\int_{(\Omega\cap D_{2\e})\setminus \Gamma_\e} (|\nabla v_0|^2+|\nabla w_0|^2) \, dx +C \e^{2|m+\rho|} \\
&\notag\le C_1\e^{2|m+\rho|}
\end{align}
for some positive constants $C>0$ and $C_1>0$ independent of $\e$, in view of \eqref{def_Le}, \eqref{def_Je}, \eqref{ineq_v_pointiweise_not_1/2}--\eqref{ineq_w_pointiweise_not_1/2}, and \eqref{ineq_v_pointiweise_1/2}--\eqref{ineq_w_pointiweise_1/2}. The first estimate in \eqref{ineq_Ee_upper_lower_bounds}  follows from \eqref{proof:prop_Ee_lower_upper_estimates_1}.

On the other hand, by \eqref{def_Je} and \eqref{def_Ee} 
\begin{align*}
\norm{(V_\e,W_\e)}_{\mc{H}_1\times \mc{H}_1}^2&=\norm{(V_\e,W_\e)}_{\mc{H}_\e\times \mc{H}_\e}^2= 2\mc{E}_\e +2L_\e(V_\e,W_\e) 
\le 2\mc{E}_\e +2|L_\e(V_\e,W_\e)| \\
& \le   2\mc{E}_\e +2 \norm{L_\e}_{(\mc{H}_1\times \mc{H}_1)^*} \norm{(V_\e,W_\e)}_{\mc{H}_1\times \mc{H}_1} \\
&\le  2\mc{E}_\e +2 \norm{L_\e}^2_{(\mc{H}_1\times \mc{H}_1)^*}+\frac{1}{2}  \norm{(V_\e,W_\e)}^2_{\mc{H}_1\times \mc{H}_1},
\end{align*}
thus implying that
\begin{equation}\label{proof:prop_Ee_lower_upper_estimates_3}
\mc{E}_\e +\norm{L_\e}_{(\mc{H}_1\times \mc{H}_1)^*}^2 \ge \frac{1}{4}  \norm{(V_\e,W_\e)}^2_{\mc{H}_1\times \mc{H}_1}\ge 0. 
\end{equation}
The second estimate in   \eqref{ineq_Ee_upper_lower_bounds} follows from Proposition \ref{prop_Le_continuity} and \eqref{proof:prop_Ee_lower_upper_estimates_3}.
\end{proof}

\begin{proposition}\label{prop_VWe_to_0}
We have  $(V_\e,W_\e) \to 0$ as $\e \to 0^+$ strongly in $\mc{H}_1\times\mc{H}_1$.
\end{proposition}

\begin{proof}
From Proposition \ref{prop_Le_continuity} and Proposition \ref{prop_Ee_lower_upper_estimates} it follows that
$\lim_{\e\to 0^+}\norm{L_\e}_{(\mc{H}_1\times\mc{H}_1)^*}=0$ and $\lim_{\e\to 0^+}\mc{E}_\e=0$. Hence the claim follows from \eqref{proof:prop_Ee_lower_upper_estimates_3}.
\end{proof}

\begin{proposition}\label{prop_Ee_o_VWe}
We have $\mc{E}_\e=o(\norm{(V_\e,W_\e)}_{\mc{H}_\e\times \mc{H}_\e})$ as $\e \to 0^+$.
\end{proposition}

\begin{proof}
By \eqref{def_Ee} and \eqref{def_Je}
\begin{equation*}
|\mc{E}_\e| \le \frac{1}{2}\norm{(V_\e,W_\e)}^2_{\mc{H}_\e\times \mc{H}_\e}+
\norm{L_\e}_{(\mc{H}_1\times \mc{H}_1)^*}\norm{(V_\e,W_\e)}_{\mc{H}_\e\times \mc{H}_\e},
\end{equation*}
hence the conclusion follows from Propositions \ref{prop_Le_continuity} and \ref{prop_VWe_to_0}.
\end{proof}

\begin{proposition}\label{prop_norm_L2_norm_nabla}
We have 
\begin{equation*}
 \int_{\Omega}( V_\e^2+W_\e^2) \, dx  =o(\norm{(V_\e,W_\e)}^2_{\mc{H}_\e \times \mc{H}_\e}) \text{ as } \e \to 0^+. 
\end{equation*}
\end{proposition}

\begin{proof}
We argue by contradiction, assuming that there exist a positive constant $C>0$ and a sequence $\{\e_n\}_{n \in \mathbb{N}}$ 
such that $\lim_{n \to \infty}\e_n=0$, $(V_{\e_n},W_{\e_n})\not\equiv(0,0)$, and
\begin{equation}\label{proof:norm_L2_norm_nabla_1}
\int_{\Omega}(V_{\e_n}^2+W_{\e_n}^2) \, dx \ge C\int_{\Omega \setminus \Gamma_{\e_n}}( |\nabla V_{\e_n}|^2 +|\nabla W_{\e_n}|^2)\, dx.
\end{equation}
Letting, for every  $n \in \mathbb{N}$,
\begin{equation*}
Y_n:=\frac{V_{\e_n}}{\norm{(V_{\e_n},W_{\e_n})}_{L^2(\Omega)\times L^2(\Omega)}} \quad  \text { and } \quad 
Z_n:=\frac{W_{\e_n}}{\norm{(V_{\e_n},W_{\e_n})}_{L^2(\Omega)\times L^2(\Omega)}},
\end{equation*} 
we have
\begin{equation}\label{eq:normalization}
\int_{\Omega}(Y_n^2+Z_n^2)\, dx =1
\end{equation}
and, by \eqref{proof:norm_L2_norm_nabla_1}, $\{Y_n\}_{n\in\mb{N}}$ and $\{Z_n\}_{n\in\mb{N}}$ are bounded in $\mc{H}_1$.
Hence there exist $Y \in \mc{H}_1$ and $Z\in \mc{H}_1$  such that $Y_n \rightharpoonup Y$ and $Z_n  \rightharpoonup  Z$ weakly in $\mc{H}_1$ as $n \to \infty$, up to a subsequence.
Since $Y_n,Z_n \in \mc{H}_{\e_n}$ for every $n \in \mathbb{N}$,
by \cite[Proposition 3.3]{FNOS_multipole} we conclude that  $Y,Z \in \mc{H}_0$.
Furthermore, in view of  
\eqref{eq:normalization} and the compactness of the natural embedding 
\begin{equation}\label{eq:compact-embedding}
    \mc{H}_1 \times \mc{H}_1 
\hookrightarrow L^2(\Omega) \times L^2(\Omega),
\end{equation}
see  \cite[Remark 3.1]{FNOS_multipole}, we have
\begin{equation}\label{proof:norm_L2_norm_nabla_1.5}
\int_{\Omega}(Y^2+Z^2)\, dx =1.
\end{equation}
Since 
\begin{equation*}
\bigg(Y_n -\frac{v_0}{\norm{(V_{\e_n}W_{\e_n})}_{L^2(\Omega)\times L^2(\Omega)}},
Z_n -\frac{w_0}{\norm{(V_{\e_n}W_{\e_n})}_{L^2(\Omega)\times L^2(\Omega)}}\bigg) \in \widetilde{\mc{H}}_{\e_n}
\quad\text{for all }n\in\N,
\end{equation*}
we have  $R^j(Y_n,Z_n)=I^j(Y_n,Z_n)=0$ on $\Gamma_0^j$ for every $j=1,\dots, k$  and  $n \in \mathbb{N}$. Then, by continuity of the operators in \eqref{def_traces}, we conclude that $R^j(Y,Z)=I^j(Y,Z)=0$ on $\Gamma_0^j$ for every $j=1,\dots, k$, i.e. $(Y,Z) \in \widetilde{\mc{H}}_0$.

Let $(\varphi,\psi) \in \widetilde{\mc{H}}_{0,0}$, where 
\begin{equation}\label{def_tilde_H_00}
\widetilde{\mc{H}}_{0,0}:=\{(\varphi,\psi) \in \widetilde{\mc{H}}_0:\varphi\equiv 0 \text{ and } \psi\equiv 0 \text{ in a neighbourhood of } 0\}.
\end{equation}
If $n$ is large enough, then $(\varphi,\psi) \in \widetilde{\mc{H}}_{\e_n}$ and $L_{\e_n}(\varphi,\psi)=0$. Hence, testing \eqref{eq_potential} with 
$(\varphi,\psi)$ yields 
\begin{equation*}
\int_{\Omega \setminus \Gamma_{\e_n}} (\nabla Y_n \cdot \nabla \varphi +\nabla Z_n \cdot \nabla \psi) \, dx =
\|(V_{\e_n},W_{\e_n})\|^{-1}_{L^2(\Omega)\times L^2(\Omega)}
L_{\e_n}(\varphi,\psi)=0.
\end{equation*}
Passing to the limit as $n \to \infty$ we conclude that 
\begin{equation}\label{proof:norm_L2_norm_nabla_2}
\int_{\Omega \setminus \Gamma_0} (\nabla Y \cdot \nabla \varphi +\nabla Z \cdot \nabla \psi) \, dx   =0
\end{equation}
for every $(\varphi,\psi) \in \widetilde{\mc{H}}_{0,0}$. 
Arguing as \cite[Lemma 3.4]{FNOS_multipole}, one can prove that the space  $\widetilde{\mc{H}}_{0,0}$ defined in \eqref{def_tilde_H_00} is dense in $\widetilde{\mc{H}}_0$. Hence
 \eqref{proof:norm_L2_norm_nabla_2} actually holds for all $(\varphi,\psi) \in \widetilde{\mc{H}}_0$. Then we may test \eqref{proof:norm_L2_norm_nabla_2} with $(Y,Z)$, thus obtaining  $Y=Z=0$ in view of \eqref{prop:appendix3_2}, and contradicting \eqref{proof:norm_L2_norm_nabla_1.5}.
\end{proof}

\section{Asymptotic expansion of the eigenvalue variation}\label{sec_asymptotic_eigenvalue_variation}
For every $\e\in [0,1]$, we define  the following bilinear form on the space $\widetilde{\mc{H}}_\e$ introduced in \eqref{def_tilde_H_e}:
\begin{equation}\label{def_qe}
q_\e:\widetilde{\mc{H}}_\e \times \widetilde{\mc{H}}_\e  \to \R, \quad 
q_\e\big((\varphi_1,\psi_1),(\varphi_2,\psi_2)\big):=\int_{\Omega \setminus \Gamma_\e}(\nabla\varphi_1 \cdot \nabla\varphi_2+\nabla\psi_1\cdot\nabla\psi_2 )\, dx,
\end{equation}
i.e. $q_\e$ is the scalar product associated to the norm \eqref{def_norm_HexHe}. For the sake of simplicity, we still denote with $q_\e$ the associated quadratic form
\begin{equation*}
q_\e:\widetilde{\mc{H}}_\e\to [0,+\infty), \quad 
q_\e(\varphi,\psi):=\int_{\Omega \setminus \Gamma_\e}(|\nabla\varphi|^2+|\nabla\psi|^2 )\, dx
=\|(\varphi,\psi)\|^2_{\mathcal H_\e\times \mathcal H_\e}.
\end{equation*}
By the Riesz Representation Theorem,  for every $\e\in[0,1]$ there exists a linear and continuous operator  
$\mathcal{F}_\e:\widetilde{\mc{H}}_\e \to \widetilde{\mc{H}}_\e$ such that, for every $(\varphi_1,\psi_1),(\varphi_2,\psi_2)\in \widetilde{\mathcal H}_\e$ 
\begin{equation}\label{def_Fe}
 q_\e(\mathcal{F}_\e(\varphi_1,\psi_1),(\varphi_2,\psi_2))=\big((\varphi_1,\psi_1),(\varphi_2,\psi_2)\big)_{L^2(\Omega)\times L^2(\Omega)}.
\end{equation}
Abstract spectral theory, see e.g. \cite{H_spectral}, together with compactness of the embedding \eqref{eq:compact-embedding}, yield the following  preliminary result,
see also \cite[Proposition 5.1]{FNOS_multipole}.
\begin{proposition}\label{prop_spectral}
Let $\e\in[0,1]$ and $\mathcal{F}_\e$ be as in \eqref{def_Fe}. 
 Then 
\begin{enumerate}[\rm (i)]
\item $\mathcal{F}_\e$ is symmetric, compact and non-negative; in particular $0$ belongs to its spectrum $\sigma(\mathcal{F}_\e)$.
\item $\sigma(\mathcal{F}_\e)\setminus\{0\}=\{\mu_{n,\e}\}_{n\in\N\setminus\{0\}}$, where $\mu_{n,\e}=1/\la_{\e,n}$ for any $n \in \mb{N}\setminus\{0\}$.
\item For any $\mu \in \R$ and $(\varphi,\psi) \in  \widetilde{\mc{H}}_\e$
\begin{equation*}
\left(\mathop{\rm{dist}}(\mu,\sigma(\mathcal{F}_\e)\right)^2 \le \frac{q_\e(\mathcal{F}_\e(\varphi,\psi)-\mu (\varphi,\psi))}{q_\e(\varphi,\psi)}.
\end{equation*}
\end{enumerate}
\end{proposition}

For some fixed $n_0 \in \mb{N}\setminus \{0\}$, let $u_0$ and  $(v_0,w_0)$ 
be as in \eqref{def_u0} and \eqref{def_v0_w0}, respectively (with the further assumption \eqref{eq:propertyP} if $\rho=\frac12$).  
 In  order to obtain an asymptotic expansion of the eigenvalue variation, we make the additional assumption that 
 \begin{equation*}
    \la_0:=\la_{0,n_0}\quad\text{satisfies \eqref{hp_la0_simple}}.
\end{equation*}
Therefore, by \eqref{limit_lae_la0}, also the eigenvalue $\la_{\e,n_0}$ is simple (as an eigenvalue of \eqref{prob_Aharonov-Bohm_multipole}, double as an eigenvalue of \eqref{prob_eigenvalue_gauged}) if $\e$ is small enough.  To simplify the notations, we will write from now on 
\begin{equation*}
\la_\e:=\la_{\e,n_0}.
\end{equation*} 
For $\e$ small, let $u_\e$ be an eigenfunction of problem \eqref{prob_Aharonov-Bohm_multipole} associated to $\la_\e$ and 
\begin{equation}\label{def_ve_we}
v_\e:=\mathop{\rm{Re}}(e^{-i\Theta_\e(x)} u_\e), \quad  w_\e:= \mathop{\rm{Im}}(e^{-i\Theta_\e(x)} u_\e).
\end{equation}
As observed in Section \ref{sec_preliminaries}, $(v_\e,w_\e) \in \widetilde{\mc{H}}_\e$  solves \eqref{prob_eigenvalue_gauged} in the weak sense \eqref{eq_multipole_gauged_e} with $\lambda=\lambda_\e$.
Furthermore, we choose $u_\e$ in the only possible way such that \eqref{hp_ue_remormalised} holds,  and consequently $(v_\e,w_\e)$ in the only possible way such that
\begin{equation}\label{hp_ve_we}
\int_{\Omega} (v_\e^2+w_\e^2) \, dx =1,\quad  \int_{\Omega} (w_\e v_0-v_\e w_0) \,dx =0 \quad \text{ and } \quad \int_{\Omega} (v_\e v_0+w_\e w_0) \,dx >0,
\end{equation}
see Remark \ref{rermark_u_vw}.

\begin{lemma}\label{remark_spect_stability_eigenfunctions}
If $u_\e$ is chosen as in \eqref{hp_ue_remormalised}, so that $(v_\e,w_\e)$ is normalized as in  \eqref{hp_ve_we}, then 
\begin{equation}\label{eq_limit_vewe_v0wo}
 v_\e \to v_0 \quad \text{and}\quad   w_\e \to w_0 \quad \text{ strongly  in } \mc{H}_1  \text{ as } \e \to 0^+.
\end{equation}
\end{lemma}
\begin{proof}Indeed, testing \eqref{eq_multipole_gauged_e} with $(v_\e,w_\e)$ we obtain  $\norm{(v_\e,w_\e)}^2_{\mc{H}_\e\times \mc{H}_\e}=\la_\e$ 
by \eqref{hp_ve_we}.  It follows that,  thanks to $\eqref{limit_lae_la0}$, $\{(v_\e,w_\e)\}_{\e \in (0,1]}$ is bounded in $\mc{H}_1\times \mc{H}_1$. Hence,  there exist a sequence $\{\e_n\}_{n \in \mb{N}}\subset (0,1)$ and $(v,w) \in \mc{H}_1\times \mc{H}_1$ such that $v_{\e_n} \rightharpoonup v$ and 
$w_{\e_n} \rightharpoonup w$ weakly in  $\mc{H}_1$ as $n \to \infty$. By \cite[Proposition 3.3]{FNOS_multipole}, $v,w \in \mc{H}_0$, while  
$\int_\Omega (v^2+w^2)\,dx=1$ by \cite[Remark 3.1]{FNOS_multipole}  and \eqref{hp_ve_we}. 
We have  that $R^j(v_{\e_n},w_{\e_n})=I^j(v_{\e_n},w_{\e_n})=0$ on $\Gamma_0^j$ for every $j=1,\dots, k$  and  $n \in \mathbb{N}$. Then,  by the continuity of the trace operators in \eqref{def_traces}, $R^j(v,w)=I^j(v,w)=0$ on $\Gamma_0^j$ for every $j=1,\dots, k$, so that  $(v,w) \in \mc{\widetilde{H}}_0$. 
If $ (\varphi,\psi) \in  \mc{\widetilde{H}}_{0,0}$, see \eqref{def_tilde_H_00}, then $(\varphi,\psi)\in \mc{\widetilde{H}}_{\e_n}$ for any $n$ large enough. Hence  we may test \eqref{eq_multipole_gauged_e} with $(\varphi,\psi)$, thus obtaining 
\begin{equation*}
\int_{\Omega\setminus \Gamma_1}(\nabla v_{\e_n}\cdot\nabla \varphi+\nabla w_{\e_n}\cdot \nabla \psi)\, dx 
=\la_{\e_n} \int_{\Omega}( v_{\e_n}\varphi+w_{\e_n}\psi) \, dx,
\end{equation*}
for any sufficiently large $n \in \mathbb{N}$. Passing to the limit as $n \to \infty$, thanks to \eqref{limit_lae_la0} and the fact that $v_{\e_n} \rightharpoonup v$ and 
$w_{\e_n} \rightharpoonup w$ weakly in  $\mc{H}_1$, we conclude that 
\begin{equation*}
\int_{\Omega\setminus \Gamma_1}(\nabla v\cdot\nabla \varphi+\nabla w\cdot \nabla \psi)\, dx 
=\la_{0} \int_{\Omega}( v\varphi+w\psi) \, dx,
\end{equation*}
for every $ (\varphi,\psi) \in  \mc{\widetilde{H}}_{0,0}$. In view of 
the density of $\widetilde{\mc{H}}_{0,0}$ in $\widetilde{\mc{H}}_0$, see
\cite[Lemma 3.4]{FNOS_multipole}, the above identity is actually satisfied by  any $(\varphi,\psi) \in  \mc{\widetilde{H}}_{0}$. Hence  $(v,w)$ is an eigenfunction of \eqref{prob_eigenvalue_gauged} with $\e=0$ associated to the eigenvalue $\la_0$. Since 
the eigenspace of \eqref{prob_eigenvalue_gauged} with $\e=0$ associated to $\lambda_0$ is generated by 
$(v_0,w_0)$ and $(-w_0,v_0)$, there exist $a,b\in \R$ such that $v=av_0-bw_0$ and $w=aw_0+bv_0$. Passing to the limit in the second condition in \eqref{hp_ve_we} yields 
\begin{equation*}
  0=\int_{\Omega} (w v_0-v w_0) \,dx =\int_{\Omega} ((aw_0+bv_0) v_0-(av_0-bw_0) w_0) \,dx
  =b\int_\Omega (v_0^2+w_0^2)\,dx=b,
  \end{equation*}
so that $(v,w)=a(v_0,w_0)$. On the other hand,  since $\int_\Omega (v^2+w^2)\,dx=1$, either $(v,w)=(v_0,w_0)$ or $(v,w)=-(v_0,w_0)$. Since $\lim_{n\to \infty}  \int_{\Omega} (v_{\e_n} v_0+w_{\e_n} w_0) \,dx = 
\int_{\Omega} (v v_0+w w_0) \,dx \geq0$  in view of  \eqref{hp_ve_we}, it must  necessarily be  $(v,w)=(v_0,w_0)$.
Finally by \eqref{limit_lae_la0}
\begin{equation*}
\norm{(v_\e,w_\e)}^2_{\mc{H}_1\times \mc{H}_1}=\la_\e \to \la_0=\norm{(v_0,w_0)}^2_{\mc{H}_1\times \mc{H}_1}, \quad \text{ as } n \to \infty
\end{equation*}
thus $v_{\e_n} \to v$ and $w_{\e_n} \to w$ strongly in  $\mc{H}_1$ as $n \to \infty$. Since the limit is independent from the sequence $\{\e_n\}_{n \in \mathbb{N}}$, by the Urysonh Subsequence Principle, we conclude that  $v_{\e} \to v$ and $w_{\e} \to w$ strongly in  $\mc{H}_1$ as $\e \to 0^+$.
\end{proof}

We denote by $\Pi_\e$ the orthogonal projection onto the 
eigenspace of \eqref{prob_eigenvalue_gauged} associated to the eigenvalue $\lambda_\e$, which is the linear space spanned by $(v_\e,w_\e)$ and $(-w_\e,v_\e)$, i.e. 
\begin{align}\label{def_Pie}
\Pi_\e:\   L^2(\Omega)\times L^2(\Omega) &\to \widetilde{\mc{H}}_\e,\\ 
\notag(\varphi,\psi)&\mapsto 
\bigg(\int_\Omega(\varphi v_\e+\psi w_\e)\,dx\bigg)(v_\e,w_\e)
+\bigg(\int_\Omega(\psi v_\e-\varphi w_\e)\,dx\bigg)(-w_\e,v_\e).
\end{align}
Theorem  \ref{theorem_asymptotic_not_precise} is the first claim of the following result, which is obtained using an approach developed in \cite{AFHL}, see also \cite{FNOS_multipole}.

\begin{theorem}\label{theorem_asymptotic_with_eigenfunction}
Suppose that \eqref{hp_la0_simple} and \eqref{hp_cirulation}  hold. Then 
\begin{equation}\label{eq_asymptotic_eignevalues}
\la_\e-\la_0=2 (\mc{E}_\e+L_\e(v_0,w_0))+o\big(\norm{(V_\e,W_\e)}_{\mc{H}_\e\times \mc{H}_\e}^2\big) \quad \text{ as } \e \to 0^+,
\end{equation}
with $V_\e,W_\e$ being as in Proposition \ref{prop_potential}. Moreover, as $\e\to0^+$,
\begin{align}
&\norm{(v_0,w_0)-(V_\e,W_\e)-\Pi_\e(v_0-V_\e,w_0-W_\e)}_{\mc{H}_\e\times\mc{H}_\e}=o(\norm{(V_\e,W_\e)}_{\mc{H}_\e\times \mc{H}_\e}),
\label{eq_v0w0_VeWe_Pi}\\[5 pt]
&\norm{(v_0,w_0)-\Pi_\e(v_0-V_\e,w_0-W_\e)}_{L^2(\Omega)\times L^2(\Omega)}=o(\norm{(V_\e,W_\e)}_{\mc{H}_\e\times \mc{H}_\e}),
\label{eq_v0w0_Pi}\\[5 pt]
&\norm{\Pi_\e(v_0-V_\e,w_0-W_\e)}_{L^2(\Omega)\times L^2(\Omega)}^2=1+o(\norm{(V_\e,W_\e)}_{\mc{H}_\e\times \mc{H}_\e}).
\label{eq_Pi(v0w0_VeWe)}
\end{align}
\end{theorem}

\begin{proof}
Let $Y_\e:=v_0-V_\e$ and $Z_\e:=w_0-W_\e$. 
From the fact that $(v_0,w_0)\in \widetilde{\mc{H}}_0$ solves \eqref{prob_eigenvalue_gauged} with $\e=0$ and $\lambda=\lambda_0$,  in the sense of \eqref{eq_multipole_gauged_0}, it follows that 
\begin{equation}\label{eq:v_0crack}
\int_{\Omega \setminus \Gamma_\e} (\nabla v_0 \cdot \nabla \varphi +\nabla w_0 \cdot \nabla\psi)\, dx 
= \lambda_0 \int_{\Omega}  (v_0 \varphi+w_0 \psi) \,  dx+L_\e(\varphi,\psi) 
\quad \text{for all } (\varphi,\psi) \in \widetilde{\mc{H}}_\e.
\end{equation}
From \eqref{eq:v_0crack} and Proposition \ref{prop_potential} 
it follows that $(Y_\e,Z_\e)\in \widetilde{\mc{H}}_\e$ is a weak solution to the problem
\begin{equation*}
\begin{cases}
-\Delta Y_\e= \la_0  v_0,  &\text{in } \Omega \setminus \Gamma_\e,\\
-\Delta Z_\e= \la_0  w_0,  &\text{in } \Omega \setminus \Gamma_\e,\\
Y_\e=Z_\e=0, &\text{on } \partial \Omega,\\
R^j(Y_\e,Z_\e)=I^j(Y_\e,Z_\e)=0,&\text{on }\Gamma^j_\e \text{ for all } j=1,\dots,k,\\
R^j(\nabla Y_\e\cdot \nu^j,\nabla Z_\e\cdot \nu^j)=I^j(\nabla Y_\e\cdot \nu^j, \nabla Z_\e\cdot \nu^j)=0, &\text{on }\Gamma^j_\e \text{ for all } j=1,\dots,k,\\
\end{cases}
\end{equation*}
in the sense that, letting $q_\e$ be as in \eqref{def_qe},
\begin{equation}\label{eq_Ye_Ze}
q_\e\big((Y_\e,Z_\e),(\varphi,\psi)\big)=\la_0\int_\Omega(v_0\varphi+w_0\psi)\,dx \quad \text{ for all } (\varphi,\psi) \in \widetilde{\mc{H}}_\e,
\end{equation}
and, equivalently 
\begin{equation}\label{proof:main_theorem_0.5}
q_\e((Y_\e,Z_\e),(\varphi,\psi))-\la_0
\int_\Omega (Y_\e\varphi+Z_\e\psi)\,dx=\la_0
\int_\Omega (V_\e\varphi+W_\e\psi)\,dx
\end{equation}
for every $(\varphi,\psi) \in \widetilde{\mc{H}}_\e$.
Given $(v_\e,w_\e)$ as in \eqref{def_ve_we} and such that \eqref{hp_ve_we} holds, let $\Pi_\e$ be as in \eqref{def_Pie}. 
For $\e>0$ small, let also
\begin{equation}\label{def_hat_ve_we}
(\hat{v}_\e,\hat{w}_\e):=\frac{\Pi_{\e}(Y_\e,Z_\e)}{\norm{\Pi_{\e}(Y_\e,Z_\e)}_{L^2(\Omega)\times L^2(\Omega)}}.
\end{equation}
Since $(v_\e,w_\e)$ and $(-w_\e,v_\e)$ solve \eqref{eq_multipole_gauged_e} with $\lambda=\lambda_\e$, then $(\hat{v}_\e,\hat{w}_\e)$ satisfies \eqref{eq_multipole_gauged_e}.
Then, choosing $(\varphi,\psi)= (\hat{v}_\e,\hat{w}_\e)$ in \eqref{proof:main_theorem_0.5} and using  \eqref{eq_multipole_gauged_e} for $(\hat{v}_\e,\hat{w}_\e)$, we obtain
\begin{align}\label{eq:hat_ve_we}
(\la_\e-\la_0)
\int_\Omega(Y_\e \hat v_\e+Z_\e  \hat w_\e)\,dx
=&\la_0
\int_\Omega(V_\e v_0+W_\e w_0)\,dx\\
\notag &\quad +\la_0
\int_\Omega\big(V_\e (\hat  v_\e-v_0)+W_\e ( \hat w_\e-w_0)\big)\,dx.
\end{align}
We claim that 
\begin{equation}\label{claim_main_theorem}
\la_0\int_{\Omega}( V_\e v_0+W_\e w_0)\,dx =2\mc{E}_\e +2 L_\e(v_0,w_0).
\end{equation}
As a simple consequence of \eqref{claim_main_theorem}, Proposition \ref{prop_norm_L2_norm_nabla} and the Cauchy-Schwartz inequality 
we observe that 
\begin{equation}\label{estimate_no_L}
2\mc{E}_\e +2 L_\e(v_0,w_0)=o(\norm{(V_\e,W_\e)}_{\mc{H}_\e\times \mc{H}_\e}) \quad \text{ as } \e \to 0^+.
\end{equation}
To prove \eqref{claim_main_theorem}, we test \eqref{eq:v_0crack} with $(\varphi,\psi)=(V_\e-v_0,W_\e-w_0)\in \widetilde{\mc{H}}_\e$ and  obtain
\begin{equation}\label{proof:main_theorem_1}
\int_{\Omega \setminus \Gamma_\e} (\nabla v_0 \cdot \nabla V_\e+\nabla w_0\cdot\nabla W_\e) \, dx -\la_0 \int_{\Omega}(v_0 V_\e+w_0W_\e)\, dx
=L_\e(V_\e,W_\e)-L_\e(v_0,w_0).
\end{equation}
 Furthermore we may test \eqref{eq_Ye_Ze} with $(Y_\e,Z_\e)=(v_0-V_\e,w_0-W_\e)$, obtaining 
\begin{equation*}
\int_{\Omega\setminus \Gamma_\e} (|\nabla(V_\e-v_0)|^2 +|\nabla(W_\e-v_0)|^2) \, dx = \la_0 \int_{\Omega} \big( v_0(v_0-V_\e)+w_0(w_0-W_\e) \big) \, dx.
\end{equation*}
The above identity, combined with \eqref{eq_multipole_gauged_0} for $\lambda=\lambda_0$ and $(v,w)=(\varphi,\psi)=(v_0,w_0)\in \widetilde{\mc{H}}_0$, provides
\begin{align}\label{proof:main_theorem_2}
2\int_{\Omega \setminus \Gamma_\e}\ (\nabla v_0 \cdot \nabla V_\e+\nabla w_0\cdot\nabla W_\e) \, dx&= \int_{\Omega\setminus\Gamma_\e}(|\nabla V_\e|^2 +|\nabla W_\e|^2) \, dx\\
\notag&\quad +\la_0\int_{\Omega} (v_0V_\e+w_0W_\e)\, dx.
\end{align}
Claim \eqref{claim_main_theorem} follows from \eqref{def_Je}, \eqref{def_Ee}, \eqref{proof:main_theorem_1} and \eqref{proof:main_theorem_2}.
By assembling \eqref{eq:hat_ve_we} and  \eqref{claim_main_theorem} we obtain
\begin{equation}\label{proof:main_theorem_3}
(\la_\e-\la_0)\!
\int_\Omega (Y_\e  \hat v_\e\!+\!Z_\e  \hat w_\e)\,dx
=2\big(\mc{E}_\e +L_\e(v_0,w_0)\big)+\la_0
\int_\Omega\! \big(V_\e ( \hat v_\e-v_0)\!+\!W_\e ( \hat w_\e-w_0)\big)\,dx.  
\end{equation}
To complete the proof, we study the asymptotics, as $\e \to 0^+$, of each term of \eqref{proof:main_theorem_3}. We divide the remaining part of the proof into several steps.

\smallskip\noindent 
\textbf{Step 1.} We claim that 
\begin{equation}\label{claim_step_1}
|\la_\e-\la_0|=o(\norm{(V_\e,W_\e)}_{\mc{H}_\e \times\mc{H}_\e}) \quad  \text{ as } \e \to 0^+.
\end{equation}
Let $\mu_0:=\la_0^{-1}$ and $\mu_\e:=\la_\e^{-1}$. Since $\la_0$ is simple as an eigenvalue of \eqref{prob_Aharonov-Bohm_multipole} and $\la_\e \to \la_0$, see \eqref{limit_lae_la0}, if $\e$ is sufficiently small we have  $|\mu_\e-\mu_0|=\mathop{\rm{dist}}(\mu_0,\sigma(\mc{F}_\e))$, hence 
\begin{equation}\label{proof:main_theorem_4}
|\la_\e-\la_0|=\la_\e \la_0|\mu_\e-\mu_0| \le 2 \la_0^2 \mathop{\rm{dist}}(\mu_0,\sigma(\mc{F}_\e)) 
\le 2 \la_0^2  \left(\frac{q_\e(\mathcal{F}_\e(Y_\e,Z_\e)-\mu_0 (Y_\e,Z_\e))}{q_\e(Y_\e,Z_\e)}\right)^\frac12,
\end{equation}
in view of Proposition \ref{prop_spectral}. Furthermore, by \eqref{def_qe} and Proposition \ref{prop_VWe_to_0}
\begin{equation}\label{proof:main_theorem_5}
q_\e(Y_\e,Z_\e)= \la_0 +\int_{\Omega \setminus \Gamma_\e} (|\nabla V_\e|^2 +|\nabla W_\e|^2) \, dx
-2 \int_{\Omega \setminus \Gamma_\e} (\nabla V_\e \cdot \nabla v_0 +\nabla W_\e\cdot \nabla w_0) \, dx=\la_0 +o(1)
\end{equation}
as $\e \to 0^+$, since $\norm{(v_0,w_0)}_{L^2(\Omega) \times L^2(\Omega)}=1$. 
By using first \eqref{def_Fe} and then testing \eqref{eq_Ye_Ze} with $\mathcal{F}_\e(Y_\e,Z_\e)-\mu_0 (Y_\e,Z_\e)$, we obtain
\begin{multline*}
q_\e(\mathcal{F}_\e(Y_\e,Z_\e)-\mu_0 (Y_\e,Z_\e))=-\Big((V_\e,W_\e),\mathcal{F}_\e(Y_\e,Z_\e)-\mu_0 (Y_\e,Z_\e)\Big)_{L^2(\Omega) \times L^2(\Omega)}\\
+\Big((v_0,w_0),\mathcal{F}_\e(Y_\e,Z_\e)-\mu_0 (Y_\e,Z_\e)\Big)_{L^2(\Omega) \times L^2(\Omega)}-\mu_0 q_\e\big((Y_\e,Z_\e),\mathcal{F}_\e(Y_\e,Z_\e)-\mu_0 (Y_\e,Z_\e)\big)\\
=-\Big((V_\e,W_\e),\mathcal{F}_\e(Y_\e,Z_\e)-\mu_0 (Y_\e,Z_\e)\Big)_{L^2(\Omega) \times L^2(\Omega)}.
\end{multline*}
Hence, by \eqref{prop:appendix3_2},  Proposition \ref{prop_norm_L2_norm_nabla} and the Cauchy-Schwarz  inequality, we conclude that 
\begin{equation}\label{proof:main_theorem_6}
\big(q_\e(\mathcal{F}_\e(Y_\e,Z_\e)-\mu_0 (Y_\e,Z_\e))\big)^{1/2}=o(\norm{(V_\e,W_\e)}_{\mc{H}_\e \times \mc{H}_\e}) \quad \text{ as } \e \to 0^+.
\end{equation}
Then \eqref{claim_step_1} follows from \eqref{proof:main_theorem_4}, \eqref{proof:main_theorem_5} and \eqref{proof:main_theorem_6}.

\smallskip\noindent 
\textbf{Step 2.} We claim that 
\begin{equation}\label{claim_step_2}
q_\e((Y_\e,Z_\e)-\Pi_\e(Y_\e,Z_\e ))=o(\norm{(V_\e,W_\e)}_{\mc{H}_\e\times\mc{H}_\e}^2) \quad \text{ as } \e \to 0^+.
\end{equation}
Let 
\begin{equation}\label{def_chi_xi_Upsilon_Kappa}
(\chi_\e,\kappa_\e):=(Y_\e,Z_\e)-\Pi_\e(Y_\e,Z_\e),\quad \text{ and } \quad (\gamma_\e,\Upsilon_\e):= \mc{F}_\e(\chi_\e,\kappa_\e)-\mu_\e(\chi_\e,\kappa_\e).
\end{equation}
We have
\begin{align*}
(\chi_\e,\kappa_\e) \in N_\e:&=\Big\{(\varphi,\psi) \in \widetilde{\mc{H}}_\e: \Pi_\e(\varphi,\psi)=0\Big\}\\
&=\bigg\{(\varphi,\psi) \in \widetilde{\mc{H}}_\e: 
\int_\Omega (\varphi v_\e+\psi w_\e)\,dx=\int_\Omega (\psi v_\e-\varphi w_\e)\,dx=0\bigg\}.
\end{align*}
Since both $(v_\e,w_\e)$ and $(-w_\e,v_\e)$ solve  \eqref{prob_eigenvalue_gauged} in the weak sense \eqref{eq_multipole_gauged_e}, from \eqref{def_Fe} we deduce that $\mc{F}_\e(\varphi,\psi)\in N_\e$ for every $(\varphi,\psi) \in N_\e$. Hence the operator 
\begin{equation*}
\widetilde{\mc{F}}_\e:= \mc{F}_\e\Big|_{N_\e}:N_\e\to N_\e
\end{equation*}
is well-defined. It is easy to verify that $\widetilde{\mc{F}}_\e$ satisfies conditions (i)--(iii) of Proposition \ref{prop_spectral}.
Furthermore $\sigma(\widetilde{\mc{F}}_\e)=\sigma(\mc{F}_\e)\setminus \{\mu_\e\}$, so that there exists a constant $K>0$, which is independent of $\e$, such that 
$\big(\mathop{\rm{dist}}(\mu_\e,\sigma(\widetilde{\mc{F}}_\e)\big)\big)^2\ge K$. Then 
\begin{align}\label{proof:main_theorem_7}
q_\e((Y_\e,Z_\e)-\Pi_\e(Y_\e,Z_\e ))&=q_\e((\chi_\e,\kappa_\e))
\le \frac{1}{K}\big(\mathop{\rm{dist}}(\mu_\e,\sigma(\widetilde{\mc{F}}_\e)\big)^2q_\e((\chi_\e,\kappa_\e))\\
&\notag
\le \frac{1}{K}q_\e(\mc{F}_\e(\chi_\e,\kappa_\e)-\mu_\e(\chi_\e,\kappa_\e))=\frac{1}{K}q_\e(\gamma_\e,\Upsilon_\e),
\end{align}
thanks to Proposition \ref{prop_spectral}--(iii) and \eqref{def_chi_xi_Upsilon_Kappa}. In order to estimate $q_\e(\gamma_\e,\Upsilon_\e)$, we test by $ (\gamma_\e,\Upsilon_\e)$ both \eqref{proof:main_theorem_0.5} and \eqref{eq_multipole_gauged_e} satisfied by $\Pi_\e(Y_\e,Z_\e)$, thus obtaining
\begin{multline*}
q_\e((\chi_\e,\kappa_\e),(\gamma_\e,\Upsilon_\e))-\la_\e
\int_\Omega 
(\chi_\e\gamma_\e+\kappa_\e\Upsilon_\e)\,dx\\
=\la_0
\int_\Omega 
(V_\e\gamma_\e+W_\e\Upsilon_\e)\,dx+(\la_0-\la_\e)
\int_\Omega (Y_\e\gamma_\e+Z_\e\Upsilon_\e)\,dx.
\end{multline*}
Then from \eqref{def_Fe} we deduce that
\begin{align*}
q_\e(\gamma_\e,\Upsilon_\e)& =q_\e(\mc{F}_\e(\chi_\e,\kappa_\e),(\gamma_\e,\Upsilon_\e))-\mu_\e q_\e((\chi_\e,\kappa_\e),(\gamma_\e,\Upsilon_\e))\\
&=-\mu_\e\Big(q_\e\big((\chi_\e,\kappa_\e),(\gamma_\e,\Upsilon_\e)\big)-\la_\e q_\e\big(\mc{F}_\e(\chi_\e,\kappa_\e),(\gamma_\e,\Upsilon_\e)\big)\Big)\\
&=-\frac{\la_0}{\la_\e}
\int_\Omega (V_\e\gamma_\e+W_\e\Upsilon_\e)\,dx
-\frac{\la_0-\la_\e}{\la_\e}
\int_\Omega (Y_\e\gamma_\e+Z_\e\Upsilon_\e)\,dx.
\end{align*}
From the Cauchy-Schwarz inequality,  \eqref{prop:appendix3_2}, and  \eqref{limit_lae_la0} it follows that 
\begin{equation}\label{proof:main_theorem_8}
(q_\e(\gamma_\e,\Upsilon_\e))^{\frac{1}{2}}\le C \left(
\norm{(V_\e,W_\e)}_{L^2(\Omega)\times L^2(\Omega)}+|\la_\e-\la_0|\norm{(Y_\e,Z_\e)}_{L^2(\Omega)\times L^2(\Omega)}\right)
\end{equation}
for some constant $C>0$ which does not depend on $\e$. Furthermore, testing \eqref{proof:main_theorem_0.5} with $(Y_\e,Z_\e)$ we obtain, as $\e\to0^+$,
\begin{equation}\label{proof:main_theorem_9}
\int_\Omega(Y_\e^2+Z_\e^2)\,dx-1=-
\int_\Omega (V_\e Y_\e+W_\e Z_\e)\,dx+o(1)=o(1),
\end{equation}
in view of  Proposition \ref{prop_VWe_to_0},  \eqref{proof:main_theorem_5}, and \eqref{prop:appendix3_2}.
Hence we can deduce \eqref{claim_step_2} from \eqref{proof:main_theorem_7}, \eqref{proof:main_theorem_8}, \eqref{proof:main_theorem_9}, \eqref{claim_step_1}, and Proposition \ref{prop_norm_L2_norm_nabla}. We observe that  \eqref{eq_v0w0_VeWe_Pi} is also  proved.

\smallskip\noindent 
\textbf{Step 3.} We claim that 
\begin{equation}\label{claim_step_3}
\norm{(v_0-\hat v_\e,w_0-\hat w_\e)}_{L^2(\Omega)\times L^2(\Omega)}=o(\norm{(V_\e,W_\e)}_{\mc{H}_\e \times \mc{H}_\e}) \quad  \text{ as } \e \to 0^+.
\end{equation}
In view of \eqref{def_hat_ve_we}
\begin{align}\label{proof:main_theorem_10}
(v_0-\hat v_\e,w_0-\hat w_\e)
&=(v_0,w_0)-\frac{\Pi_\e(Y_\e,Z_\e)}{\norm{\Pi_\e(Y_\e,Z_\e)}_{L^2(\Omega)\times L^2(\Omega)}}\\
&=\frac{
(\norm{\Pi_\e(Y_\e,Z_\e)}_{L^2(\Omega)\times L^2(\Omega)}-1)(v_0,w_0)  
+(v_0,w_0)-\Pi_\e(Y_\e,Z_\e)}{\|\Pi_\e(Y_\e,Z_\e)\|_{L^2(\Omega)\times L^2(\Omega)}}.
\end{align}
Furthermore, by definition of $Y_\e,Z_\e$,  Proposition \ref{prop_norm_L2_norm_nabla}, \eqref{prop:appendix3_2}, and \eqref{claim_step_2},
\begin{multline}\label{proof:main_theorem_11}
 \norm{(v_0,w_0)-\Pi_\e(Y_\e,Z_\e)}_{L^2(\Omega)\times L^2(\Omega)} \le \norm{(v_0,w_0)-(Y_\e,Z_\e)}_{L^2(\Omega)\times L^2(\Omega)}\\
 +\norm{(Y_\e,Z_\e)-\Pi_\e(Y_\e,Z_\e)}_{L^2(\Omega)\times L^2(\Omega)}=o(\norm{(V_\e,W_\e)}_{\mc{H}_\e\times \mc{H}_\e}) \quad \text{ as } \e \to 0^+.
\end{multline}
Hence we have proved \eqref{eq_v0w0_Pi}. Finally, since $\norm{(v_0,w_0)}_{L^2(\Omega)\times L^2(\Omega)}=1$, from \eqref{proof:main_theorem_11} and
the Cauchy-Schwarz inequality we deduce that 
\begin{align}
    \label{proof:main_theorem_12}
 \quad \norm{\Pi_\e(Y_\e,Z_\e)}^2_{L^2(\Omega)\times L^2(\Omega)}& =\norm{(v_0,w_0)-\Pi_\e(Y_\e,Z_\e)}^2_{L^2(\Omega)\times L^2(\Omega)}+\|(v_0,w_0)\|^{2}_{L^2(\Omega)\times L^2(\Omega)}\\
 &\notag\qquad\quad-2\big((v_0,w_0)-\Pi_\e(Y_\e,Z_\e),(v_0,w_0)\big)_{L^2(\Omega)\times L^2(\Omega)}\\
 &\notag =1+o(\norm{(V_\e,W_\e)}_{\mc{H}_\e \times \mc{H}_\e}) \quad \text{as } \e \to 0^+,
\end{align}
thus proving \eqref{eq_Pi(v0w0_VeWe)}. Finally \eqref{claim_step_3} follows from \eqref{proof:main_theorem_10}, \eqref{proof:main_theorem_11} and \eqref{proof:main_theorem_12}.

\smallskip\noindent 
\textbf{Step 4.} We claim that 
\begin{equation}\label{claim_step_4}
\int_\Omega(Y_\e \hat v_\e+Z_\e \hat w_\e)\,dx = 1 +o(\norm{(V_\e,W_\e)}_{\mc{H}_\e \times \mc{H}_\e}) \quad  \text{ as } \e \to 0^+.
\end{equation}
Indeed, by \eqref{def_hat_ve_we} we have
\begin{equation*}
\int_\Omega\!(Y_\e \hat v_\e+Z_\e \hat w_\e)dx
\!=\!\frac{\big((Y_\e,Z_\e)\!-\!\Pi_\e(Y_\e,Z_\e),\Pi_\e(Y_\e,Z_\e)\big)_{L^2(\Omega)\times L^2(\Omega)}\!\!+\norm{\Pi_\e(Y_\e,Z_\e)}^2_{L^2(\Omega)\times L^2(\Omega)}}{\norm{\Pi_\e(Y_\e,Z_\e)}_{L^2(\Omega)\times L^2(\Omega)}}.
\end{equation*}
Hence \eqref{claim_step_4} follows from \eqref{claim_step_2}, \eqref{prop:appendix3_2}, and \eqref{proof:main_theorem_12}.

\medskip
Combining  \eqref{estimate_no_L}, \eqref{proof:main_theorem_3}, Proposition \ref{prop_norm_L2_norm_nabla}, \eqref{claim_step_3}  and \eqref{claim_step_4} we obtain 
\begin{multline*}
\la_\e-\la_0=\left(1+o(\norm{(V_\e,W_\e)}_{\mc{H}_\e \times \mc{H}_\e})\right)\left(2\mc{E}_\e +2L_\e(v_0,w_0)+o(\norm{(V_\e,W_\e)}_{\mc{H}_\e \times \mc{H}_\e}^2)\right)\\
=2\mc{E}_\e +2L_\e(v_0,w_0)+ o(\norm{(V_\e,W_\e)}_{\mc{H}_\e \times \mc{H}_\e}^2)\quad  \text{ as } \e \to 0^+.
\end{multline*}
Estimate \eqref{eq_asymptotic_eignevalues} is therefore proven.
\end{proof}

\section{Blow-up Analysis}\label{sec:blowup}
In this section, we perform a blow-up analysis, from which it is possible to extract information on the asymptotic behavior of
$\mc{E}_\e$ and $(V_\e,W_\e)$ as $\e \to 0^+$. 
Assumption  \eqref{hp_cirulation} allows obtaining a Hardy-type inequality, necessary to characterize the concrete functional space where  the limit profile is found.

\subsection{A Hardy type inequality}\label{subsection_hardy_ineq}
Let $\widetilde{\mc{X}}$ and $\widetilde{\mc{H}}$ be as in \eqref{def_tilde_X} and \eqref{def_tilde_H}, respectively. A Hardy-type inequality in  $\widetilde{\mc{X}}$ can be deduced from the following inequality on annuli $D_{2r}\setminus D_r$, for couples of functions in the space
\begin{equation*}
\widetilde{\mc{X}}_r:=\{(\varphi,\psi): \varphi,\psi\in H^1((D_{2r}\setminus D_r) \setminus \Gamma_0),
R^j(\varphi,\psi)=I^j(\varphi,\psi)=0 \text{ on } \Gamma^j_0 \text{ for all } j=1,\dots k\}.
\end{equation*}

\begin{proposition}\label{prop_hardy_annuli}
Under assumption \eqref{hp_cirulation},  for every $r>0$ and $(\varphi,\psi) \in \widetilde{\mc{X}}_r$,
\begin{equation}\label{ineq_hardy_annuli}
\int_{D_{2r}\setminus D_r} \frac{\varphi^2+\psi^2}{|x|^2} \, dx 
\le \max\left\{\frac1{\rho^2},\frac1{(1-\rho)^2}\right\} \int_{(D_{2r}\setminus D_r)\setminus \Gamma_0} (|\nabla \varphi|^2 +|\nabla \psi|^2) \, dx.
\end{equation}
\end{proposition}

\begin{proof}
By a scaling argument, it is enough to prove \eqref{ineq_hardy_annuli} for $r=1$. If $(\varphi,\psi)\in \widetilde{\mc{X}}_1$, then 
\begin{equation*}
u:=e^{i\Theta_0}(\varphi+i\psi)\in H^1(D_2\setminus D_1,\C).
\end{equation*}
Hence, from \cite[Remark 3.2]{FFT} it follows that 
\begin{equation*}
    \int_{D_{2}\setminus D_1} \frac{|u(x)|^2}{|x|^2} \, dx 
\le \max\left\{\frac1{\rho^2},\frac1{(1-\rho)^2}\right\} \int_{D_{2}\setminus D_1} 
|(i\nabla +A^\rho_0)u|^2\,dx,
\end{equation*}
which can be rewritten as
\begin{equation*}
\int_{D_{2}\setminus D_1} \frac{\varphi^2+\psi^2}{|x|^2} \, dx 
\le \max\left\{\frac1{\rho^2},\frac1{(1-\rho)^2}\right\} \int_{(D_{2}\setminus D_1)\setminus \Gamma_0} (|\nabla \varphi|^2 +|\nabla \psi|^2) \, dx.
\end{equation*}
in view of \eqref{eq:gradiente-gauge}, thus proving \eqref{ineq_hardy_annuli} for $r=1$.
\end{proof}

Since the constant in inequality \eqref{ineq_hardy_annuli} does not depend on $r$, we may sum over annuli to fill $\R^2 \setminus D_1$ and obtain the following result.
\begin{proposition}\label{prop_ineq_hardy_R2}
Under assumption \eqref{hp_cirulation}, for every $(\varphi,\psi)\in \widetilde{\mc{X}}$ we have
\begin{equation}\label{ineq_hardy_R2}
\int_{\R^2 \setminus D_1} \frac{\varphi^2+\psi^2}{|x|^2}\, dx 
\le \max\left\{\frac1{\rho^2},\frac1{(1-\rho)^2}\right\} \int_{(\R^2 \setminus D_1)\setminus \Gamma_1} (|\nabla \varphi|^2 +|\nabla \psi|^2) \, dx.
\end{equation}
Moreover, there exists a constant $C_\rho>0$ (depending only on $\rho$) such that, for every $(\varphi,\psi)\in \widetilde{\mc{X}}$,
\begin{equation}\label{ineq_poincarre_D1}
\int_{D_1} \left(\varphi^2 +\psi^2\right) \, dx 
\le C_\rho  \int_{\R^2 \setminus \Gamma_1} (|\nabla \varphi|^2 +|\nabla \psi|^2) \, dx.
\end{equation}
\end{proposition}

\begin{proof}
If $(\varphi,\psi)\in \widetilde{\mc{X}}$ then $(\varphi,\psi)\in \widetilde{\mc{X}}_r$ for every $r> 1$, so that, by
\eqref{ineq_hardy_annuli},
\begin{align*}
\int_{\R^2\setminus D_1} \frac{\varphi^2+\psi^2}{|x|^2} \, dx 
&=\sum_{h=0}^{\infty} \int_{D_{2^{h+1}}\setminus D_{2^h}}\frac{\varphi^2+\psi^2}{|x|^2}\, dx \\
&\le \max\left\{\frac1{\rho^2},\frac1{(1-\rho)^2}\right\}  \sum_{h=0}^{\infty} \int_{(D_{2^{h+1}}\setminus D_{2^h})\setminus \Gamma_0} (|\nabla \varphi|^2 +|\nabla \psi|^2) \, dx\\
&=\max\left\{\frac1{\rho^2},\frac1{(1-\rho)^2}\right\}  \int_{(\R^2\setminus D_1)\setminus\Gamma_1} (|\nabla \varphi|^2 +|\nabla \psi|^2) \, dx.
\end{align*}
 Inequality \eqref{ineq_hardy_R2} is thereby proved.

Integrating the identity $\dive((\varphi^2+\psi^2)x)=2(\varphi\nabla\varphi +\psi\nabla \psi) \cdot x+2 (\varphi^2 +\psi^2)$ 
on each subset of $D_1$ obtained by cutting along the lines $\Sigma^j$ and  observing that  $x\cdot \nu^j=0$ on $\Sigma^j$ for all $j=1,\dots, k$, the Diverge Theorem yields
\begin{equation}\label{eq:sthain1}
\int_{D_1}(\varphi^2 +\psi^2) \, dx \le \int_{\partial D_1} (\varphi^2 +\psi^2)\, dS +   \int_{D_1 \setminus \Gamma_1} (|\nabla\varphi|^2 +|\nabla\psi|^2) \, dx,
\end{equation}
for every $(\varphi,\psi)\in \widetilde{\mc{X}}$.
Since the trace operator $H^1((D_2\setminus D_1)\setminus \Gamma_1) \to L^2(\partial D_1)$ is continuous, 
there exists a constant $C>0$ such that 
\begin{equation}\label{eq:sthain2}
    \int_{\partial D_1} v^2\, dS\leq 
    C\left( \int_{D_2\setminus D_1} v^2\, dx
+ \int_{(D_2\setminus D_1) \setminus \Gamma_1} |\nabla v|^2 \, dx\right)\quad\text{for all }v\in 
H^1((D_2\setminus D_1)\setminus \Gamma_1).
\end{equation}
From \eqref{eq:sthain1}, \eqref{eq:sthain2}, and \eqref{ineq_hardy_R2} we deduce that  
\begin{align*}
\int_{D_1}(\varphi^2 +\psi^2) \, dx& \le C\left( \int_{D_2\setminus D_1} (\varphi^2 +\psi^2)\, dx
+ \int_{(D_2\setminus D_1) \setminus \Gamma_1} (|\nabla\varphi|^2 +|\nabla\psi|^2) \, dx\right) \\
&\qquad+\int_{D_1\setminus \Gamma_1} (|\nabla\varphi|^2 +|\nabla\psi|^2) \, dx\\
&\le 4 C \int_{D_2\setminus D_1}\frac{\varphi^2+\psi^2}{|x|^2}\, dx +(C+1)\int_{D_2\setminus \Gamma_1} (|\nabla\varphi|^2 +|\nabla\psi|^2) \, dx\\
&\le \Big(C+1+4C\max\left\{\rho^{-2},(1-\rho)^{-2}\right\}\Big) \int_{\R^2\setminus \Gamma_1} (|\nabla\varphi|^2 +|\nabla\psi|^2) \, dx,
\end{align*}
thus proving \eqref{ineq_poincarre_D1}.
\end{proof}

As a consequence of Proposition \ref{prop_ineq_hardy_R2}, we have that 
\begin{equation*}\label{norm_widetilde_X}
\norm{(\varphi,\psi)}_{\widetilde{\mc{X}}}=\left(\int_{\R^2 \setminus \Gamma_1} (|\nabla \varphi|^2+|\nabla \psi|^2) \, dx \right)^{\frac{1}{2}}
\end{equation*}
is a norm on $\widetilde{\mc{X}}$. Furthermore, $\widetilde{\mc{X}}$ is a Hilbert space with respect to the scalar product associated to this norm, and the restriction operator 
\begin{equation}\label{def_restr_operator}
\widetilde{\mc{X}} \to H^1(D_r \setminus \Gamma_1)\times H^1(D_r \setminus \Gamma_1)
\end{equation}
is well defined and continuous for all $r>0$.
This, together with the continuity of the trace operators in \eqref{def_traces}   ensures that
\begin{equation}\label{sup_tilde_H_bounded}
\sup_{(\varphi,\psi) \in \widetilde{\mc{X}}}\setminus \{(0,0)\} 
\frac{\|\gamma^j_+(\varphi)\|^2_{L^p(S^j_1)}+\|\gamma^j_+(\psi)\|^2_{L^p(S^j_1)}}{\norm{(\varphi,\psi)}^2_{\widetilde{\mc{X}}}} < +\infty
\end{equation}
for all $p \in [1,+\infty)$ and $j=1,\dots,k$.

\begin{remark}\label{remark_tilde_Hc}
Let
\begin{equation}\label{def_tilde_Hc}
\widetilde{\mc{H}}_c
:=\{(\varphi,\psi) \in \widetilde{\mc{H}}: \text{ there exists } r>0 \text{ such that } \varphi\equiv\psi\equiv 0 \text{ in } \R^2 \setminus D_r\}   
\end{equation}
Proceeding as in \cite[Proposition 6.3]{FNOS_multipole}, one can prove that $\widetilde{\mc{H}}_c$ is dense in $\widetilde{\mc{H}}$.
\end{remark}

\subsection{The asymptotic behaviour of $\mc{E}_\e$}
In this subsection, we prove an equivalent characterization of the quantity  $\mc{E}_\e$  introduced in \eqref{def_Ee} 
and use it to prove an optimal estimate for $|\mc{E}_\e|$
as $\e \to 0^+$, thus refining the estimates  obtained preliminarily in Proposition \ref{prop_Ee_lower_upper_estimates}.

\begin{proposition}\label{prop_Ee_equiv}
Let $\eta_\e$ be as in \eqref{def_eta_r} with $r=\e$. Then, for every $\e \in (0,1]$,
\begin{align}\label{eq_Ee_equiv}
\mc{E_\e}=&\,
\frac{1}{2}\int_{\Omega\setminus \Gamma_0}(|\nabla (\eta_\e v_0)|^2+|\nabla(\eta_\e w_0)|^2) \, dx -L_\e(v_0,w_0)\\
&\notag 
-\frac{1}{2} \sup_{\substack{(\varphi,\psi) \in \widetilde{\mc{H}}_\e\\(\varphi,\psi)\neq(0,0)}}
\frac{\left(\int_{\Omega\setminus \Gamma_\e} (\nabla \varphi \cdot \nabla (\eta_\e v_0)+\nabla\psi \cdot \nabla (\eta_\e w_0)) \, dx- L_\e(\varphi,\psi)\right)^2}
{\int_{\Omega\setminus \Gamma_\e}(|\nabla \varphi|^2+|\nabla \psi|^2) \, dx}.
\end{align}
\end{proposition}

\begin{proof}
Since $(\varphi,\psi)-(v_0,w_0) \in \widetilde{\mc{H}}_\e$ if and only if  $(\varphi,\psi)-\eta_\e(v_0,w_0) \in \widetilde{\mc{H}}_\e$, by \eqref{def_Ee} we have 
\begin{align}\label{proof:prop_Ee_equiv_1}
\mc{E_\e}&=\inf_{(\varphi,\psi) \in \widetilde{\mc{H}}_\e}J_\e\big((\varphi,\psi)+\eta_\e(v_0,w_0)\big) =\inf_{\substack{(\varphi,\psi) \in \widetilde{\mc{H}}_\e\\
(\varphi,\psi)\neq (0,0)}}\left(\inf_{t \in\R}J_\e\big(t(\varphi,\psi)+\eta_\e(v_0,w_0)\big)\right).
\end{align}
Furthermore, thanks to \eqref{def_Je},
\begin{align*}
J_\e\big(t(\varphi,\psi)+\eta_\e(v_0,w_0)\big)=&\,\frac{t^2}{2}\int_{\Omega\setminus \Gamma_\e}(|\nabla \varphi|^2+|\nabla \psi|^2) \, dx\\
&\quad + t\left(\int_{\Omega\setminus \Gamma_\e} (\nabla \varphi \cdot \nabla (\eta_\e v_0)+\nabla\psi \cdot \nabla (\eta_\e w_0)) \, dx - L_\e(\varphi,\psi)\right)\\
&\quad +\frac{1}{2}\int_{\Omega\setminus \Gamma_0}(|\nabla (\eta_\e v_0)|^2+|\nabla(\eta_\e w_0)|^2) \, dx -L_\e(v_0,w_0).
\end{align*}
Hence, for any $(\varphi,\psi) \in \widetilde{\mc{H}}_\e\setminus \{(0,0)\}$,
\begin{align*}
&\inf_{t \in \R}J_\e\big(t(\varphi,\psi)+\eta_\e(v_0,w_0)\big)\\
&\quad =-\frac{1}{2} 
\frac{\left(\int_{\Omega\setminus \Gamma_\e} (\nabla \varphi \cdot \nabla (\eta_\e v_0)+\nabla\psi \cdot \nabla (\eta_\e w_0)) \, dx- L_\e(\varphi,\psi)\right)^2}
{\int_{\Omega\setminus \Gamma_\e}(|\nabla \varphi|^2+|\nabla \psi|^2) \, dx}\\
&\quad \quad +\frac{1}{2}\int_{\Omega\setminus \Gamma_0}(|\nabla (\eta_\e v_0)|^2+|\nabla(\eta_\e w_0)|^2) \, dx -L_\e(v_0,w_0),
\end{align*}
so that \eqref{eq_Ee_equiv} follows from \eqref{proof:prop_Ee_equiv_1}.
\end{proof}

As a consequence of \eqref{eq_Ee_equiv} we obtain the following improvement of the estimate on  $\mc{E}_\e$ obtained initially in Proposition \ref{prop_Ee_lower_upper_estimates}.
\begin{proposition}\label{prop_asymptotic_Ee}
Let $m \in \mathbb{Z}$ be as in Proposition \ref{prop_vw_asympotic_not_1/2}  with $(v,w)=(v_0,w_0)$ if $\rho\neq\frac12$, or as in  Proposition \ref{prop_vw_asympotic_1/2} if $\rho=\frac12$. Then 
\begin{equation}\label{eq_asymptotic_Ee_precise}
\mc{E}_\e= O(\e^{2|m+\rho|}) \quad \text{ as } \e \to 0^+.
\end{equation}
\end{proposition}

\begin{proof}
Thanks to Proposition \ref{prop_Ee_equiv} and the Cauchy-Schwartz  inequality we have 
\begin{align}\label{proof:prop_asymptotic_Ee_1}
|\mc{E}_\e| \le &\sup_{\substack{(\varphi,\psi) \in \widetilde{\mc{H}}_\e\\(\varphi,\psi)\neq(0,0)}} \frac{ |L_\e(\varphi,\psi)|^2}
{\int_{\Omega\setminus \Gamma_\e}(|\nabla \varphi|^2+|\nabla \psi|^2) \, dx}\\
&\notag \qquad+\frac{3}{2}\int_{\Omega\setminus \Gamma_0}(|\nabla (\eta_\e v_0)|^2+|\nabla(\eta_\e w_0)|^2) \, dx +|L_\e(v_0,w_0)|.
\end{align}    
In view of \eqref{def_Le},  Proposition \ref{prop_vw_asympotic_not_1/2} and Proposition \ref{prop_vw_asympotic_1/2}
\begin{multline}\label{proof:prop_asymptotic_Ee_2}
\int_{\Omega\setminus \Gamma_0}(|\nabla (\eta_\e v_0)|^2+|\nabla(\eta_\e w_0)|^2) \, dx \le 
2\int_{D_{2\e}}|\nabla \eta_\e|^2( v_0^2+ w_0^2) \, dx\\
+2\int_{D_{2\e}\setminus \Gamma_0}\eta_\e(|\nabla v_0|^2+|\nabla w_0|^2) \, dx=O(\e^{2|m+\rho|}) \quad  \text{as } \e \to 0^+,
\end{multline}
and 
\begin{equation}\label{proof:prop_asymptotic_Ee_3}
|L_\e(v_0,w_0)|=O(\e^{2|m+\rho|}) \quad  \text{as } \e \to 0^+.
\end{equation}
It only remains to estimate the first term in the right hand side of  \eqref{proof:prop_asymptotic_Ee_1}.
To this aim we notice that, 
fixing   any 
$p \in \big(1,
\min\big\{\frac1\rho,\frac1{1-\rho}\big\}\big)$ and letting $p'=\frac{p}{p-1}$,
by a change of variables,  \eqref{sup_tilde_H_bounded},
and the fact that $\widetilde{\mc{H}} \subset \widetilde{\mc{X}}$,
\begin{equation}\label{proof:prop_asymptotic_Ee_4}
\sup_{\substack{(\varphi,\psi) \in \widetilde{\mc{H}}_\e\\(\varphi,\psi)\neq(0,0)}} 
\frac{\|\gamma^j_+(\varphi)\|^2_{L^{p'}(S_\e^j)}}
{\int_{\Omega\setminus \Gamma_\e}(|\nabla \varphi|^2+|\nabla \psi|^2) \, dx}
\le \e^{{2}/{p'}}\hskip-10pt \sup_{
\substack{(\varphi,\psi) \in \widetilde{\mc{H}}\\(\varphi,\psi)\neq(0,0)}} 
\!\!\!\frac{\|\gamma^j_+(\varphi)\|^2_{L^{p'}(S_1^j)}}{\norm{(\varphi,\psi)}^2_{\widetilde{\mc{X}}}}=O\big(\e^{2/p'}\big)
\end{equation}
as $\e \to 0^+$. 
From the H\"older inequality and Propositions \ref{prop_vw_asympotic_not_1/2} and \ref{prop_vw_asympotic_1/2} it follows that
\begin{multline*}
\sup_{\substack{(\varphi,\psi) \in \widetilde{\mc{H}}_\e\\(\varphi,\psi)\neq(0,0)}} \frac{ \left(\int_{S_\e^j}|\nabla v_0| |\gamma^j_+(\varphi)|\, dS\right)^2 }
{\int_{\Omega\setminus \Gamma_\e}(|\nabla \varphi|^2+|\nabla \psi|^2) \, dx}  
\le \left(\int_{S_\e^j}|\nabla v_0|^p\, dS\right)^{\!\!\frac{2}{p}}
\hskip-7pt\sup_{\substack{(\varphi,\psi) \in \widetilde{\mc{H}}_\e\\(\varphi,\psi)\neq(0,0)}}
\frac{ \|\gamma^j_+(\varphi)\|^2_{L^{p'}(S_\e^j)} }
{\int_{\Omega\setminus \Gamma_\e}(|\nabla \varphi|^2+|\nabla \psi|^2) \, dx}\\
=O\big(\e^{2|m+\rho|-2+\frac{2}{p}}\big)\hskip-3pt\sup_{\substack{(\varphi,\psi) \in \widetilde{\mc{H}}_\e\\(\varphi,\psi)\neq(0,0)}} 
\frac{ \|\gamma^j_+(\varphi)\|^2_{L^{p'}(S_\e^j)} }
{\int_{\Omega\setminus \Gamma_\e}(|\nabla \varphi|^2+|\nabla \psi|^2) \, dx} = O(\e^{2|m+\rho|})
\end{multline*}
as $\e \to 0^+$, where we used \eqref{proof:prop_asymptotic_Ee_4} in the last estimate.
Similarly, we can prove that 
\begin{align*}
&\sup_{\substack{(\varphi,\psi) \in \widetilde{\mc{H}}_\e\\(\varphi,\psi)\neq(0,0)}}
\frac{ \left(\int_{S_\e^j}|\nabla v_0| |\gamma^j_+(\psi)|\, dS\right)^2 }
{\int_{\Omega\setminus \Gamma_\e}(|\nabla \varphi|^2+|\nabla \psi|^2) \, dx}= O(\e^{2|m+\rho|}) \quad \text{ as } \e \to 0^+, \\
& \sup_{\substack{(\varphi,\psi) \in \widetilde{\mc{H}}_\e\\(\varphi,\psi)\neq(0,0)}}
\frac{ \left(\int_{S_\e^j}|\nabla w_0| |\gamma^j_+(\varphi)|\, dS\right)^2 }
{\int_{\Omega\setminus \Gamma_\e}(|\nabla \varphi|^2+|\nabla \psi|^2) \, dx}= O(\e^{2|m+\rho|}) \quad \text{ as } \e \to 0^+, \\ 
&\sup_{\substack{(\varphi,\psi) \in \widetilde{\mc{H}}_\e\\(\varphi,\psi)\neq(0,0)}}
\frac{ \left(\int_{S_\e^j}|\nabla w_0| |\gamma^j_+(\psi)|\, dS\right)^2 }
{\int_{\Omega\setminus \Gamma_\e}(|\nabla \varphi|^2+|\nabla \psi|^2) \, dx}= O(\e^{2|m+\rho|}) \quad \text{ as } \e \to 0^+. 
\end{align*}
By the definition of $L_\e$ in \eqref{def_Le} we conclude that
\begin{equation}\label{proof:prop_asymptotic_Ee_5}
\sup_{\substack{(\varphi,\psi) \in \widetilde{\mc{H}}_\e\\(\varphi,\psi)\neq(0,0)}} \frac{ |L_\e(\varphi,\psi)|^2}
{\int_{\Omega\setminus \Gamma_\e}(|\nabla \varphi|^2+|\nabla \psi|^2) \, dx}= O(\e^{2|m+\rho|}) \quad \text{ as } \e \to 0^+.
\end{equation}
By combining \eqref{proof:prop_asymptotic_Ee_1}, \eqref{proof:prop_asymptotic_Ee_2}, \eqref{proof:prop_asymptotic_Ee_3},  and \eqref{proof:prop_asymptotic_Ee_5},
we finally obtain \eqref{eq_asymptotic_Ee_precise}.
\end{proof}

\subsection{The blow-up analysis}\label{subsec_blow_up}

To perform a blow-up analysis that provides sharp information about the asymptotic behavior of $\mathcal E_\e$ as $\e\to0^+$, a crucial initial step lies in identifying the profile that emerges as the limit of an appropriate scaling of the family $\{(V_\e,W_\e)\}_{\e}$. 
Such a limit profile turns out to be precisely  the solution to the minimization problem \eqref{min_prob_tilde}, whose existence and uniqueness are proved in the following proposition.

\begin{proposition}\label{prop_existence_tilde_VW}
There exists a unique solution $(\widetilde{V},\widetilde{W}) \in \widetilde{\mc{X}}$ to the minimization problem \eqref{min_prob_tilde}. 
Furthermore  $(\widetilde{V},\widetilde{W})$  satisfies
\begin{equation}\label{eq_potetial_tilde}
\begin{cases}
(\widetilde{V},\widetilde{W})-\eta (\Phi_0,\Psi_0) \in \widetilde{\mc{H}},\\[5pt]
\int_{\R^2\setminus \Gamma_1}(\nabla \widetilde{V} \cdot \nabla \varphi +\nabla \widetilde{W} \cdot \nabla \psi)\, dx=L(\varphi,\psi)\quad \text{for every} \quad 
(\varphi,\psi) \in \widetilde{\mc{H}},
\end{cases}
\end{equation}
with  $L$ and   $\eta$ being as in \eqref{def_L} and \eqref{def_eta}, respectively.
\end{proposition}

\begin{proof}
In view of  \eqref{def_Psi_rho} and \eqref{def_Phi_rho},  $|\nabla \Phi_0|,|\nabla \Psi_0|\in L^p(S_1^j)$ for every
$p \in \big[1,
\min\big\{\frac1\rho,\frac1{1-\rho}\big\}\big)$ and $j=1,\dots, k$. Hence the linear functional $L$ in \eqref{def_L} is well-defined and continuous, thanks to the continuity of the trace operators in \eqref{def_traces}.
In particular, the functional $J$ defined in \eqref{def_J} is continuous, convex and coercive on the closed, convex set
$\{(\varphi,\psi) \in \widetilde{\mc{X}}:(\varphi,\psi)-\eta (\Phi_0,\Psi_0) \in \widetilde{\mc{H}}\}$.
Therefore \eqref{min_prob_tilde} admits a minimizer $(\widetilde{V},\widetilde{W}) \in \widetilde{\mc{X}}$ that solves \eqref{eq_potetial_tilde}.

To prove uniqueness, we assume that  $(\widetilde{V}_1,\widetilde{W}_1)$ and $(\widetilde{V}_2,\widetilde{W}_2)$ are both solutions of \eqref{eq_potetial_tilde}. Then 
$(\widetilde{V}_1-\widetilde{V}_2,\widetilde{W}_1-\widetilde{W}_2) \in \widetilde{\mc{H}}$, and hence we may test the difference between \eqref{eq_potetial_tilde} for $(\widetilde{V}_1,\widetilde{W}_1)$  and \eqref{eq_potetial_tilde} for  $(\widetilde{V}_2,\widetilde{W}_2)$ with 
$(\widetilde{V}_1-\widetilde{V}_2,\widetilde{W}_1-\widetilde{W}_2)$. 
It follows that $\widetilde{V}_1=\widetilde{V}_2$ and $\widetilde{W}_1=\widetilde{W}_2$ by \eqref{ineq_hardy_R2}.
\end{proof}

The next step consists in considering a scaling of the functions $V_\e,W_\e$ with a factor determined by  the optimal estimate on $\mathcal E_\e$ obtained in Proposition \ref{prop_asymptotic_Ee}. To this aim,
let $m \in \mb{\Z}$ be as in Proposition \ref{prop_vw_asympotic_not_1/2} or Proposition \ref{prop_vw_asympotic_1/2} for $(v_0,w_0)$. For every $\e\in(0,1]$, letting 
$(V_\e,W_\e)$ be as in Proposition \ref{prop_potential}, we  define 
\begin{align}
& \widetilde{V}_\e (x):= \e^{-|m+\rho|}V_\e(\e x), \quad  \widetilde{W}_\e (x):= \e^{-|m+\rho|}W_\e(\e x), \label{def_tilde_VW_e}\\
& \widetilde{V}_{0,\e} (x):= \e^{-|m+\rho|}v_0(\e x),\quad \widetilde{W}_{0,\e} (x):= \e^{-|m+\rho|}w_0(\e x). \label{def_tilde_VW_0}
\end{align}
We  still denote by $\widetilde{V}_\e,\widetilde{W}_\e,\widetilde{V}_{0,\e},\widetilde{W}_{0,\e}$ 
their respective  trivial extensions in $\R^2\setminus \frac1\e\Omega$.
Then 
\begin{equation}\label{eq_tilde_VeWe_in_H}
(\widetilde{V}_\e,\widetilde{W}_\e),(\widetilde{V}_{0,\e},\widetilde{W}_{0,\e})\in  \widetilde{\mc{X}} \quad \text{ and } \quad 
(\widetilde{V}_\e-\widetilde{V}_{0,\e},\widetilde{W}_\e-\widetilde{W}_{0,\e}) \in  \widetilde{\mc{H}}.
\end{equation}  
By a change of variables and \eqref{eq_potential},  for every $(\varphi,\psi) \in \widetilde{\mc{H}}$ such that $\varphi\equiv\psi\equiv0$ in $\R^2\setminus \frac1\e\Omega$
we have
\begin{multline}\label{eq_tilde_VeWe}
\int_{\R^2\setminus \Gamma_1}\! (\nabla \widetilde{V}_\e \!\cdot\! \nabla \varphi +\nabla \widetilde{W}_\e \!\cdot \!\nabla \psi) \, dx
=\sum_{j=1}^{k}(b_j-1)\!\int_{S^j_1}[ \nabla \widetilde{V}_{0,\e} \cdot \nu^j	\gamma_+^j (\varphi)+ \nabla \widetilde{W}_{0,\e} \cdot \nu^j\gamma_+^j (\psi)] \, dS \\
-\sum_{j=1}^{k}d_j\int_{S^j_1}[ \nabla \widetilde{V}_{0,\e} \cdot \nu^j\gamma_+^j (\psi)- \nabla \widetilde{W}_{0,\e}\cdot \nu^j	\gamma_+^j (\varphi)]\, dS.
\end{multline}
In particular, \eqref{eq_tilde_VeWe} holds for every  $(\varphi,\psi) \in \widetilde{\mc{H}}_c$ (see \eqref{def_tilde_Hc}) provided $\e$ is sufficiently small. 
Furthermore, letting $\Phi_0$ and $\Psi_0$ be as in \eqref{def_Psi_rho} and \eqref{def_Phi_rho} respectively,  Proposition \ref{prop_vw_asympotic_not_1/2} in case $\rho\neq\frac12$, or Proposition \ref{prop_vw_asympotic_1/2} in case $\rho=\frac12$, imply that 
\begin{equation*}
\nabla \widetilde{V}_{0,\e}(x) \cdot \nu^j \to \nabla\Phi_0(x) \cdot \nu^j\quad \text{and} \quad 
\nabla \widetilde{W}_{0,\e}(x) \cdot \nu^j \to \nabla\Psi_0(x) \cdot \nu^j\quad\text{as }\e\to0^+
\end{equation*}
for every $x \in S_1^j$ and $j=1,\dots,k$. 
On the other hand, Proposition \ref{prop_vw_asympotic_not_1/2} and Proposition \ref{prop_vw_asympotic_1/2} imply  
\begin{equation*}
|\nabla \widetilde{V}_{0,\e}| \le C |x|^{|m+\rho|-1} \quad \text{and} \quad 
|\nabla \widetilde{W}_{0,\e}| \le C|x|^{|m+\rho|-1}\quad  \text{in } \R^2 \setminus \Gamma_0.
\end{equation*}
In particular, for every $j=1,\dots, k$ and  $p \in \big(1,
\min\big\{\frac1\rho,\frac1{1-\rho}\big\}\big)$,
\begin{equation*}
\nabla \widetilde{V}_{0,\e} \cdot \nu^j, \  \nabla \widetilde{W}_{0,\e} \cdot \nu^j  \in   L^p(S_1^j).
\end{equation*}
By the Dominated Convergence Theorem, we conclude that 
\begin{equation}\label{limit_nabla_tilde_V0W0_Lp}
\nabla \widetilde{V}_{0,\e} \cdot \nu^j \to \nabla\Phi_0 \cdot \nu^j  \quad \text{and} \quad 
\nabla \widetilde{W}_{0,\e} \cdot \nu^j \to \nabla\Psi_0 \cdot \nu^j \text{ in } L^p(S_1^j)
\end{equation}
as $\e \to 0^+$, for all $p \in \big(1,
\min\big\{\frac1\rho,\frac1{1-\rho}\big\}\big)$ and $j=1,\dots,k$.
Finally, in view of Propositions \ref{prop_vw_asympotic_not_1/2} and \ref{prop_vw_asympotic_1/2}, we know that, for any $r>0$,
\begin{equation}\label{limit_tilde_V0W0_H1}
\widetilde{V}_{0,\e} \to  \Phi_0  \quad \text{and} \quad 
\widetilde{W}_{0,\e}  \to  \Psi_0 \quad \text{in } H^1(D_r\setminus\Gamma_0),\quad\text{as }\e\to0^+.
\end{equation}

The following blow-up result ensures the convergence of the scaled family $\{(\widetilde{V}_{\e},\widetilde{W}_{\e})\}_\e$, defined in \eqref{def_tilde_VW_e}, to the non trivial profile $(\widetilde{V},\widetilde{W})$, introduced in Proposition \ref{prop_existence_tilde_VW}.
\begin{proposition}\label{prop_blow_up}
Let $m \in \mathbb{Z}$ be as in Proposition \ref{prop_vw_asympotic_not_1/2}  with $(v,w)=(v_0,w_0)$ if $\rho\neq\frac12$, or as in  Proposition \ref{prop_vw_asympotic_1/2} if $\rho=\frac12$. For every $\e \in ( 0,1]$, we consider the pairs 
$(V_\e,W_\e)$ as given in Proposition \ref{prop_potential},  and $(\widetilde{V}_\e,\widetilde{W}_\e)$ as defined in \eqref{def_tilde_VW_e}. Then 
\begin{equation}\label{limit_VWe}
(\widetilde{V}_\e,\widetilde{W}_\e) \to (\widetilde{V},\widetilde{W}) \quad \text{ strongly in } \widetilde{\mc{X}} \text{ as } \e \to 0^+,
\end{equation}
where $(\widetilde{V},\widetilde{W})$ is as in Proposition \ref{prop_existence_tilde_VW}.
\end{proposition}

\begin{proof}
By \eqref{def_tilde_VW_e}, a change  of variables, \eqref{def_Ee}, \eqref{def_Le}, Proposition \ref{prop_asymptotic_Ee}, 
Proposition \ref{prop_vw_asympotic_not_1/2}, Proposition \ref{prop_vw_asympotic_1/2},  the H\"older inequality, 
and \eqref{sup_tilde_H_bounded},
we have 
\begin{align}\label{proof:prop_blow_up_0.5}
\|(\widetilde{V}_\e,\widetilde{W}_\e)\|_{\widetilde{\mc{X}}}^2&=\int_{\R^2 \setminus \Gamma_1} (|\nabla \widetilde{V}_\e |^2+|\nabla \widetilde{W}_\e|^2) \, dx
=\e^{ -2|m+\rho|}\norm{(V_\e,W_\e)}_{\mc{H}_\e \times\mc{H}_\e}^2\\
\notag &=2\e^{ -2|m+\rho|}\left(L_\e(V_\e,W_\e)+\mc{E}_\e\right)\\
\notag &\le O(1)+4\e^{-2|m+\rho|}\sum_{j=1}^{k}\int_{S^j_\e}( |\nabla v_0|	|\gamma_+^j (V_\e)|+ |\nabla w_0| |\gamma_+^j (W_\e)|) \,  dS \\
\notag &\quad\quad  +2\e^{-2|m+\rho|}\sum_{j=1}^{k}\int_{S^j_\e}( |\nabla v_0||\gamma_+^j (W_\e)|+ |\nabla w_0| |\gamma_+^j (V_\e)|)\, dS \\
\notag &\le O(1)+O(1)\sum_{j=1}^{k}\int_{S^j_1}|x|^{|m+\rho|-1}	(|\gamma_+^j (\widetilde V_\e)|+|\gamma_+^j (\widetilde{W}_\e)|)\,  dS\\
\notag &\le O(1)+O(1)\|(\widetilde{V}_\e,\widetilde{W}_\e)\|_{\widetilde{\mc{X}}} \quad \text{ as } \e \to 0^+.
\end{align}   
We conclude that $\{(\widetilde{V}_\e,\widetilde{W}_\e)\}_{\e \in (0,1]}$ is bounded in $\widetilde{\mc{X}}$. Hence, for any sequence $\e_n \to 0^+$, there exist a subsequence, still denoted by $\{\e_n\}_{n \in  \mathbb{N}}$, and $(V,W) \in \widetilde{\mc{X}}$ such that 
$(\widetilde{V}_\e,\widetilde{W}_\e) \rightharpoonup (V,W)$ weakly in  $\widetilde{\mc{X}}$ as $n \to \infty$.
Since $(V-\eta \Phi_0,W-\eta \Psi_0) \in \widetilde{\mc{H}}$ by  \eqref{eq_tilde_VeWe_in_H}
and \eqref{limit_tilde_V0W0_H1}, we deduce from
\eqref{eq_tilde_VeWe}, \eqref{limit_tilde_V0W0_H1}, and the density of $\widetilde{\mc{H}}_c$ in $\widetilde{\mc{H}}$ (see Remark \ref{remark_tilde_Hc}) that $(V,W)$ solves \eqref{eq_potetial_tilde}.
Hence, by  uniqueness of the solution of \eqref{eq_potetial_tilde} proved in Proposition \ref{prop_existence_tilde_VW}, we conclude that $(V,W)=(\widetilde{V},\widetilde{W})$.

Furthermore, since $(\widetilde{V}-\eta \Phi_0,\widetilde{W}-\eta \Psi_0) \in \widetilde{\mc{H}}$, we may test \eqref{eq_potetial_tilde} with 
$(\widetilde{V}-\eta \Phi_0,\widetilde{W}-\eta \Psi_0) \in \widetilde{\mc{H}}$, thus obtaining
\begin{multline}\label{proof:prop_blow_up_1}
\int_{\R^2\setminus \Gamma_1}  (|\nabla \widetilde{V}|^2+|\nabla \widetilde{W}|^2) \, dx
=\int_{\R^2\setminus \Gamma_1} [\nabla \widetilde{V}\cdot \nabla (\eta \Phi_0) +\nabla \widetilde{W}\cdot \nabla (\eta \Psi_0)]\, dx\\
+\sum_{j=1}^{k}(b_j-1)\int_{S^j_1}[ \nabla\Phi_0 \cdot \nu^j	\gamma_+^j (\widetilde{V}-\eta \Phi_0)
+ \nabla\Psi_0 \cdot \nu^j\gamma_+^j (\widetilde{W}-\eta \Psi_0)] \, dS \\
-\sum_{j=1}^{k}d_j\int_{S^j_1}[ \nabla\Phi_0 \cdot \nu^j\gamma_+^j (\widetilde{W}-\eta \Psi_0)
- \nabla\Psi_0 \cdot \nu^j	\gamma_+^j (\widetilde{V}-\eta \Phi_0)]\, dS.
\end{multline}
On the other hand, testing \eqref{eq_tilde_VeWe} with $(\widetilde{V}_{\e_n}-\eta \widetilde{V}_{0,\e_n},\widetilde{W}_{\e_n}-\eta \widetilde{W}_{0,\e_n})$ yields
\begin{multline}\label{proof:prop_blow_up_2}
\int_{\R^2\setminus \Gamma_1}  (|\nabla \widetilde{V}_{\e_n}|^2+|\nabla \widetilde{W}_{\e_n}|^2) \, dx
=\int_{\R^2\setminus \Gamma_1} [\nabla \widetilde{V}_{\e_n}\cdot \nabla (\eta \widetilde{V}_{0,\e_n})
+\nabla \widetilde{W}_{\e_n}\cdot \nabla (\eta \widetilde{W}_{0,\e_n})]\, dx\\
+\sum_{j=1}^{k}(b_j-1)\int_{S^j_1}[ \nabla \widetilde{V}_{0,\e_n}\cdot \nu^j	\gamma_+^j (\widetilde{V}_{\e_n}-\eta \widetilde{V}_{0,\e_n})
+ \nabla  \widetilde{W}_{0,\e_n}\cdot \nu^j\gamma_+^j (\widetilde{W}_{\e_n}-\eta \widetilde{W}_{0,\e_n})] \, dS \\
-\sum_{j=1}^{k}d_j\int_{S^j_1}[ \nabla \widetilde{V}_{0,\e_n} \cdot \nu^j\gamma_+^j (\widetilde{W}_{\e_n}-\eta \widetilde{W}_{0,\e_n})
- \nabla \widetilde{W}_{0,\e_n}\cdot \nu^j	\gamma_+^j (\widetilde{V}_{\e_n}-\eta \widetilde{V}_{0,\e_n})]\, dS.
\end{multline}
Taking into account \eqref{proof:prop_blow_up_1}, \eqref{proof:prop_blow_up_2}, 
\eqref{limit_nabla_tilde_V0W0_Lp}, \eqref{limit_tilde_V0W0_H1},
the continuity of the trace operators in \eqref{def_traces}, and the weak convergence of $(\widetilde{V}_{\e_n},\widetilde{W}_{\e_n})$ to $(\widetilde{V},\widetilde{W})$ in $\widetilde{\mathcal X}$, we conclude that 
\begin{equation*}
\int_{\R^2\setminus \Gamma_1}  (|\nabla \widetilde{V}|^2+|\nabla \widetilde{W}|^2) \, dx
=\lim_{n \to \infty}\int_{\R^2\setminus \Gamma_1}  (|\nabla \widetilde{V}_{\e_n}|^2+|\nabla \widetilde{W}_{\e_n}|^2) \, dx.
\end{equation*}
This proves \eqref{limit_VWe}, taking into account  the weak convergence  $(\widetilde{V}_{\e_n},\widetilde{W}_{\e_n})\rightharpoonup(\widetilde{V},\widetilde{W})$ in $\widetilde{\mathcal X}$ and the Urysohn Subsequence Principle.
\end{proof}

With Proposition \ref{prop_blow_up} established, we are now ready to prove Theorem \ref{theorem_asymptotic_precise}.
\begin{proof}[Proof of Theorem \ref{theorem_asymptotic_precise}]
 By \eqref{def_Le}, \eqref{def_Ee}, \eqref{def_tilde_VW_e}, \eqref{def_tilde_VW_0}, and a change of variables, we have  
\begin{align}\label{proof:theorem_asymptotic_precise_1}
\e^{-2 |m+\rho|} \mc{E}_\e
&=\frac{1}{2}\int_{\R^2\setminus \Gamma_1} (|\nabla \widetilde{V}_\e|^2+ |\nabla \widetilde{W}_\e|^2) \, dx\\
&\notag \quad -\sum_{j=1}^{k}(b_j-1)\int_{S^j_1}[ \nabla \widetilde{V}_{0,\e} \cdot \nu^j	\gamma_+^j (\widetilde{V}_\e)
+ \nabla \widetilde{W}_{0,\e} \cdot \nu^j\gamma_+^j (\widetilde{W}_\e)] \, dS \\
&\notag \quad+\sum_{j=1}^{k}d_j\int_{S^j_1}[ \nabla \widetilde{V}_{0,\e} \cdot \nu^j\gamma_+^j (\widetilde{W}_\e)
-\nabla \widetilde{W}_{0,\e}\cdot \nu^j\gamma_+^j (\widetilde{V}_\e)]\, dS.
\end{align}
In view of \eqref{limit_nabla_tilde_V0W0_Lp}, \eqref{limit_VWe}, and  the continuity of the trace operators in \eqref{def_traces}, we may pass to the limit in \eqref{proof:theorem_asymptotic_precise_1}, as $\e \to 0^+$, and conclude that 
\begin{align}\label{proof:theorem_asymptotic_precise_2}
\lim_{\e \to 0^+}\e^{-2 |m+\rho|} \mc{E}_\e
&=\frac{1}{2}\int_{\R^2\setminus \Gamma_1} (|\nabla \widetilde{V}|^2+ |\nabla \widetilde{W}|^2) \, dx\\
&\notag \quad-\sum_{j=1}^{k}(b_j-1)\int_{S^j_1}[ \nabla\Phi_0 \cdot \nu^j\gamma_+^j (\widetilde{V})
+ \nabla\Psi_0 \cdot \nu^j\gamma_+^j (\widetilde{W})] \, dS \\
&\notag \quad+\sum_{j=1}^{k}d_j\int_{S^j_1}[ \nabla\Phi_0 \cdot \nu^j\gamma_+^j (\widetilde{W})
-\nabla\Psi_0\cdot \nu^j\gamma_+^j (\widetilde{V})]\, dS=\mc{E}.
\end{align}
Hence we have proved (i). Furthermore, by \eqref{def_tilde_VW_0}, \eqref{limit_nabla_tilde_V0W0_Lp} and \eqref{limit_tilde_V0W0_H1},
\begin{align}
    \label{proof:theorem_asymptotic_precise_3}
\e^{-2 |m+\rho|} L_\e(v_0,w_0)
&=\sum_{j=1}^{k}(b_j-1)\int_{S^j_1}[ \nabla \widetilde{V}_{0,\e} \cdot \nu^j	\gamma_+^j (\widetilde{V}_{\e,0})
+ \nabla \widetilde{W}_{0,\e} \cdot \nu^j\gamma_+^j (\widetilde{W}_{\e,0})] \, dS \\
&\notag \quad -\sum_{j=1}^{k}d_j\int_{S^j_1}[ \nabla \widetilde{V}_{0,\e} \cdot \nu^j\gamma_+^j (\widetilde{W}_{\e,0})
-\nabla \widetilde{W}_{0,\e}\cdot \nu^j\gamma_+^j (\widetilde{V}_{\e,0})]\, dS\\
&\notag=\sum_{j=1}^{k}(b_j-1)\int_{S^j_1}[ \nabla\Phi_0 \cdot \nu^j	\gamma_+^j (\Phi_0)
+ \nabla\Psi_0 \cdot \nu^j\gamma_+^j (\Psi_0)] \, dS \\
&\notag\quad -\sum_{j=1}^{k}d_j\int_{S^j_1}[ \nabla\Phi_0 \cdot \nu^j \gamma_+^j(\Psi_0)
-\nabla\Psi_0\cdot \nu^j\gamma_+^j (\Phi_0)]\, dS+o(1)\\
&\notag=L(\Phi_0,\Psi_0)+o(1), \text{ as } \e \to 0^+.
\end{align}
Finally claim (ii) follows from \eqref{eq_asymptotioc_eigenvlaues_not_precise}, \eqref{proof:theorem_asymptotic_precise_2}, \eqref{proof:theorem_asymptotic_precise_3}, and \eqref{proof:prop_blow_up_0.5}, which in particular implies that $\norm{(V_\e,W_\e)}_{\mc{H}_\e \times \mc{H}_\e}^2=O(\e^{2|m+\rho|})$ as $\e \to 0^+$.
\end{proof}

We conclude this subsection by proving Proposition \ref{prop_positive_negative}. We start with  a technical lemma.

\begin{lemma}\label{lemma_positive_negative}
Under the same assumptions as in Proposition \ref{prop_positive_negative}, the following holds.
 \begin{enumerate}[\rm (i)]
\item $L\big|_{\widetilde{\mathcal H}}\not \equiv 0$ if $\alpha^j \in -\gamma+\frac{\pi}{2m+1}(1+2 \mathbb Z)$ 
for every $j=1, \dots, k$. 
\item $\eta(\Phi_0,\Psi_0) \not \in \widetilde{\mc{H}}$ if $\alpha^j \in  -\gamma+\frac{2\pi}{2m+1}\mathbb Z$   for every $j=1, \dots, k$, with $\eta$ as in \eqref{def_eta}.
\end{enumerate}
\end{lemma}
\begin{proof}
Let $j_0$ be such that $\alpha^{j_0}=\min\{\alpha^j:j=1,\dots,k\}$.

(i) Let $\ell \in \mathbb Z$ be such that $\alpha^{j_0}= -\gamma+\frac{\pi}{2m+1}(1+2\ell) \in
(-\pi,\pi]$. If $\alpha^{j_0} \in [0,\pi]$
it is easy to see that, letting $f$ be as in \eqref{def_f}, $f(\alpha^{j_0})=0$, and so 
\begin{align*}
&\nabla \Phi_0(r\cos(\alpha^{j_0}),r\sin(\alpha^{j_0})) \cdot \nu^{j_0}  
=(-1)^{\ell+1}\beta \left(m+\tfrac{1}{2}\right)r^{m-\frac{1}{2}},\\
&\nabla \Psi_0(r\cos(\alpha^{j_0}),r\sin(\alpha^{j_0})) \cdot \nu^{j_0}
=0.
\end{align*}
On the other hand,  if $\alpha^{j_0} \in (-\pi,0)$,
\begin{align}
&\nabla \Phi_0(r\cos(\alpha^{j_0}),r\sin(\alpha^{j_0})) \cdot \nu^{j_0}  
=(-1)^{\ell}\beta \left(m+\tfrac{1}{2}\right)\cos(2\pi f(\alpha^{j_0}+2\pi))r^{m-\tfrac{1}{2}},\\
&\nabla \Psi_0(r\cos(\alpha^{j_0}),r\sin(\alpha^{j_0})) \cdot \nu^{j_0}
=(-1)^{\ell}\beta \left(m+\tfrac{1}{2}\right)\sin(2\pi f(\alpha^{j_0}+2\pi))r^{m-\frac{1}{2}}.
\end{align}
Let $\varphi,\psi:\overline{\pi^{j_0}_+}\to\R$ be non-negative,  smooth functions such that   $\gamma^{j_0}_+(\varphi)\not\equiv0$, $\gamma^{j_0}_+(\psi)\not\equiv0$, and
$\varphi\equiv \psi \equiv 0$ in $\pi^{j_0}_+ \setminus  D_{r}\big({\frac{1}{2}{a_{j_0}}}\big)$, where  $r>0$ is chosen sufficiently small to have $r<\frac12|a^{j_0}|$ and $D_r\big(\frac12 a^{j_0}\big)\cap \Gamma^j_1=\emptyset$ for all all $j\neq j_0$.
Let  
\begin{align}
&\widetilde\varphi:=c_1\left((-1)^{\ell}\beta \left(m+\tfrac{1}{2}\right)\int_{S_1^{j_0}}r^{m-\frac{1}{2}}\gamma_+^j (\varphi) \, dS\right)^{-1}  \varphi, \\
&\widetilde \psi:=c_2\left((-1)^{\ell}\beta \left(m+\tfrac{1}{2}\right)\int_{S_1^{j_0}}r^{m-\frac{1}{2}}\gamma_+^j (\psi) \, dS\right)^{-1}  \psi.
\end{align}
It is possible to extend $(\widetilde \varphi,\widetilde\psi)$ to the whole $\R^2$ obtaining a pair in $\widetilde{\mc{H}}$, still denoted as $(\widetilde \varphi,\widetilde\psi)$.

Assume that $\alpha^{j_0} \in (-\pi,0)$.
Then, by \eqref{def_L},
\begin{align}
\label{eq:lnot01}L(\widetilde \varphi,\widetilde\psi)&=c_1  [(b_{j_0}-1)
\cos(2\pi f(\alpha^{j_0}+2\pi))+d_{j_0}\sin(2\pi f(\alpha^{j_0}+2\pi))]\\
\notag&\quad +c_2[(b_{j_0}-1)
\sin(2\pi f(\alpha^{j_0}+2\pi))-d_{j_0}\cos(2\pi f(\alpha^{j_0}+2\pi))].
\end{align}
If $L\big|_{\widetilde{\mathcal H}} \equiv 0$ then, by \eqref{eq:lnot01}  and the arbitrariness of $c_1,c_2$, we could conclude that 
\begin{align}
&(b_{j_0}-1)\cos(2\pi f(\alpha^{j_0}+2\pi))+d_{j_0}\sin(2\pi f(\alpha^{j_0}+2\pi))=0,\\
&(b_{j_0}-1)\sin(2\pi f(\alpha^{j_0}+2\pi))-d_{j_0}\cos(2\pi f(\alpha^{j_0}+2\pi))=0.
\end{align}
Hence $b_{j_0}=1$ and $d_{j_0}=0$, which contradicts \eqref{def_bj_dj} since $\rho^{j_0}\not \in \mb{Z}$.
If  $\alpha^{j_0} \in [0,\pi]$ we can argue similarly.

(ii) By \eqref{def_Psi_rho}--\eqref{def_Phi_rho} we have, for every $j=1,\dots,k$, $\gamma^j_+(\Phi_0)=\gamma^j_-(\Phi_0)$ and 
$\gamma^j_+(\Psi_0)=\gamma^j_-(\Psi_0)$  on $S_1^j$. 
Moreover $\eta\equiv 1$ on $S_1^j$ for every $j=1,\dots,k$.
Hence the condition $\eta(\Phi_0, \Psi_0) \in \widetilde{\mc{H}}$, i.e. 
$R^j(\eta(\Phi_0, \Psi_0))=I^j(\eta(\Phi_0, \Psi_0))=0$ for all $j=1,\dots,k$, 
 would imply that $\Phi_0=\Psi_0=0$ on $S_1^j$  for all $j=1,\dots,k$.

Let $\ell \in \mathbb Z$ be such that $\alpha^{j_0}= -\gamma+\frac{2\pi}{2m+1}\ell \in (-\pi,\pi]$.
By   \eqref{def_Psi_rho} and \eqref{def_Phi_rho}, either 
\begin{equation*}
 \Phi_0(r\cos(\alpha^{j_0}),r\sin(\alpha^{j_0}))
=(-1)^{\ell}\beta r^{m+\frac{1}{2}},\quad 
\Psi_0(r\cos(\alpha^{j_0}),r\sin(\alpha^{j_0})) 
=0,
\end{equation*}
 if $\alpha^{j_0} \in [0,\pi]$, or 
\begin{align}
& \Phi_0(r\cos(\alpha^{j_0}),r\sin(\alpha^{j_0}) )
=(-1)^{\ell +1}\beta\cos(2\pi f(\alpha^{j_0}+2 \pi))r^{m+\frac{1}{2}},\\
&\Psi_0(r\cos( \alpha^{j_0}),r\sin( \alpha^{j_0}) )
=(-1)^{\ell +1}\beta \sin(2\pi f(\alpha^{j_0}+2 \pi))r^{m+\frac{1}{2}},
\end{align}
if $\alpha^{j_0} \in  (-\pi,0)$. In both cases this is a contradiction. 
\end{proof}

We are now in a position to prove Proposition \ref{prop_positive_negative}.
\begin{proof}[Proof of Proposition \ref{prop_positive_negative}]
(i) If $\alpha^j \in  -\gamma+\frac{\pi}{2m+1}(1+2 \mathbb Z)$ for every $j=1,\dots,k$, a direct computation yields that $\Phi_0\equiv\Psi_0\equiv0$ on $S_1^j$ for all $j=1, \dots, k$, in view of \eqref{def_Psi_rho} and \eqref{def_Phi_rho}. It follows that   $L(\Phi_0,\Psi_0)=0$ and  $\eta(\Phi_0,\Psi_0) + \widetilde{\mc{H}} = \widetilde{\mc{H}}$.
Furthermore, $L \not \equiv 0$ in $\widetilde{\mc{H}}$ by Lemma \ref{lemma_positive_negative}-(i). 

Fixing  $(v,w) \in \widetilde{\mc{H}}$ such that $L(v,w) \neq 0$, we have 
\begin{equation*}
J(t(v,w))=\frac{t^2}{2}\int_{\R^2 \setminus \Gamma_1}(|\nabla v|^2+ |\nabla w|^2) \, dx -tL(v,w)<0
\end{equation*}
for some small $t$. Hence  $\mc{E}<0$, by \eqref{min_prob_tilde} and \eqref{def_E}. Then,  from Theorem \ref{theorem_asymptotic_precise}--(ii) it follows that that $\la_{\e,n_0}<\la_{0,n_0}$ for sufficiently small $\e$.

(ii) If $\alpha^j \in  -\gamma+\frac{2\pi}{2m+1}\mathbb Z$ for every $j=1,\dots,k$, we have $\nabla \Phi_0\cdot \nu^j\equiv0$ and $\nabla \Psi_0\cdot \nu^j\equiv0$ on $S_1^j$ for all $j=1, \dots, k$. It follows that $L=0$ and consequently $\mc{E}>0$, by Lemma \ref{lemma_positive_negative}-(ii), \eqref{min_prob_tilde} and \eqref{def_E}.
Hence $\la_{\e,n_0}>\la_{0,n_0}$ for small $\e$, in view of Theorem \ref{theorem_asymptotic_precise}.
\end{proof}

\subsection{Convergence  of eigenfunctions} The energy estimates proved in Theorem  \ref{theorem_asymptotic_with_eigenfunction} 
and the blow-up analysis performed in Subsection \ref{subsec_blow_up} allow us to  obtain the following blow-up theorem for scaled eigenfunctions and a sharp estimate for their rate of convergence in the $\mc{H}_1\times \mc{H}_1$-norm.

\begin{proposition}\label{prop_convergence_eigenfunctions}
Suppose that \eqref{hp_cirulation} and \eqref{hp_la0_simple} hold and let $(v_0,w_0)$ be as in \eqref{def_v0_w0} with $u_0$ as in \eqref{def_u0} (and \eqref{eq:propertyP} if $\rho=\frac12$).  For $\e >0$ small, let
 $u_\e$ be an eigenfunction of problem \eqref{prob_Aharonov-Bohm_multipole}, associated to the eigenvalue $\la_\e:=\lambda_{\e,n_0}$, such that the corresponding pair  $(v_\e,w_\e)$ defined in \eqref{def_ve_we} satisfies   \eqref{hp_ve_we}. 
 Let $m \in \mathbb{Z}$ be as in Proposition \ref{prop_vw_asympotic_not_1/2}  with $(v,w)=(v_0,w_0)$ if $\rho\neq\frac12$, or as in  Proposition \ref{prop_vw_asympotic_1/2} if $\rho=\frac12$. 
 Then, for every $r>0$,
\begin{equation}\label{limit_blow_up_eigenfunciton}
\e^{-|m+\rho|}(v_\e(\e x),w_\e(\e x)) \to (\Phi_0-\widetilde{V},\Psi_0-\widetilde{W}) \quad  
\text{in } H^1(D_r\setminus \Gamma_1)\times H^1(D_r\setminus \Gamma_1),
\end{equation}
as  $\e \to 0^+$,  
with $\Phi_0$ and $\Psi_0$ being as in  \eqref{def_Psi_rho}--\eqref{def_Phi_rho} 
and $(\widetilde{V},\widetilde{W})$ as in \eqref{min_prob_tilde}.
Furthermore 
\begin{equation}\label{limit_blow_up_norm_eigenfunciton}
\lim_{\e \to 0^+}\e^{-|m+\rho|}\norm{(v_\e-v_0,w_\e-w_0)}_{\mc{H}_1\times \mc{H}_1}=\|(\widetilde{V},\widetilde{W})\|_{\widetilde{\mathcal X}}. 
\end{equation}
\end{proposition}

\begin{proof}
Following the notations of Theorem  \ref{theorem_asymptotic_with_eigenfunction}, let $Y_\e:=v_0-V_\e$ and $Z_\e:=w_0-W_\e$, where $V_\e,W_\e$ are as in Proposition \ref{prop_potential}.
We can write the projection operator $\Pi_\e$ defined in \eqref{def_Pie} as $\Pi_\e=\Pi_\e^1+\Pi_\e^2$, where, for every $(\varphi,\psi)\in L^2(\Omega)\times L^2(\Omega)$, 
\begin{equation*}
\Pi_\e^1(\varphi,\psi)=
\bigg(\int_\Omega(\varphi v_\e+\psi w_\e)\,dx\bigg)(v_\e,w_\e),\quad 
\Pi_\e^2(\varphi,\psi)=
\bigg(\int_\Omega(\psi v_\e-\varphi w_\e)\,dx\bigg)(-w_\e,v_\e).
\end{equation*}
We observe that, in view of \eqref{hp_ve_we},
\begin{equation*}
    \Pi_\e^2(Y_\e,Z_\e)=\bigg(\int_\Omega(V_\e w_\e-W_\e v_\e)\,dx\bigg)(-w_\e,v_\e),
\end{equation*}
so that, by Proposition \ref{prop_norm_L2_norm_nabla} and the fact that 
 $\{(v_\e,w_\e)\}_{\e \in (0,1]}$ is bounded in $\mc{H}_1\times \mc{H}_1$, 
 \begin{equation}\label{eq:stimaPi2}
\|\Pi_\e^2(Y_\e,Z_\e)\|_{\mc{H}_\e \times \mc{H}_\e}=
o(\norm{(V_\e,W_\e)}_{\mc{H}_\e \times \mc{H}_\e}) \quad \text{as } \e \to 0^+. 
\end{equation}
Estimates \eqref{eq:stimaPi2} and   \eqref{eq_v0w0_VeWe_Pi} imply that  
\begin{align}\label{proof_prop_convergence_eigenfunctions_1}
\norm{\Pi_\e^1(Y_\e,Z_\e)-(Y_\e,Z_\e)}_{\mc{H}_\e \times \mc{H}_\e}
&\leq 
\norm{\Pi_\e(Y_\e,Z_\e)-(Y_\e,Z_\e)}_{\mc{H}_\e \times \mc{H}_\e}
+\norm{\Pi_\e^2(Y_\e,Z_\e)}_{\mc{H}_\e \times \mc{H}_\e}\\
&\notag =o(\norm{(V_\e,W_\e)}_{\mc{H}_\e \times \mc{H}_\e}) \quad \text{as } \e \to 0^+. 
\end{align}
Furthermore, from \eqref{eq_Pi(v0w0_VeWe)}, \eqref{eq:stimaPi2}, and \eqref{prop:appendix3_2} it follows that 
\begin{equation}\label{eq:stimaPi2L}
    \|\Pi_\e^1(Y_\e,Z_\e)\|_{L^2(\Omega)\times L^2(\Omega)}=1+o(\norm{(V_\e,W_\e)}_{\mc{H}_\e \times \mc{H}_\e}) \quad \text{as } \e \to 0^+. 
\end{equation}
Letting
\begin{equation*}
(\widetilde{Y}_\e(x), \widetilde{Z}_\e(x)):=\e^{-|m+\rho|} (\Pi_\e^1(Y_\e,Z_\e)-(Y_\e,Z_\e))(\e x) \quad \text {for every } x \in \tfrac{1}{\e}\Omega,
\end{equation*}
and extending $\widetilde{Y}_\e$ and  $\widetilde{Z}_\e$ trivially in $\R^2 \setminus \frac{1}{\e}\Omega$, we have  
$(\widetilde{Y}_\e, \widetilde{Z}_\e) \in \widetilde{\mc{H}}$ and 
\begin{align*}
\|(\widetilde{Y}_\e, \widetilde{Z}_\e)\|_{\widetilde{\mc{X}}}^2&=\e^{-2|m+\rho|}\norm{\Pi_\e^1(Y_\e,Z_\e)-(Y_\e,Z_\e)}^2_{\mc{H}_\e \times \mc{H}_\e}\\
&=\e^{-2|m+\rho|}o(\norm{(V_\e,W_\e)}^2_{\mc{H}_\e \times \mc{H}_\e}) =\|(\widetilde{V}_\e,\widetilde{W}_\e)\|^2_{\widetilde{\mc{X}}}\,o(1)=o(1)
\quad\text{as }\e \to 0^+,
\end{align*}
 in view of  a change of variables, \eqref{proof_prop_convergence_eigenfunctions_1}, \eqref{def_tilde_VW_e} and Proposition \ref{prop_blow_up}.
From the  continuity of the operator in \eqref{def_restr_operator} it follows that 
\begin{equation}\label{proof_prop_convergence_eigenfunctions_3}
\widetilde{Y}_\e, \widetilde{Z}_\e \to 0 \quad  \text{as } \e \to 0^+, \text{ strongly in } H^1(D_r\setminus \Gamma_1) \text{ for every } r>0.
\end{equation}
Let 
\begin{equation}\label{def_FeGe}
(F_\e(x),G_\e(x)):=\e^{-|m+\rho|}(\Pi_\e^1(Y_\e,Z_\e))(\e x) \quad \text {for every } x \in \tfrac{1}{\e}\Omega.
\end{equation}
We still denote  with $F_\e$ and $G_\e$ their trivial extensions  in $\R^2 \setminus  \frac{1}{\e}\Omega$.  We have
\begin{equation*}
(F_\e,G_\e)=(\widetilde{V}_{0,\e},\widetilde{W}_{0,\e})-(\widetilde{V}_\e,\widetilde{W}_\e)+(\widetilde{Y}_\e,\widetilde{Z}_\e),
\end{equation*}
with  $(\widetilde{V}_\e,\widetilde{W}_\e)$ and  $(\widetilde{V}_{0,\e},\widetilde{W}_{0,\e})$
being as in \eqref{def_tilde_VW_e} and  \eqref{def_tilde_VW_0}, respectively.
By \eqref{limit_tilde_V0W0_H1}, \eqref{limit_VWe} and \eqref{proof_prop_convergence_eigenfunctions_3} we conclude that 
\begin{equation}\label{proof_prop_convergence_eigenfunctions_4}
(F_\e, G_\e) \to (\Phi_0-\widetilde{V},\Psi_0-\widetilde{W}) 
\quad  \text{as } \e \to 0^+,
\end{equation}
strongly in $H^1(D_r\setminus \Gamma_1)
\times H^1(D_r\setminus \Gamma_1)$ for every $r>0$.

By \eqref{eq_limit_vewe_v0wo}, Proposition \ref{prop_VWe_to_0}, and Proposition \ref{prop_norm_L2_norm_nabla} we have 
\begin{equation*}
    \int_\Omega(Y_\e v_\e+Z_\e w_\e)\,dx=
        \int_\Omega(v_0 v_\e+w_0 w_\e)\,dx-
            \int_\Omega(V_\e v_\e+W_\e w_\e)\,dx=1+o(1)\quad\text{as }\e\to0^+.
\end{equation*}
Hence, taking into account also \eqref{hp_ve_we} and the definition of $\Pi_\e^1$, 
\begin{equation*}
\int_\Omega(Y_\e v_\e+Z_\e w_\e)\,dx=\|\Pi_\e^1(Y_\e,Z_\e)\|_{L^2(\Omega)\times L^2(\Omega)},
\end{equation*}
 so that 
\begin{equation}\label{eq:identify-Vewe}
(v_\e,w_\e)=\frac{\Pi_\e^1(Y_\e,Z_\e)}{\norm{\Pi_\e^1(Y_\e,Z_\e)}_{L^2(\Omega) \times L^2(\Omega)}}
\end{equation}
provided  $\e$ is sufficiently small.
In conclusion, \eqref{limit_blow_up_eigenfunciton} follows from \eqref{def_FeGe}, \eqref{proof_prop_convergence_eigenfunctions_4}, \eqref{eq:identify-Vewe} and \eqref{eq:stimaPi2L}.

Furthermore, by \eqref{eq:identify-Vewe}, \eqref{eq:stimaPi2L}, and \eqref{eq_limit_vewe_v0wo},
\begin{align}\label{proof_prop_convergence_eigenfunctions_6}
\norm{(v_\e,w_\e)-\Pi_\e^1(Y_\e,Z_\e)}_{\mc{H}_1\times{\mc{H}_1}}&=\frac{\big|1-\norm{\Pi_\e^1(Y_\e,Z_\e)}_{L^2(\Omega) \times L^2(\Omega)}\big|}{\norm{\Pi_\e^1(Y_\e,Z_\e)}_{L^2(\Omega) \times L^2(\Omega)}}\norm{\Pi_\e^1(Y_\e,Z_\e)}_{\mc{H}_1\times{\mc{H}_1}}\\
&\notag =\big|1-\norm{\Pi_\e^1(Y_\e,Z_\e)}_{L^2(\Omega) \times L^2(\Omega)}\big|\; {\norm{(v_\e,w_\e)}_{\mc{H}_1\times{\mc{H}_1}}}\\
&\notag 
=o(\norm{(V_\e,W_\e)}_{\mc{H}_\e \times \mc{H}_\e}), \quad \text{ as } \e \to 0^+. 
\end{align}
On the other hand, by \eqref{proof_prop_convergence_eigenfunctions_1},
\begin{align}\label{proof_prop_convergence_eigenfunctions_7}
\norm{\Pi_\e^1(Y_\e,Z_\e)-(v_0,w_0)}_{\mc{H}_1\times{\mc{H}_1}}^2 &=\norm{(V_\e,W_\e)}_{\mc{H}_\e\times{\mc{H}_\e}}^2
+\norm{\Pi_\e^1(Y_\e,Z_\e)-(Y_\e,Z_\e)}_{\mc{H}_\e\times{\mc{H}_\e}}^2\\
&\notag \qquad\qquad -2\big((V_\e,W_\e),\Pi_\e^1(Y_\e,Z_\e)-(Y_\e,Z_\e)\big)_{\mc{H}_\e\times \mc{H}_\e}\\
&\notag =
\norm{(V_\e,W_\e)}_{\mc{H}_\e\times{\mc{H}_\e}}^2+o(\norm{(V_\e,W_\e)}_{\mc{H}_\e\times{\mc{H}_\e}}^2) \end{align}
as $\e \to 0^+$.
Putting together \eqref{proof_prop_convergence_eigenfunctions_6} and \eqref{proof_prop_convergence_eigenfunctions_7} we obtain 
\begin{equation}\label{proof_prop_convergence_eigenfunctions_8}
\norm{(v_\e -v_0,w_\e-w_0)}_{\mc{H}_1\times{\mc{H}_1}}^2=
\norm{(V_\e,W_\e)}_{\mc{H}_\e\times{\mc{H}_\e}}^2+o(\norm{(V_\e,W_\e)}_{\mc{H}_\e\times{\mc{H}_\e}}^2) \quad \text{as } \e \to 0^+. 
\end{equation}
Letting $\widetilde{V}_\e$ and $\widetilde{W}_\e$ be as in \eqref{def_tilde_VW_e}, from
\eqref{proof_prop_convergence_eigenfunctions_8} and \eqref{limit_VWe} it follows that 
\begin{equation*}
\e^{-2 |m+\rho|}\norm{(v_\e -v_0,w_\e-w_0)}_{\mc{H}_1\times{\mc{H}_1}}^2=\|(\widetilde{V}_\e,\widetilde{W}_\e)\|_{\widetilde{\mc{X}}}^2(1+o(1))
=\|(\widetilde{V},\widetilde{W})\|_{\widetilde{\mc{X}}}^2+o(1)\quad \text{ as } \e \to 0^+,  
\end{equation*}
thus proving \eqref{limit_blow_up_norm_eigenfunciton}.
\end{proof}

\begin{proof}[Proof of Theorem \ref{theorem_blow_up_eigenfunctions}]
Let $u_0$ be an eigenfunction  of \eqref{prob_Aharonov-Bohm_0} associated to $\la_{0,n_0}$, with $\norm{u_0}_{L^2(\Omega)}=1$, and let $u_\e$ be an eigenfunction of \eqref{prob_Aharonov-Bohm_multipole} associated to $\la_{n_0,\e}$ satisfying \eqref{hp_ue_remormalised}. Then, letting $v_\e,w_\e$ be  as in \eqref{def_ve_we}, we have that $v_\e,w_\e$ satisfy \eqref{hp_ve_we}.
Letting 
\begin{equation*}
\widetilde u_\e(x)=\e^{-|m+\rho|}u_\e(\e x), \quad \widetilde v_\e(x)=\e^{-|m+\rho|}v_\e(\e x), \quad\text{and}\quad
\widetilde w_\e(x)=\e^{-|m+\rho|}w_\e(\e x),
\end{equation*}
by \eqref{eq:tras-grad}
and the fact that 
$\mc{A}_1^{(\rho_1\dots,\rho_k)}(x)=\e\mc{A}_\e^{(\rho_1\dots,\rho_k)}(\e x)$
and $\Theta_\e(\e x)=\Theta_1(x)$ for every $\e \in (0,1]$, we have 
\begin{align*}
    (i\nabla +\mc{A}_1^{(\rho_1\dots,\rho_k)})\widetilde u_\e(x)&=
    \e^{-|m+\rho|+1}((i\nabla +\mc{A}_\e^{(\rho_1\dots,\rho_k)})u_\e)(\e x)\\
&    =\e^{-|m+\rho|+1}ie^{i\Theta_\e(\e x)}(\nabla v_\e+i\nabla w_\e)(\e x)  
    =ie^{i\Theta_1}(\nabla \widetilde v_\e+i\nabla \widetilde w_\e),
    \end{align*}
and 
\begin{equation*}
    \widetilde u_\e=e^{i\Theta_1}  (\widetilde v_\e+i\widetilde w_\e),
\end{equation*}
so that \eqref{limit_blow_up_eigenfunciton} and again \eqref{eq:tras-grad}
with $\e=1$ yield, for every $r>0$,
\begin{align*}
   (i\nabla +\mc{A}_1^{(\rho_1\dots,\rho_k)})\widetilde u_\e&\to 
   ie^{i\Theta_1}  (\nabla(\Phi_0-\widetilde V)+i\nabla(\Psi_0-\widetilde W))\\
   &\quad=   
   (i\nabla +\mc{A}_1^{(\rho_1\dots,\rho_k)})(e^{i\Theta_1}((\Phi_0-\widetilde V)+i(\Psi_0-\widetilde W)) )\quad\text{as $\e\to0^+$ in }L^2(D_r,\C)
\end{align*}
and
\begin{equation*}
    \widetilde u_\e\to e^{i\Theta_1}\big((\Phi_0-\widetilde V)+i(\Psi_0-\widetilde W)\big)\quad\text{as $\e\to0^+$ in }L^2(D_r,\C).
\end{equation*}
We have thereby proved that $\widetilde u_\e\to 
e^{i\Theta_1}\big((\Phi_0-\widetilde V)+i(\Psi_0-\widetilde W)\big)$ as $\e\to0^+$ in $H^{1,1}(D_r,\C)$, as stated in \eqref{limit_blow_up_eigenfunction_not_gauged}.
 Finally \eqref{limit_convergence_rate_eigenfunction} follows directly from 
\eqref{eq:tras-grad} and \eqref{limit_blow_up_norm_eigenfunciton}.
\end{proof}


\bigskip\noindent {\bf Acknowledgments.}  
The authors are members of GNAMPA-INdAM.
V. Felli and G. Siclari are partially supported by the MUR-PRIN project no. 20227HX33Z 
``Pattern formation in nonlinear phenomena'' granted by the European Union -- Next Generation EU.
B. Noris is supported by the MUR grant Dipartimento di Eccellenza 2023-2027, by the MUR-PRIN project no. 2022R537CS ``NO$^3$'' granted by the European Union -- Next Generation EU, and by the
INdAM-GNAMPA Project 2023 ``Interplay between parabolic and elliptic PDEs''.


\end{document}